\documentclass[a4paper,10pt]{amsart}
\usepackage{amsfonts,epsfig}
\usepackage{color}
\usepackage[colorlinks=true, citecolor=red]{hyperref}  
\usepackage{amssymb,latexsym}
\usepackage[mathscr]{eucal}
\usepackage{graphicx}
\usepackage{epstopdf}
\usepackage[all]{xy}
\usepackage{titletoc}
\usepackage{pdfsync}

\newtheorem{thm}{Theorem}[section]

\newtheorem{cor}[thm]{Corollary}

\newtheorem{lem}[thm]{Lemma}
\newtheorem{propn}[thm]{Proposition}

\theoremstyle{definition}
\newtheorem{defn}[thm]{Definition}
\newtheorem{rem}[thm]{Remark}

\newtheorem{constr}[thm]{Construction}
\newtheorem{notn}[thm]{Notation}
\newtheorem{asm}[thm]{Assumption}

\newcommand{\bfi}{{\boldsymbol{i}}}
\newcommand{\bfj}{{\boldsymbol{j}}}
\newcommand{\bfk}{{\boldsymbol{k}}}

\newcommand{\be}{\mbox{$\mathbb{E}$}}
\newcommand{\bp}{\mbox{$\mathbb{P}$}}
\newcommand{\ce}{\mbox{$\mathcal E$}}
\newcommand{\cf}{\mbox{$\mathcal F$}}

\newcommand{\bff}{{\boldsymbol{F}}}

\newcommand{\hmu}{\hat {\mu}}
\newcommand{\heta}{\hat {\eta}}

\newcommand{\br}{\mbox{$\mathbb R$}}

 \DeclareMathOperator{\diam}{diam}

\begin{document}
\bibliographystyle{plain}

\title{Spectral Asymptotics for $V$-variable Sierpinski Gaskets}

\author{U.~Freiberg}
\address{Department of Mathematics \\ University of Siegen \\ DE-57068 Siegen \\ Germany
}
\email{freiberg@mathematik.uni-siegen.de}

\author{B.M.~Hambly}
\address{Mathematical Institute\\ University of Oxford\\
24-29 St Giles\\ Oxford OX1 3LB \\ UK}
\email{hambly@maths.ox.ac.uk}

\author{John~E. Hutchinson}
\address{Mathematical Sciences Institute\\
Australian National University\\
Canberra, ACT, 0200\\
Australia}
\email{John.Hutchinson@anu.edu.au} 

\date{\today}

\begin{abstract}
The family of $V$-variable fractals provides a means of interpolating between two 
families of random fractals previously considered in the literature; scale irregular fractals ($V=1$)
and random recursive fractals ($V=\infty$). We consider 
a class of $V$-variable affine nested fractals based on the Sierpinski gasket with a general
class of measures. We calculate the spectral exponent for a general measure and 
find the spectral dimension for these fractals. We show that the spectral properties and on-diagonal heat kernel estimates 
for $V$-variable fractals are closer to those of scale irregular fractals, in that it is the fluctuations in scale that determine
their behaviour but that there are also effects of the spatial variability. 
\end{abstract}

\maketitle

%

\section{Introduction}

The field of analysis on fractals has been primarily concerned with the construction
and analysis of Laplace operators on self-similar sets. This has yielded a well
developed theory for post critically finite (or p.c.f.) self-similar sets, a class of finitely ramified fractals 
\cite{Kig}.
One motivation for the development of such a theory, aside from its intrinsic mathematical
interest, has come from the study of transport in disordered media. However, in this setting
the fractals arise naturally in models from statistical physics at or near a phase transition 
and are therefore 
random objects without exact self-similarity but with some statistical self-similarity. 

In order to develop the mathematical tools to tackle analysis on such random fractals one approach 
has been to work with simple models based on self-similar sets but exhibiting randomness. 
The first case to be treated was that of scale irregular fractals \cite{Ham1},
\cite{barham}, \cite{HKKZ} and \cite{DreStr}, which have spatial homogeneity but randomness in their scaling. A more 
natural setting is provided by random recursive fractals, initially constructed by \cite{MauWil}, \cite{Fal}, \cite{Gra}, 
where the fractal can be decomposed into a random number of independent scaled copies.  
The study of some analytic properties of 
classes of random recursive Sierpinski gasket can be found in \cite{Ham2}, \cite{Ham3}
and \cite{HamKum}. 

Recently there has been work tackling random sets arising from critical phenomena directly, with a particular focus on
the percolation model. Substantial progress has been made in the study of random walk on critical percolation clusters 
in the high dimensional case, see \cite{BJKS} and \cite{KozNac}. 
A bridge between these two approaches can be found in work on the continuum random tree \cite{Cro}, \cite{CroHam} 
or on critical percolation clusters on hierarchical lattices \cite{HamKum09}, both of which have random self-similar 
decompositions and hence have descriptions as random recursive fractals.

In this paper we consider $V$-variable fractals recently introduced in \cite{BHS0,BHS1}. 
This class of random fractals
is defined via a family of iterated function systems and a  positive integer parameter $V$. It interpolates between 
the class of homogeneous (scale irregular) random fractals, corresponding to $V=1$, and the class of random 
recursive fractals,
corresponding to $V=\infty$. As for the random recursive fractals we can regard these
$V$-variable fractals as determined by a probability measure on the set of labelled
trees. In this case the measure is not a product measure, but is defined in a natural 
(if not completely obvious) manner which allows for at most $ V $ distinct subtrees rooted at each level.

Our aim in this paper is to investigate the analytic properties of the class of $V$-variable Sierpinski gaskets 
and to compare their behaviour to the scale irregular and random recursive cases. We show their Hausdorff dimension in the 
resistance metric is the zero of a certain pressure function and their
spectral dimension, the exponent for the growth of the eigenvalue counting function, is the zero of another pressure 
function. The connection between these two dimensions is established.   We develop and extend standard methodology to examine more 
detailed properties of the eigenvalue counting function and the on-diagonal heat kernel. These results show that 
the $V$-variable fractals are closer to the scale irregular case, in that
their fine properties are generally determined by fluctuations in scale rather than fluctuations which occur spatially across 
the fractal. 

\subsection*{Model problems}

We  consider two model problems. Recall from \cite{Hut} the description of a self-similar set as an iterated function 
system (or IFS) at each node of a tree generated by the address space.

\subsubsection*{Homogeneous and Random Recursive Fractals} 

For the first model problem  we consider the two IFSs generating the Sierpinski gasket fractal SG(2) and the fractal SG(3)  
defined in \cite{Ham1}.  
The scale factors for SG(2) are mass $m_2=3$, length $\ell_2=2$ and time $s_2=5$.
For SG(3) we have mass $m_3=6$, length $\ell_3=3$ and time $s_3=90/7$. The conductance scale factors can 
be computed directly, or from the Einstein relation $ \rho = s/m $, 
giving $\rho_2=5/3, 
\rho_3=15/7$. Let $(M, S, L)$ be a triple of random variables taking each of the values  $(m_i,s_i,\ell_i)$ where 
$i=2,3$ with probabilities $p,1-p$ respectively. 

Then, for the $V=1$ (homogeneous) case, we construct a random fractal using a sequence 
taking its values in $\{2,3\}$ and applying the corresponding IFS to all sets at a given level of construction. A realization
of the first few stages can be seen in Figure~\ref{1-var}. Then 
a simple scaling analysis shows that the Hausdorff dimension is given by $d_f = \mathbb{E} \log M/\mathbb{E}  
\log L$ where $\mathbb{E} $ denotes the expectation with respect to the probability measure generating the 
sequence.  For the spectral dimension with respect to the natural ``flat measure'' one can extend the idea from 
\cite{Fuk} and \cite{kiglap} in the case of a single IFS fractal and apply a scaling argument to the Dirichlet form 
together with a Dirichlet-Neumann bracketing argument, see \cite{Ham4}.   This gives the spectral dimension 
$d_s = 2\mathbb{E} \log M/\mathbb{E} \log S$.
For the $V=\infty$ (random recursive) case, each IFS is chosen independently for each node at each level. In this case 
we have $d_s=2d_f^r/(d_f^r+1)$ where $d_f^r$ is the Hausdorff dimension
in the resistance metric, that is $d_f^r$ is such that $\mathbb{E} (M (S/M)^{-d_f^r})=
\mathbb{E} M^{1+d_f^r}S^{-d_f^r} = 1$.
The argument again uses scaling properties of the Dirichlet form and a Dirichlet-Neumann bracketing argument, 
see \cite{Ham4, Ham3}.  An alternative approach to computing the spectral dimension for random $ V=1,\infty $ 
fractals is via heat kernel estimates, see \cite{barham} and \cite{Ham1,Ham2,Ham4,Ham3}.

The second model problem is drawn from the class of affine nested fractals considered in \cite{FHK}. 
This model interpolates between the slit triangle (which is not itself an affine nested fractal) and SG(3). 
Consider  7 triangles in the configuration shown in Figure~1 and take
$\ell$ as the side length of the three triangles at the corners of the original triangle. 
The side lengths of the other triangles are given as $1-2\ell$ for
the three triangles on the centre of each side and $3\ell-1$ for the downward pointing central triangle, where 
$1/3 < \ell < 1/2$. As $\ell \to 1/2$ we have the slit triangle and at $\ell=1/3$ we have SG(3).
\begin{figure}[htbp] \label{modsg}
\centerline{\epsfig{file=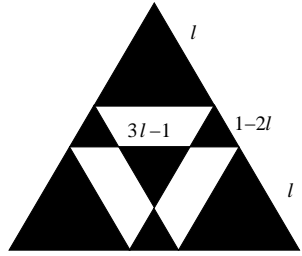, height=1.3in}}
\caption{A member of the family of Sierpinski gaskets interpolating SG(3) and the slit triangle, where $1/3 < \ell < 1/2$.}
\end{figure}
We   construct a homogeneous random or random recursive fractal by taking a suitable distribution for $\ell$ on 
$[1/3, 1/2)$ and either using a sequence, applying the same IFS at each node in the construction tree for the $V=1$ case, 
or independently for each node in the $V=\infty$ case. 

We note that even scale irregular ($V=1$) affine nested gaskets of this type have not been
treated before and as a consequence of our results we will be able to calculate the Hausdorff and
spectral dimension for the random homogeneous version ($V=1$). By the triangle-star transform, 
if we assume that the resistance of each piece is proportional to its length, then the resistance 
scale factor is \mbox{$(2\ell+1)/(\ell+2)$} in that
if we take resistances  on the three different types of triangle to be 
\mbox{$  (\ell+2)/(2\ell+1) (\ell,  1-2\ell ,  3\ell-1 ) $} 
then this is electrically equivalent to the triangle with unit resistance on each edge.

In Section~2 we recall from \cite{BHS2} the Hausdorff dimension result for   $V$-variable fractals,
and we derive the spectral dimension from our calculations in Sections~4 and~5. 

\subsubsection*{$V$-Variable Fractals}

To understand the $V$-variable versions of our model problems, first consider the $ V=1 $ (spatially homogeneous, 
scale irregular) case of a $ V $-variable labelled tree in a manner parallel to the approach taken in the general setting. 
See Figure~\ref{1-var}.  For $ V=1 $ all subtrees rooted at each fixed level are the same, as are the corresponding
subfractals at each fixed level, hence the terminology ``homogeneous''.  The subtrees at one level are typically 
not the same as the subtrees at another level, hence the terminology ``scale irregular''. 

\begin{figure}[htbp] 
  \centering
  \includegraphics[width=3in]{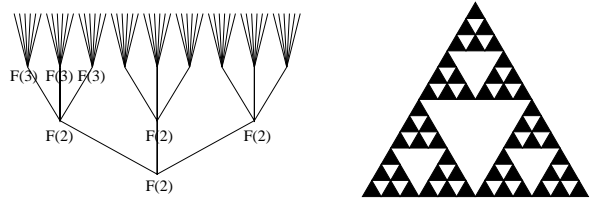} 
  \caption{The level 3 approximation  to a 1-variable   tree, and the  prefractal approximation  to the associated 
  1-variable, or scale irregular,  fractal.  Here the family of IFSs is $ \boldsymbol{F} =\{  F(2), F(3) \} $ with 
  members generating the sets $ SG(2) $ and $ SG(3) $ respectively.}  
  \label{1-var}
\end{figure}

For a general $ V$-variable tree and for the corresponding $ V $-variable fractal, there are at most $ V $ distinct subtrees 
up to isomorphism rooted at each fixed level,  and correspondingly at most $V$ distinct subfractals up to rescaling at each 
fixed level of refinement.  See Figure~\ref{IFStree} for a level 2 approximation to a $ V $-variable tree with $ V \geq 2 $. In
Section~\ref{secgva} we discuss this in some detail and see that there is a natural probability distribution  on the 
class of $ V $-variable fractals for each fixed $V$.

The construction of $ V $-variable trees and hence $ V $-variable fractals will require an assignment of a \emph{type}
chosen from $ \{1,\dots, V\} $, as well as an IFS, to each node of the tree.  Nodes with the same type  and at the same 
level will have identical subtrees rooted at those nodes. The subfractals corresponding to those nodes will be identical 
up to scaling.  See Figure~\ref{V-var}. We choose the IFSs according to a probability measure and will write $P_V$ for 
the probability measure on the space of trees or $V$-variable fractals and $E_V$ for expectation with respect to $P_V$.

Let $n(1)$ be a random variable denoting the first level after level 0 at which all nodes are assigned the same type. 
Since the number of types is finite and we will assume a uniform upper bound on the branching number, 
$E_V n(1)<\infty$. Note that $n(1)=1$ if $ V=1 $, and clearly $E_V n(1)$ increases with $ V $.

We write $\boldsymbol{i} =i_1\dots i_k$ for a node in the tree and denote its height or length by $ | \bfi | = k $. 
The root node is denoted by $ \emptyset $ and 
$ | \emptyset | = 0 $. The Hausdorff dimension $ d_f $ of the $V$-variable gasket formed from SG(2) and SG(3) is 
given $P_V$ almost surely by the zero of a pressure function in that ($P_V$ almost surely) it is the unique 
$ d_f $ such that 
$   E_V \log \sum_{|\bfi | = n(1) } (\ell_{i_1}\cdot \ldots \cdot \ell_{i_{n(1)}})^{ d_f} = 0$,
where $\ell_{i_k}$ is the length scale value $1/2$ or $1/3$ according to which of SG(2) or SG(3) is chosen.  
See Theorem~\ref{bdhd},  also Theorem~\ref{thm:lil}.

\subsection*{Results}
\emph{For further detail see the Overview  at the beginning of the following Sections~2--5.}

Our main results first establish an expression for the spectral exponent over a general class of measures and 
 determine the spectral dimension for these fractals. We   then provide finer results of two types. We   consider 
the eigenvalue counting function and the on-diagonal heat kernel and obtain upper and lower bounds on these 
quantities which hold for all $V$-variable trees. By placing a probability measure on the trees we   obtain 
almost sure results capturing more explicitly their fluctuations.  In the model problems  the expectation is either over a discrete measure on $\{2,3\}$ or over a 
suitable distribution on $[1/3, 1/2]$.

We   show in Theorem~\ref{thm:Nspecdim} that the spectral exponent can also be expressed as the zero of a 
pressure function. In Theorems~\ref{spfm} and~\ref{spmax} we see  
that the spectral dimension, the maximum value of the spectral exponent over all measures $\mu$ defined using a product 
of weights, satisfies the equation $d_s/2=d_f^r/(d_f^r+1)$ where $d_f^r$ is the Hausdorff dimension in the resistance metric.  
This dimension in turn is the zero of another pressure function, see Theorem~\ref{rdzp}.


We  establish upper and lower estimates for the eigenvalue counting
function and on-diagonal heat kernel for a general class of measures. 
We show that the observed fluctuations
arise from two different effects. The
first is due to global scaling fluctuations as observed for scale irregular nested Sierpinski gaskets \cite{barham}. 
The second effect, which arises in the $V$-variable setting for $V>1$
or $V=1$ when the contraction factors are not all the same, 
gives additional, though much smaller, fluctuations due to the spatial variability of
these fractals. 

We first establish from Lemma~\ref{firsteest} the non-probabilistic result that if ${\mathcal N}(\lambda)$ denotes the
number of eigenvalues less than $\lambda$ (for the Dirichlet or
Neumann Laplacian), then there is a time scale
factor $T_k$, a mass scale factor $M_k$ and a correction factor
$A_k$, such that there are
constants $c_1,c_2$ with
\[ c_1 M_k \leq \mathcal{N}(A_k T_k) \mbox{ and }  \mathcal{N}(T_k) \leq c_2 M_k, \;\; \forall k. \]  
As in the scale irregular gaskets of \cite{barham}, this result is 
true for all realizations. By construction the scale factors $M_k,T_k$ grow exponentially in $k$ but we 
will be able to show that $P_V$ almost surely we have 
$A_k\leq c k^{\beta}$, and even in certain cases
$A_k\leq C (\log{k})^{\beta}$, for some constant $\beta$.
The spectral exponent for any measure $\mu$ defined by a
set of weights associated with a given IFS is
\[ \frac{d_s(\mu)}{2}:=\lim_{x\to\infty} \frac{\log{\mathcal{N}(\lambda)}}{\log \lambda}, \]
and we give a formula for this quantity as the zero of a suitable pressure function.
In the case where the weights are `flat' in
the resistance metric we can show that there is a function $\phi(\lambda) = \exp(\sqrt{\log{\lambda}\log\log\log{\lambda}})$ such that 
$P_V$-almost surely
\begin{equation}
c_1 \lambda^{d_s/2} \phi(\lambda)^{-c_2} \leq \mathcal{N}(\lambda) \leq c_3 \lambda^{d_s/2} \phi(\lambda)^{c_4}, \label{eq:Nasymp}
\end{equation}
for large $\lambda$, where $d_s=2d_f^r/(d_f^r+1)$ and $d_f^r$ is the Hausdorff dimension in the resistance metric.

 To compare our results with previous work we note that in the $V=1$ case for nested Sierpinski gaskets 
it is shown in \cite{barham} that the Weyl limit for the normalized counting function
does not exist in general and we have for all realizations that
\[ c_1 M_k \leq \mathcal{N}(T_k) \leq c_2 M_k. \]
This leads to the same size scale fluctuations as for the $V$-variable case given in (\ref{eq:Nasymp}).
For the random recursive case of \cite{Ham3}, the averaging leads to a Weyl limit in that
\[
 \lim_{\lambda\to\infty} \frac{\mathcal{N}(\lambda)}{\lambda^{d_s/2}}  \text{ exists } \ P_{\infty} \;a.s.,
  \]
where $d_s=2d_f^r/(d_f^r+1)$ and $d_f^r$ is the Hausdorff dimension in the resistance metric.

We will also be able to remark on the on-diagonal heat kernel. We note
that the measures we work with in this setting do not have the volume
doubling property and hence it is harder work to produce good heat
kernel estimates. In the setting considered here we can extend the arguments of \cite{barham} and \cite{BarKum} 
to get fluctuation results for the heat kernel.
In Theorems~\ref{thm:hkub} and ~\ref{thm:hklb} we show that the on-diagonal heat kernel estimate is determined by
the local environment. In the case
where the measure is the `flat' measure in the resistance metric
we can describe the small time global fluctuations in that for almost every point $x$ in the fractal,
\[ c_1 t^{-d_s/2} \phi(1/t)^{-c_2} \leq p_{t}(x,x) \leq c_3 t^{-d_s/2} \phi(1/t)^{c_4} , \;\;0<t\leq c_5, \;\;P_V\;a.s., \]
for suitable deterministic constants $ c_1,c_2,c_3,c_4 $, and for all $t\leq c_5$, a random constant depending on the
point $x$.
These are of the same order as the $V=1$ case obtained in \cite{barham} and much larger than those in 
the random recursive case, \cite{HamKum}.

In the case of general measures we will see that $P_V$-almost surely, $\mu$-almost every $x$ in the fractal does not 
have the same spectral exponent as the counting function (except when we choose the flat measure) and 
thus there will be a multifractal structure to the local heat kernel estimates in the same way as observed in 
\cite{BarKum}, \cite{HamKigKum}.

\bigskip
We restrict ourselves to affine nested fractals
based on the Sierpinski gasket in $ \mathbb{R}^d $ where $  d \geq 2 $. The problem of the existence of a
limiting Dirichlet form is not solved more generally, even for the case of homogeneous
random fractals. If this problem
were solved, then the techniques used here would enable more general results to be
obtained concerning $V$-variable p.c.f. fractals.

\bigskip The structure of the paper is as follows. We give the construction of $V$-variable
affine nested Sierpinski gaskets in Section~2. We show that by using the structure
of $V$-variability there is a natural decomposition of the fractals at `necks'; a level at which all subtrees are the same. 
This idea was first used by Scealy in~\cite{Sce}. In Section~3
we construct the Dirichlet form, compute the resistance dimension, and determine other properties which will facilitate analysis
on these sets. In Section~4 we treat the spectral asymptotics.  The heat kernel is
dealt with in Section~5.

\subsection*{Acknowledgement}
We particularly wish to thank an anonymous referee for an unusually careful and detailed set of comments.  Addressing these has led to a number of improvements in the results of the paper.


\section{Geometry of $V$-Variable Fractals}\label{secgva}

\subsection{Overview} Random $ V $-variable fractals are generated from a possibly uncountable 
family $ \boldsymbol{F}$ of IFSs.  Each individual IFS $ F \in \boldsymbol{F} $  generates an affine nested fractal.  
We also  impose various probability distributions on~$ \boldsymbol{F}   $. 

For motivation, consider the two model problems in the Introduction.  Namely, $ \boldsymbol{F}  =\{ F_2, F_3 \} $ 
is the pair of IFSs generating $ SG(2) $ and $ SG(3) $, or $ \boldsymbol{F}$ is the family of affine nested fractals 
$ F_\ell $ generating the prefractal in Figure~1 for $ \ell \in [1/3,1/2] $.  

A $ V $-variable tree corresponding to $ \boldsymbol{F} $ is a tree with an IFS from $ \boldsymbol{F} $ associated 
to each node, a \emph{type} from the set $ \{1,\dots, V \} $ associated to each node, and such that if two nodes 
at the same level have the same type, then the corresponding (labelled) subtrees rooted at those two nodes are 
isomorphic.  This last requirement is achieved by using a sequence of \emph{environments}, one at each level, to 
construct a $ V $-variable tree.  Each $ V $-variable tree generates a $ V $-variable fractal set in the natural way.  
The case $ V=1 $ corresponds to homogeneous fractals and $ V \to \infty $ corresponds to random recursive fractals.

If all nodes at some level have the same type, the level is called a \emph{neck}.  Neck levels are given by a sequence of 
independent geometric random variables.   In Lemma~\ref{lem:geomrvs} we record some useful results for such 
random variables.  In Section~\ref{hbd} we recall the Hausdorff  dimension result from~\cite{BHS2} but in the 
framework of necks as used in this paper, and then give a refinement by using the law of the iterated logarithm. 
This provides motivation for some of the spectral results.

\subsection{Families of Affine Nested Fractals} \label{afn}

Let   ${\boldsymbol{F}}  $ be a possibly uncountable class of IFSs $ F $, each generating a compact fractal $ K^F$, 
and each defined via a set of similitudes $ \{\psi_i^F\}_{i\in S^F } $ acting on $ \mathbb{R}^d $, with contraction 
factors $\{\ell_i^F\}_{i\in S^F}$ and $ S^F=\{1,\dots, N^F \} $.  If it 
is clear from the context we write $ K $, $ \psi_i $, $ N $ and $S$ for $ K^F $, $\psi_i^F $, $ N^F $ and $ S^F $ respectively, 
and similarly for other notation.  


We will have
\begin{equation} \label{Nbd} 
\begin{gathered} 
3 \leq N_{\inf}    :=   \inf \{N^F :      F \in \boldsymbol{F} \}   \leq 
 \sup \{N^F :     F \in \boldsymbol{F} \} =:N_{\sup}   < \infty,\\
0 < \ell_{\inf} :=   \inf \{\ell^F :      F \in \boldsymbol{F} \} .
\end{gathered} 
\end{equation} 
The first follows from our later constructions, see \eqref{V0ass}.
The second and third are  for technical reasons arising in the study of the heat kernel and spectral asymptotics. 
See also the comments after Definition~\ref{dfneck}, from which it is clear that weaker conditions will suffice to construct $V$-variable fractals and establish their Hausdorff dimension.

Let $  \Psi^F $ denote the \emph{set of fixed points} of the $ \{\psi_i^F\}_{i\in S^F }  $.  Then $ x\in  \Psi^F $  is an 
\emph{essential fixed point} if there exists $ y \in   \Psi^F $ and $ i \neq j $ such that $ \psi_i^F(x) = \psi_j^F(y) $.  
Let \emph{$ V_0 $ denote the set of essential fixed points}.  

We always assume that \emph{  $ V_0 $ does not depend on $ F$}.

\bigskip Assume the \emph{uniform} open set condition  for the
$\{\psi^F_i\} $. That is, there is a non-empty, bounded open set $O$, \emph{independent of $ F $},  
such that
$\{\psi^F_i (O)\}_{i\in S^F}$ are disjoint and $\bigcup_{i\in S^F}\psi_i^F (O)\subset O$.

Let $ \psi^F_{i_1\cdots i_n} =  \psi^F_{i_1}\circ\cdots\circ\psi^F_{i_n} $ and let
\begin{equation} \label{} 
\displaystyle {V^F_n =  \bigcup_{i_1,\cdots, i_n \in S^F}  \psi^F_{i_1\cdots i_n}(V_0)},  \quad 
V^F_{*}=\bigcup_{n\ge 0} V^F_n.
\end{equation} 
 Then $K^F=cl (V^F_*)$, the closure of $V^F_*$. 
 
 For
$i_1,\dots,i_n\in S^F$, we call $ \psi^F_{i_1\cdots i_n}(V_0)$ an \emph{$n$-cell} and
$ \psi^F_{i_1\cdots i_n}(K^F)$ an \emph{$n$-complex}.

For $x,y \in {\br}^d  (x \ne y)$, set
$H_{xy}=\{z \in {\br}^d : |z-x|=|z-y|\}$ and let
$U_{xy}:{\br}^d \to {\br}^d$ be the reflection transformation
with respect to $H_{xy}$. 

\medskip When computing the spectral dimensions we further  assume  each $K^F$ is  
an  \emph{affine nested fractal}.  That is, the open set condition holds, $|V_0|\ge 2$, and:
\begin{enumerate}
\item $K^F$ is connected;
\item  (Nesting) If $(i_1,\cdots,i_n)$ and $(j_1,\cdots,j_n)$
are distinct $ n $-tuples  of elements from $S^F$, then
\[
\psi^F_{i_1\cdots i_n}(K^F)\cap
\psi^F_{j_1\cdots j_n}(K^F)=\psi^F_{i_1\cdots i_n}(V_0)\cap
\psi^F_{j_1\cdots j_n}(V_0);
\]
\item  (Symmetry) For $x,y \in V_0~(x \ne y)$,
$U_{xy}$ maps $n$-cells to $n$-cells, and it maps any $n$-cell which
contains elements in both sides of $H_{xy}$ to itself for each $n\ge 0$.
\end{enumerate} 

We also make the technical assumption that $|\psi^F_i(V_0)\cap \psi^F_j(V_0)|\leq 1$ for all $1\leq i\neq j\leq N^F$.

%

\subsection{Trees and Recursive Fractals}
 
Fix a family ${\boldsymbol{F}}  $ of IFSs as before.  For our initial purposes it is sufficient only that the IFSs consist of  
uniformly contractive maps on~$ \mathbb{R}^d $.
 
Each realisation  of a random fractal is built by means of an \emph{IFS construction tree}, or \emph{tree} for short, 
defined as follows.

\begin{defn}  (See Figure~\ref{IFStree}) An \emph{(IFS construction) tree $ T $} corresponding to $ \boldsymbol{F} $ is  a    
tree with the following properties: 
\begin{enumerate}
\item there is a single, level 0, root node $ \emptyset $;
\item the \emph{branching number} $ N^{\boldsymbol{i}} $ at each node $ \boldsymbol{i} $ has 
$ 2 \leq N^{\boldsymbol{i}} <\infty$ ($N^{\boldsymbol{i}} \geq 3$ later);
\item the  edges with  initial node $ \boldsymbol{i} $ are numbered (``left to right'') by $ 1,\dots, N ^{\boldsymbol{i}} $;
where  $ \boldsymbol{i} = i_1\dots i_{k }$ in the usual manner and $| \boldsymbol{i} | := k \geq 1$ is the \emph{level} 
of~$ \boldsymbol{i}$, or $  \boldsymbol{i} =\emptyset $ in which case   $ |\boldsymbol{i} | := 0 $ is the level;
\item there is an \emph{IFS} $ F^{ \boldsymbol{i} } \in \boldsymbol{F} $ associated with each node $ \boldsymbol{i} $,  
$ N^ {\boldsymbol{i}} = |  F^{ \boldsymbol{i} } |  $ (the cardinality of $  F^{ \boldsymbol{i} }$), and the $k $th edge 
with initial node $ \boldsymbol{i} $ is associated with the $ k  $th function in the IFS $ F^{ \boldsymbol{i}}$.
\end{enumerate} 
The   unique compact set $ K = K(T)$ associated with $ T $ in the usual manner is called a \emph{recursive  fractal}.  
\end{defn} 

\begin{figure}[htbp] 
  \centering
  \includegraphics[width=3in]{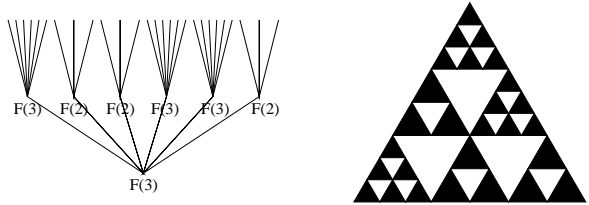} 
  \caption{Level 2 approximations to an IFS tree and to the associated fractal. Here $ \boldsymbol{F} = \{F(2), F(3)\}$ 
contains the IFSs generating $ SG(2) $ and $ SG(3) $ respectively.  Edges of the tree with a given initial node are 
enumerated from left to right; they correspond to subcells enumerated anticlockwise from the bottom left corner 
of the cell corresponding to the given node.}
  \label{IFStree}
\end{figure}

\begin{figure}[htbp] 
  \centering
  \includegraphics[width=4.5in]{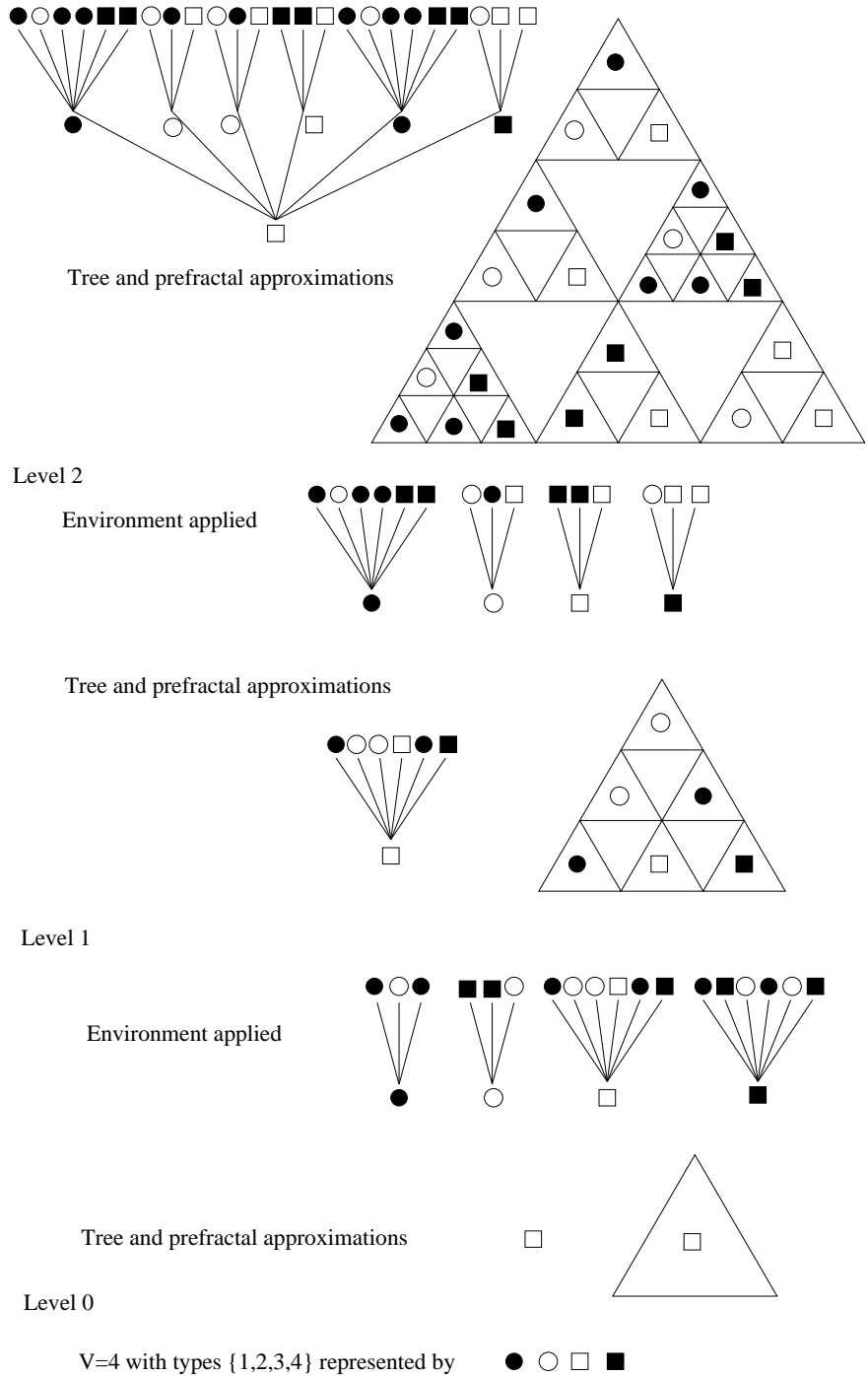} 
 \caption{Approximations   to a 4-variable   tree  and the     prefractal approximations to the corresponding 4-variable 
 fractal.  The IFSs are $ F(2) $ and $ F(3) $ generating $ SG(2) $ and $ SG(3) $. The environment at each level is applied 
 to the approximation  at the previous level.   The IFS labels  are not shown  since   in this case they are   determined by 
 the branching number.}
  \label{V-var}
\end{figure}

\begin{notn} \label{notn1}

The \emph{boundary} $ \partial T $ of a tree $ T $ is the set of infinite paths through $ T $ beginning at $ \emptyset $.

For $ \boldsymbol{i} \in T $   the   \emph{cylinder set} $ [ \boldsymbol{i} ] \subset \partial T $ is the set of all infinite paths $ \boldsymbol{w} \in \partial T $ such that $ \boldsymbol{i} $ is an \emph{initial segment} of $ \boldsymbol{w} $, written $ \boldsymbol{i} \prec \boldsymbol{w} $.

\smallskip The \emph{concatenation} of two sequences $ \boldsymbol{i}$ and $ \boldsymbol{j} $, where $\boldsymbol{i}$ is of finite length, is denoted by the juxtaposition $ \boldsymbol{i} \boldsymbol{j} $.

The \emph{truncation} of $ \boldsymbol{i} $ to the first $ n $ places is defined by  $ \boldsymbol{i} | n = i_1\dots i_n $.

A \emph{cut}  for the tree $ T $ is a finite  set $\Lambda \subset  T $ with the property that for every $ \boldsymbol{w}  \in \partial T $ there is exactly one $ \boldsymbol{i} \in \Lambda $ such that $ \boldsymbol{i} \prec \boldsymbol{w} $.  Equivalently, $ \{ [ \boldsymbol{i}] : \boldsymbol{i} \in \Lambda \}$ is a partition of $ \partial T $.

\smallskip  For a tree $ T $  and a node $ \boldsymbol{i} \in T $,  there will usually be associated quantities such as an IFS $ F^{ \boldsymbol{i} } $, a type $ \tau^ {\boldsymbol{i}  } \in \{1,\dots, V \} $ (see Definition~\ref{dfVt}) or a branching number $ N^{ \boldsymbol{i} } $.  In this case $ \boldsymbol{i} $ is shown as a  \emph{superscript}.
 
In particular,  the \emph{transfer operator}  $ \sigma^{\boldsymbol{i}}  $ acts on $ T $ to produce the   tree
 $ \sigma^{\boldsymbol{i} } T $, where, 
 writing $T^{\bfj}$ for the address of node $\bfj$,
 \begin{equation} \label{} 
   \left( \sigma^{\boldsymbol{i} } T \right)^{\boldsymbol{j} } := T^{\boldsymbol{i} \boldsymbol{j} } .
   \end{equation} 
That is, $  \sigma^{\boldsymbol{i} } T $ is the subtree of $ T $ which has its base (or root) node at $ \boldsymbol{i} $.

We frequently need to multiply a sequence of quantities, or compose a sequence of functions, along a finite branch corresponding to a node $ \boldsymbol{i} = i_1\dots i_n $ of $ T $.  In this case, $ \boldsymbol{i} $ is shown as a \emph{subscript}.  For example,
if $ \boldsymbol{i} = i_1\dots i_n $ then, with some abuse of notation for the second term,
\begin{equation} \label{lprod} 
\ell_{\boldsymbol{i}} := \ell_{i_1}\cdot \ldots \cdot \ell_{i_n} := \ell^{F^\emptyset}_{i_1} \cdot 
                 \ell^{F^{i_1}}_{i_2} \cdot     \ell^{F^{i_1i_2}}_{i_3}  \cdot     \ldots \cdot \ell^{F^{i_1\dots i_{n-1}}}_{i_n}
\end{equation} 
is the product of scaling factors corresponding to the edges along the branch $ i_1\dots i_n $, and analogously for other scaling factors. Similarly,
\begin{equation} \label{} 
\psi_{\boldsymbol{i}} :=  \psi_{i_1}\circ \dots \circ \psi_{i_n} := \psi^{F^\emptyset}_{i_1} \circ 
                \psi^{F^{i_1}}_{i_2} \circ     \psi^{F^{i_1i_2}}_{i_3} \circ      \dots \circ \psi^{F^{i_1\dots i_{n-1}}}_{i_n}
\end{equation} 
is the composition of functions along the same branch.
\end{notn}

 \begin{notn}[Cells and Complexes]\label{notcc}
The recursive   fractal   $K = K(T)$ generated by $ T $  satisfies 
\begin{equation} \label{} 
 K(T)  = \bigcup_{i=1}^{N^{\emptyset}} \psi_i^{F^{\emptyset}} \big(K(\sigma^iT) \big)
  = \bigcup_{| \boldsymbol{i} | = n} \psi_{ \boldsymbol{i} } (K(\sigma^{ \boldsymbol{i} } T)),
\end{equation} 
where the second equality comes from iterating the first.   

For $ | \boldsymbol{i} | = n $ the $n$-complex and $n$-cell with address $ \boldsymbol{i} $ are respectively
\begin{equation} \label{celldf}
K_{\boldsymbol{i} } = \psi_{ \boldsymbol{i} } (K(\sigma^ { \boldsymbol{i} } T)), \quad 
\Delta_{ \boldsymbol{i} } := \psi_{ \boldsymbol{i} } (V_0),
\end{equation} 
recalling that $ V_0 $ is the set of essential fixed points of $ F \in \boldsymbol{F} $ and is the same for all ~$ F $. 


\begin{asm} In Section \ref{secanv} and subsequently we assume 
 \begin{equation} \label{V0ass}
\begin{aligned} 
&\text{$ V_0 $ is the set of vertices of an equilateral tetrahedron in $ \mathbb{R}^d $ for some $ d \geq 2 $}, \\
&\text{$ E_0 $ is the set of edges},  
 G_0= (V_0, E_0)\text{ is the complete graph on } V_0 .
\end{aligned} 
\end{equation} 
\end{asm}

We will need various sequences of graph approximations $\{G_n\}_{n=0}^{\infty}$ to the fractal $ K(T) $.  In particular we  
use the notation $G_n = (V_n, E_n)$, where 
\begin{equation} \label{GVE} 
V_n := \bigcup_{| \boldsymbol{i} | = n } \psi_{ \boldsymbol{i}} (V_0) = \bigcup_{| \boldsymbol{i} | = n }  
\Delta_{ \boldsymbol{i} }, \quad E_n := \bigcup_{| \boldsymbol{i} | = n } \psi_{ \boldsymbol{i}} (E_0).
\end{equation} 
We can recover the   fractal itself  as $K(T) = cl(\bigcup_n V_n)$, where $ cl $ denotes closure.

We will write $x\sim_n y$ for $x,y \in V_n$ if $x,y$ are connected by
an edge in $E_n$.


 \end{notn}
 
\subsection{$ V $-Variable Trees and $ V $-Variable Fractals} 

Fix a natural number $ V $.
For motivation   see Figure~\ref{V-var}.

\medskip 
The following definition of a $ V $-variable tree and $ V $-variable fractal  is equivalent to that in~\cite{BHS1} 
and~\cite{BHS2}, but avoids working with $ V $-tuples of trees and fractals.

\begin{defn} \label{dfVt}
A  \emph{$ V $-variable  tree} corresponding to $ \boldsymbol{F} $ is an IFS construction tree $ T $ corresponding to 
$ \boldsymbol{F} $, with a \emph{type} $ \tau  ^{ \boldsymbol{i} } \in \{1,\dots, V \} $ associated to each node 
$ \boldsymbol{i} $.  Moreover, if two nodes $ \boldsymbol{i} $ and $ \boldsymbol{j} $ at the \underline{same  level} 
$ | \boldsymbol{i} | = | \boldsymbol{j} | $  have the same type $ \tau ^ { \boldsymbol{i} } = 
\tau ^ { \boldsymbol{j} } $, then:
\begin{enumerate}
\item  $ \boldsymbol{i} $ and $ \boldsymbol{j} $   have the same   associated IFS  $ F^ { \boldsymbol{i} } = 
F^ { \boldsymbol{j} } $ and hence the same branching number $ N^ { \boldsymbol{i} } = N^ { \boldsymbol{j} } $; 
 
 \item comparable successor nodes $ \boldsymbol{i} p $ and $ \boldsymbol{j} p $,  where $ 1 \leq p \leq  
 N^ { \boldsymbol{i} } = N^ { \boldsymbol{j} } $,  have the same type  $ \tau ^{ \boldsymbol{i}  p } =   
 \tau ^{ \boldsymbol{j}  p } $.
   \end{enumerate} 

The recursive fractal $ K = K(T) $ associated to a  $ V $-variable   tree $ T $ as above is called a 
\emph{$ V $-variable fractal} corresponding to $ \boldsymbol{F} $.    
  
The  \emph{class  of   $ V $-variable trees and class of $ V $-variable fractals} corresponding to $ \boldsymbol{F} $ 
are denoted by $\Omega _  V = \Omega^{ \boldsymbol{F} }_V    $ and $ \mathcal{K} _  V  =  
\mathcal{K} ^{ \boldsymbol{F} }_V  $ respectively.
\end{defn}     

\begin{rem} \label{rm1v}
A $ V $-variable   tree  has at most $  V $ distinct IFSs associated to the nodes at each fixed level.
If two nodes at the \emph{same} level  of  a $ V $-variable   tree have the same type then the subtrees rooted at 
these two nodes are identical, i.e.\
\begin{equation} \label{} 
| \boldsymbol{i} | = | \boldsymbol{j}  |
 \ \&\ \tau ^{\boldsymbol{i}}= \tau ^{\boldsymbol{j} }
 \  \Longrightarrow\  \sigma ^{ \boldsymbol{i} } T = \sigma ^ { \boldsymbol{j} } T .
\end{equation} 
In particular, for each level,  there are at most $ V $ distinct subtrees rooted at that level.

A 1-variable tree is essentially the same as an IFS tree which generates a scale irregular or homogeneous fractal as 
in~\cite{Ham1}, \cite{Ham4} and \cite{barham}. 
\end{rem}

The following is used in the construction and analysis of $ V $-variable fractals.

\begin{defn} \label{dfenv}  An \emph{environment} $ E $ assigns to each type $v \in \{1,\dots, V \} $  both 
an IFS  $  F_v \in \boldsymbol{F} $ and a sequence of types $( \tau_{v,i})_{i=1}^ {|F_v|}$, where $  |F_v| $ is 
the number of functions in $ F_v $.  We write 
\begin{equation} \label{} 
E = \big(E(1),\dots, E(V)\big),\quad   E(v)  = \big(F^E_v,  \tau^E_{v,1}, \dots,  \tau^E_{v, |F_v| }\big). 
\end{equation} 
\end{defn} 

For a pictorial example see Figure~\ref{V-var}. For the following consider the case $ n=2 $ in Figure~\ref{V-var}.

\begin{constr} 
A $ V $-variable   tree  is constructed  from a sequence of environments $ (E^ k )_{k\geq 1} $ in the natural way  as follows:
\begin{itemize} 

\item[\emph{Stage 0}:]   Begin with the root node $ \emptyset $ and an initial type $ \tau^\emptyset $ assigned to 
this node.

\item[\emph{Stage 1}:]  Use $ E^1 $ and the type $ \tau^\emptyset $ in the natural way to  assign an IFS to the level 0 
node, construct the level 1 nodes and assign a type to each of them. 

More precisely, use $ E^1(v) $ where 
\[ v:= \tau^\emptyset , \quad E^1 \big(v\big) =  \Big(F^{E^1}_ {v},  \tau^{E^1}_{v, 1} , \dots,     
\tau^{E^1}_{v, | F^{E^1}_ {v}|}\Big), \]
to assign the IFS $F^{\emptyset} :=F^{E^1}_ {v}  $  to the node $ \emptyset $ and in particular determine the 
branching number $N^{\emptyset} :=   \big| F^{E^1}_ {v}\big| $ at $ \emptyset $, and to assign the type 
$\tau^{j}:= \tau^{E^1}_{v, j} $  to each level $ 1 $ node $ j $.  

\item[\vdots ]  

\item[\emph{Stage n}:]   (By the completion of stage  $  n-1$ for $ n \geq 2 $, an IFS  $F ^ { \boldsymbol{i} } $ will have been assigned to each node   $  \boldsymbol{i} $ of level $| \boldsymbol{i} |  \leq n-2 $,   all nodes $ \boldsymbol{i} $ of level  $ | \boldsymbol{j} | \leq n-1 $ will  have been constructed and a type 
$ \tau ^ { \boldsymbol{j} } $ will have been assigned to each.)
 
Use $ E^n $ in the natural way to  assign an IFS to each level $ n-1 $  node according to its type, to construct the level $ n $  nodes and to assign a type to them. 

More precisely, use $ E^n(v) $  for $ 1\leq v \leq V $ where 
\[
  E^n \big(v\big) =  \Big(F^{E^n}_ {v},  \tau^{E^n}_{v, 1} , \dots,     \tau^{E^n}_{v, | F^{E^n}_ {v} |}\Big),
\]
to assign the IFS $F^{ \boldsymbol{i} } := F^{E^n}_ {v}  $  to each level $ n-1 $  node $ \boldsymbol{i} $ of type $ v $ and in particular to determine the branching  number $N^{ \boldsymbol{i} }: = \big| F^{E^n}_ {v}\big| $ at the node $ \boldsymbol{i} $, and to assign the  type $ \tau^{ \boldsymbol{i} j} := \tau^{E^n}_{v, j} $  to the  level $ n $  node $ \boldsymbol{i} j $.  
\end{itemize} 

It follows by an easy induction that the properties in  Definition~\ref{dfVt} hold at all nodes.
\hfill $\square$  \end{constr} 

We now note the following facts about the connectivity properties of $V$-variable fractals.

\begin{lem} \label{lem:cnprop}
Let $K$ be a $V$-variable fractal. Then\\
(1) $K$ is connected \\ 
(2) $K$ is nested:
For all $\bfi,\bfj \in T$, if $[\bfi] \cap [\bfj] = \emptyset$, then 
$K_{\bfi} \cap K_{\bfj} = \psi_{\bfi}(V_0)\cap \psi_{\bfj}(V_0)$.
\end{lem}

\begin{proof}
(1) The connectedness is clear as all the affine nested fractals in the family are connected.

(2) In our setting this is straightforward to see as 
if $[\bfi] \cap [\bfj] = \emptyset$, there exists a $\bfk$ of maximal length with $\bfk\prec \bfi$ and 
$\bfk\prec\bfj$, such that   $K_{\bfi} \subset K_{\bfk}$ and $ K_{\bfj} \subset K_{\bfk}$. If we write $\bfi = \bfk i_1\dots$ and $\bfj = 
\bfk j_1\dots$, then $i_1\neq j_1$ and by the nesting axiom for $F^{\bfk}$ we have $K_{\bfk i_1} \cap K_{\bfk j_1} 
= \psi_{\bfk i_1}(V_0)\cap V_{\bfk j_1}(V_0)$. If the intersection is empty we are done.  Otherwise, by our technical
assumption on affine nested fractals that $|\psi^{F^{\bfk}}_{i_1}(V_0) \cap \psi^{F^{\bfk}}_{j_1}(V_0)| \leq 1$, there is a single intersection point which is the image of a fixed point in $V_0$.  If $K_{\bfi} \cap K_{\bfj}\neq \emptyset$, 
this is the intersection point of $K_{\bfi} \cap K_{\bfj}$ and therefore of $\psi_{\bfi}(V_0)\cap V_{\bfj}(V_0)$ as required.
If $K_{\bfi} \cap K_{\bfj}= \emptyset$ we are done.
\end{proof}

\subsection{Random $ V $-Variable Trees and Random $ V $-Variable Fractals}\label{rvtf}

\begin{defn} \label{dfpv} Fix a probability distribution $ P $ on $ \boldsymbol{F} $.  This induces a 
probability distribution $ P_V $ on the set of environments as follows. Choose the IFSs $ F^{{E} } _v $ for 
$ v \in \{1,\dots,V\} $ in an i.i.d. manner according to $ P $.   Choose types $ \tau ^{{E}} _{v,j} $ for 
$ 1 \leq j \leq | F^{{E} } _v | $   in an i.i.d. manner according to the uniform distribution on $ \{1,\dots, V \} $ 
and otherwise independently of the  $ F^{{E} } _w$.
\end{defn}

\begin{defn}  \label{dfpv2}
The \emph{probability distribution on the set $ \Omega_V $} of $ V $-variable trees is obtained by choosing 
$ \tau ^\emptyset \in \{1,\dots, V\} $ according to the uniform distribution and independently choosing the
environments at each stage in an i.i.d. manner according to $ P_V $. This probability distribution on $ V $-variable 
trees induces a probability distribution on the set $ \mathcal{K} _V $ of $ V $-variable fractals.  Both the probability
distribution on trees and that on fractals are denoted by $ P_V $. We will write $E_V$ for expectation with respect to $P_V$. 

\emph{Random $ V$-variable   trees} and \emph{random $ V $-variable fractals} are random labelled trees and 
random compact subsets of $ \mathbb{R}^d  $ respectively, having the distribution $ P_V$.
Later, when we add additional scale factors for resistance and weights associated with each $F\in \boldsymbol{F}$, 
we will assume they are measurable with respect to $F\in\boldsymbol{F}$. 
\end{defn}

Although the distribution $ P_V $ on environments is a product measure, this is far from the case for the corresponding
distribution $ P_V $ on $ \Omega_V $ and $ \mathcal{K} _V $.  There is a high degree of dependency between the 
types (and hence the IFSs) assigned to different nodes at the same level. 

\begin{rem} 
The classes $ \mathcal{K}_V $ interpolate  between the class of homogeneous fractals in the case $ V=1 $ and the class of 
recursive fractals as $ V \to \infty $.  The probability spaces  $ (\mathcal{K}_V,P_V) $ interpolate  between the natural 
probability distribution on  homogeneous fractals in the case $ V=1 $ and the natural probability distribution on the class 
of recursive fractals as $ V \to \infty $.   See \cite{BHS0} and \cite{BHS1}.
\end{rem} 

\begin{notn} 
It will often be convenient to identify the \emph{sample space} for random quantities such as trees, fractals, 
functions  associated to a branch of a tree, etc., with the set $ \Omega _V $ of $ V $-variable trees. 
We use $ \omega $ to denote a generic element of  $ \Omega_V $ and combine this with other notations in 
the natural manner.  Thus we may write  $ T^ \omega $,  $ K^ \omega $, $ \psi_{ \boldsymbol{i} }^\omega $ etc.

In particular,  $ \sigma^{ \boldsymbol{i} } \omega $ is   the transfer operator   defined in Notation~\ref{notn1} 
for a tree $ T $. See for example the first equality in \eqref{dfen}.  
However, we usually suppress $ \omega $  as in the second equation in \eqref{dfen}.
Also see \eqref{maxfe} and the explanation which follows it.
\end{notn}

\subsection{Necks}

The notion of a \emph{neck} is critical for the analysis that follows.  

\begin{defn} \label{dfneck} The environment $ E$ in Definition~\ref{dfenv}  is a \emph{neck}  if all 
$ \tau _{v,i} ^{E} $ are equal.

A \emph{neck} for a $ V $-variable tree $ \omega$ is a natural number $n$ such that the environment $E$ applied 
at stage $n$ in the construction of $ \omega$ is a neck environment.  In this case we say a \emph{neck occurs 
at level~$ n $}.   If $  \boldsymbol{i}  $ is a node in $ \omega$ and $ |  \boldsymbol{i}  | = n $, then 
$  \boldsymbol{i}  $ is called a \emph{level $ n $ neck node}.
\end{defn} 

\begin{figure}[htbp] 
   \centering
   \includegraphics[width=4in]{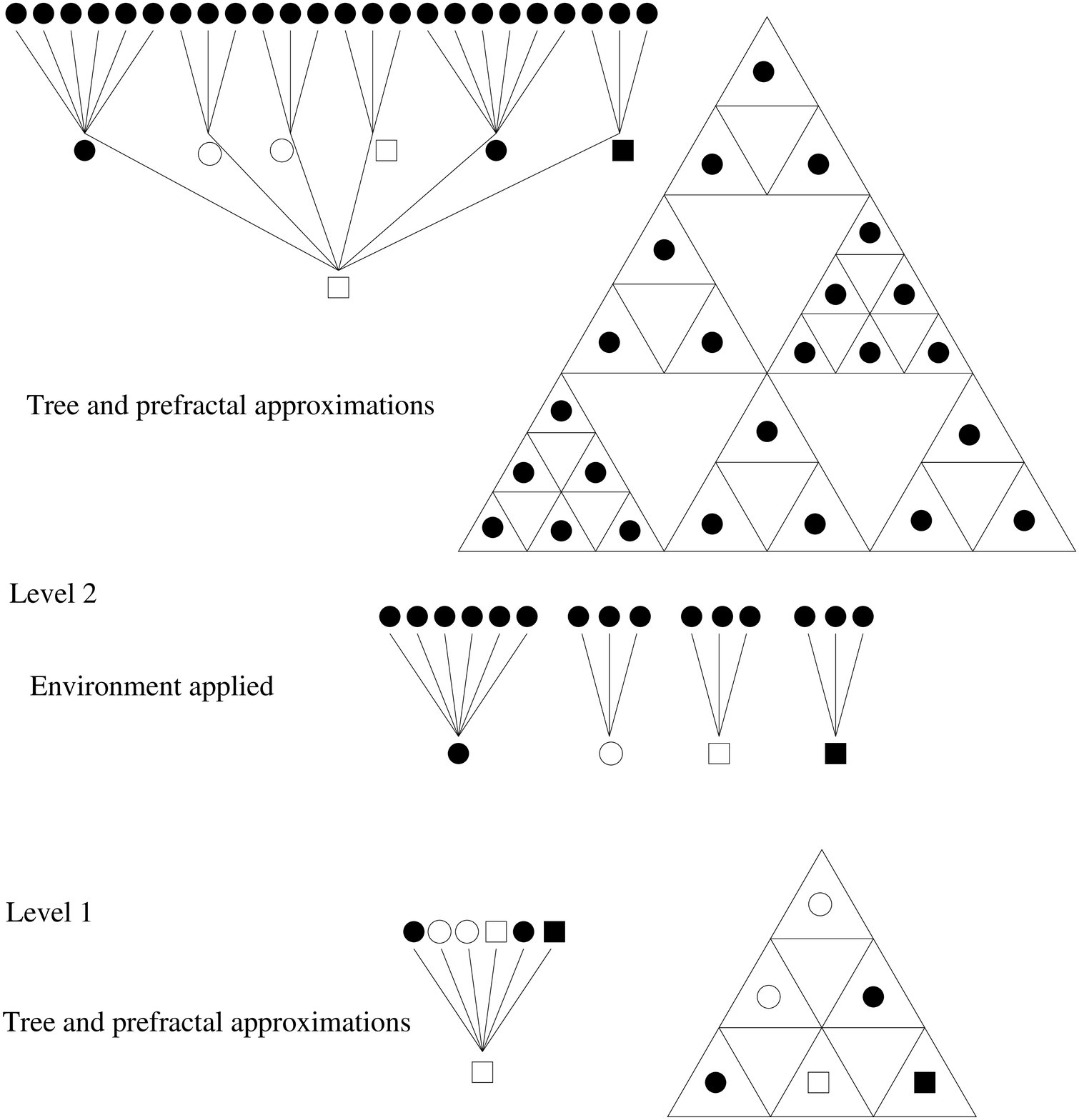} 
   \caption{Compare with Figure~\ref{V-var}, except that now a neck occurs at level 2.  All subtrees rooted at this level
    will be the same, although they have not yet been constructed.  All 2-complexes will be the same up to scaling by 
    factors determined by the construction up to this level.}
   \label{figneck}
\end{figure}

If a neck occurs at level $n$ then the type assigned to every node at that level is the same. See Figure~\ref{figneck}. 
It follows from Remark~\ref{rm1v}  that all subtrees rooted at level $ n $ will be the same. Note that the subtrees 
themselves are only constructed at later stages, and even the common value of the IFS at a level $ n $ neck node is 
not determined until stage $ n+1 $.  

There is however no restriction on the IFSs occurring in a neck environment ${E} $. For a level $ n $ neck 
these IFSs are applied at level $ n-1 $.  

Because there is an upper bound on the number of functions $ N^F $ in any IFS $ F \in \boldsymbol{F} $, there is only a 
finite number of type choices to be made in selecting an environment. It follows that necks occur infinitely often 
almost surely with respect to the probability $ P_V $ defined in Definition~\ref{dfpv2}. The sequence of neck levels 
in the construction of a $ V $-variable tree or fractal is denoted by
\begin{equation} \label{seqneck} 
0= n(0) < n(1) < \dots < n(k) < \cdots .
\end{equation} 

The sequence of times between necks is a sequence of independent geometric random variables, and in particular 
the expected first neck satisfies
\begin{equation} \label{}
E_V n(1) < \infty.
\end{equation} 

Many of our future estimates rely on various a.s.\ properties of necks.   However, some estimates just require that 
there be an infinite sequence of necks. For this reason we make the definition:
\begin{equation} \label{dfwvd} 
\text{\emph{$ \Omega'_V \subset \Omega_V $  is  the set of $ V $-variable trees with an infinite sequence of necks.}}
\end{equation} 

 \medskip We next give an elementary result on the asymptotic behaviour of   a sequence of geometric random variables 
 $ (Y_k)_{k\geq 1} $. It follows that $ Y_k $ grows at most logarithmically in $ k $, and powers of $ Y_k $  grow at 
 most geometrically, with similar results for the maximum and the mean of $\{ Y_1, \dots, Y_k  \}$. 

The following is standard  but included for completeness. Note that the $Y_k$ need not actually be geometric random variables.

\medskip 






\begin{lem}\label{lem:geomrvs}
Suppose  $\{Y_k\}_{k=1}^{\infty}$ is a sequence of not necessarily independent    random
variables  with $P(Y_k>x) \leq A p^x$, where $0<p<1 $ and for all $ x>0 $.  Suppose $ n \geq 1 $ is a natural number. 
Then   a.s.  
\begin{gather}
  \limsup_{k\to\infty} \frac{  Y_{nk}}{\log{k}} \leq \frac{1}{\log 1/p} , \quad  
  \limsup_{k\to\infty} \frac{\max_{1\leq i\leq  nk} Y_i}{\log{k}} \leq \frac{1}{\log 1/p} ,  \label{geomrvs1}  \\
  \limsup_{k \to \infty} \frac{\sum_{i=1}^{nk} Y_i}{ k \log k } \leq \frac{2n}{\log 1/p } .   \label{geomrvs2} 
  \end{gather}
  \end{lem}

\begin{proof} The case $ n>1  $  is a direct consequence of the case $ n =1$, which we establish.   

\medskip Suppose $ \epsilon > 0 $.  Since $ P(Y_k > x ) \leq Ap^{x} $ for $ x>0 $,
\[
\sum_{k\geq 1} P\left(Y_k > \frac{(1+ \epsilon) \log k}{\log 1/p } \right)
 \leq A \sum_{k \geq 1} p ^{  (1+ \epsilon ) \log k /( \log 1/p) } 
 = A \sum_{k \geq 1 } k^{ - ( 1 + \epsilon ) } < \infty .
 \]
Hence by the first Borel-Cantelli lemma, 
 \[
   \limsup_{k\to\infty} \frac{  Y_k}{\log{k}} \leq \frac{1+ \epsilon }{\log 1/p}    \text{\quad a.s.}
   \]
   Since $ \epsilon > 0 $ is arbitrary, the first inequality in \eqref{geomrvs1} follows.
 
\smallskip The second inequality in \eqref{geomrvs1}  is now an elementary consequence. 
Suppose $ \delta  > 0 $.  Using   the first inequality in  \eqref{geomrvs1} to get the second inequality below, $ P $ a.s.\ there exists 
$ k_0 = k_0(\omega, \delta  ) $ such that  $ k \geq k_0 $ implies
 \[
\max_{k_0\leq i\leq k} \frac{Y_i}{\log{k}} \leq \max_{k_0\leq i\leq k} \frac{Y_i}{\log i} \leq \frac{1 }{\log 1/p} + \delta  .
\]
Hence 
\[ 
 \limsup_{k \to \infty} \max_{k_0\leq i\leq k}  \frac{Y_i}{\log{k}} \leq \frac{1}{\log 1/p} + \delta   \text{\quad a.s.}
\]
Replacing $ k_0 $  by $ 1 $ and letting $ \delta \to 0 $  in the above  implies the second inequality in~\eqref{geomrvs1}.

\smallskip  For \eqref{geomrvs2} fix $ \gamma > 0 $.  Then
\begin{align*}
\sum_{k\geq 1} P  \Bigg( &\sum_{i=1}^k Y_i >   \frac{ \gamma k \log k } { \log 1/p } \Bigg)  
 \leq  \sum_{k\geq 1} \sum_{i=1}^k P \left( Y_i > \frac{ \gamma  \log k } { \log 1/p } \right)  \\ 
&\leq \sum_{k \geq 1} k A p^{ \gamma \log k / \log 1/p } 
=A \sum_{k\geq 1} k^{1 - \gamma } < \infty, \text{ if } \gamma > 2 .
\end{align*} 
By the first Borel-Cantelli lemma,  if $ \gamma  > 2 $,
\[
  \limsup_{k \to \infty} \frac{\sum_{i=1}^k Y_i}{ k \log k } \leq \frac{\gamma }{\log 1/p }
  \quad \text{a.s} .   
\]
This gives \eqref{geomrvs2}.
 \end{proof} 
 
We also include a decomposition of sums of products of scale factors. 

It may help to note that the factors on the right side of \eqref{eq:spdecomp} in the next Lemma are calculated by first  choosing and fixing, 
for each $ j=1 \dots k $, an arbitrary  node of $ T $  at level $ n(j-1) $. For fixed $ j $  all subtrees of $ T $ rooted at this level are identical 
by the definition of a neck.  The factor in  \eqref{eq:spdecomp}  is  the sum, of products of   $ s_i^p $ type weights, along all paths in such 
a \emph{subtree} starting from its root node  and ending at a   first neck level node.  There is a one-one correspondence between the set 
of  such paths in the \emph{subtree} and the set of paths   in the \emph{original} tree starting from the \emph{chosen} node at  level 
$ n(j-1) $ and ending at a level $ n(j) $ node.

\begin{lem}\label{lem:sumprod}
Let $s_i=s_i^F \in \mathbb{R}$ for  $ i = 1,\dots, N^F$ be scaling factors associated with each family $F$, 
where
\begin{equation} \label{sest} 
\begin{aligned}
0< s_{\inf}  &:=  \inf\{s^F_i  :      i  \in 1,\dots,N^F,\,  F\in \bff\} ,  \\
s_{\sup}  &:=  \sup\{s^F_i  :  i\in 1,\dots,N^F,\,  F\in \boldsymbol{F} \}   < \infty.
\end{aligned} 
\end{equation} 
Then, writing $s_{\bfi} = s_{i_1}\cdot \ldots \cdot s_{i_n}$ for $\bfi = i_1\dots i_n \in T$, and with $ s^{(j)}_{ \boldsymbol{i} } $ defined in the natural way in the body of the proof,
we have
\begin{equation} 
 \sum_{\boldsymbol{i} \in T, | \boldsymbol{i} | = n(k) } s_{ \boldsymbol{i} }     = \prod_{j=1}^k \Bigg( \sum_{| \boldsymbol{i} | = n(j) - n(j-1) } s^{(j-1)}_{ \boldsymbol{i} }  \Bigg)  \label{eq:spdecomp}.
\end{equation}
Moreover,
\begin{equation} 
\lim_{k\to \infty} \frac{1}{k} \log \sum_{ | \boldsymbol{i} | = n(k) } s_{ \boldsymbol{i} } = E_V  \log \sum_{ | \boldsymbol{i} | = n(1) } s_{ \boldsymbol{i} }  \quad P_V \text{ a.s.} \label{eq:slln}
\end{equation}
\end{lem}

\begin{proof}
Let $ T^{(k)} $ denote the unique subtree of $ T $ rooted at the neck level $ n(k) $, so that in particular $ T^{(0)} = T $.  

Then, as explained subsequently 
(and following the notation of \eqref{lprod} but with the $ F $ there suppressed),
{\allowdisplaybreaks 
\begin{align} 
 &\sum_{\boldsymbol{i} \in T, | \boldsymbol{i} | = n(k) } s_{ \boldsymbol{i} }     =
\sum_{\boldsymbol{i} \in T, | \boldsymbol{i} | = n(k) }
  s_{i_1}^\emptyset \cdot s_{i_2}^{i_1} \cdot   s_{i_3}^{i_1i_2} \cdot \ldots  
          \cdot s_{ i_{n(k)} }^{ \boldsymbol{i} | ( n(k) - 1) }  \notag \\
 & \qquad  = 
  \sum_{\boldsymbol{i} \in T, | \boldsymbol{i} | = n(k)  }
\bigg\{ \left(    s_{i_1}^\emptyset \cdot s_{i_2}^{i_1}\cdot s_{i_3}^{i_1i_2} \cdot
                            \ldots \cdot s_{ i_{n(1)} }^{ \boldsymbol{i} | ( n(1) - 1) }    \right)  \notag\\
&\qquad  \qquad \qquad \quad  \cdot    \left(    s_{i_{n(1)+1}}^{ \boldsymbol{i} | n(1)}  
               \cdot  s_{ i_{n(1)+2} }^{ \boldsymbol{i} | n(1)+1 } 
                     \cdot \ldots \cdot s_{ i_{n(2)} }^{ \boldsymbol{i} | ( n(2) - 1) } \right) \cdot \ldots \notag\\
& \qquad  \qquad \qquad  \quad  \cdot   \left(    s_{i_{n(k-1)+1}}^{ \boldsymbol{i} | n(k-1)}  
                      \cdot  s_{ i_{n(k-1)+2} }^{ \boldsymbol{i} | n(k-1)+1 } 
                     \cdot \ldots \cdot s_{ i_{n(k)} }^{ \boldsymbol{i} | ( n(k) - 1) } \right)  \bigg\}  
          \notag    \\
&\qquad =  
  \sum_{\boldsymbol{i} \in T^{(0)}, | \boldsymbol{i} | = n(1) }
         s_{i_1}^{(0),\emptyset} \cdot s_{i_2}^{(0),i_1}\cdot s_{i_3}^{(0),i_1i_2} \cdot
                            \ldots \cdot s_{ i_{n(1)} }^{(0), \boldsymbol{i} | ( n(1) - 1) }   \notag \\
 &\quad   \times                               
         \sum_{\boldsymbol{i} \in T^{(1)}, | \boldsymbol{i} | = n(2)-n(1) }
         s_{i_1}^{(1),\emptyset} \cdot s_{i_2}^{(1),i_1}\cdot s_{i_3}^{(1),i_1i_2} \cdot
                            \ldots \cdot s_{ i_{n(2)-n(1)} }^{(1), \boldsymbol{i} | ( n(2) -n(1) - 1) }   
                                         \cdot \ldots \notag\\
& \quad  \times                             
     \sum_{\boldsymbol{i} \in T^{(k-1)}, | \boldsymbol{i} | = n(k)-n(k-1) }
         s_{i_1}^{(k-1),\emptyset} \cdot s_{i_2}^{(k-1),i_1}\cdot s_{i_3}^{(k-1),i_1i_2} \cdot
                           \ldots \cdot s_{ i_{n(k)-n(k-1)} }^{(k-1), \boldsymbol{i} | ( n(k) -n(k-1) - 1) }  \notag \\                          
&=   \Bigg( \sum_{| \boldsymbol{i} | = n(1) } s^{(0)}_{ \boldsymbol{i} }  \Bigg)                             
                    \cdot  \Bigg( \sum_{| \boldsymbol{i} | = n(2) - n(1) } s^{(1)}_{ \boldsymbol{i} }  \Bigg)  
    \cdot \ldots  \cdot  \Bigg( \sum_{| \boldsymbol{i} | = n(k) - n(k-1) } s^{(k-1)}_{ \boldsymbol{i} }  \Bigg).  \label{et}
\end{align} 
}

The first and last equality are immediate from the definitions.  The second equality is just a bracketing of terms.  

For the third equality note that each $ n(j) $ is a neck. A term such as $  s_{i_{n(1)+1}}^{ \boldsymbol{i} | n(1)} $, 
which corresponds to the edge in $ T $ from $   \boldsymbol{i} | n(1)  = i_1 \dots i_{n(1)} $ to $ i_1 \dots i_{n(1)}
 i_{n(1)+1}$, is independent of $\boldsymbol{i} | n(1)$ and  can also be regarded as corresponding to 
the level one edge from $ \emptyset $ to $ i_{n(1)+1} $ of the unique tree $ T^{(1)} $ rooted at every  level $ n(1) $ node.
Thus we   rewrite $  s_{i_{n(1)+1}}^{ \boldsymbol{i} | n(1)} $  as $ s_{i_1}^{(1),\emptyset} $,
with an abuse of notation in that $ \boldsymbol{i} $ and $i_{n(1)+1} $  in the first term refer to words from 
$ T = T^{(0)} $ whereas $ i_1 $  in the second term is the first element of a word from $ T^{(1)} $.
Similarly,  $  s_{i_{n(1)+2}}^{ \boldsymbol{i} | n(1)+1} $ is also independent of   $\boldsymbol{i} | n(1)$ and 
can also be regarded as corresponding to a level two edge from $ T^{(1)} $, etc.
Now use simple algebra 
to put the summations inside the parentheses.

The final equality is   a rewriting of the previous line and provides the definition of~$ s^{(j)}_{ \boldsymbol{i} } $.

\smallskip
For the $P_V$ almost sure convergence in \eqref{eq:slln}   let  
\[
X_k = \log  \sum_{| \boldsymbol{i} | = n(k) - n(k-1) } s^{(k-1)}_{ \boldsymbol{i} } , \ k \geq 1.
\]
By construction the $ X_k $  are i.i.d. and in particular
$ 
X_1 =  \log  \sum_{| \boldsymbol{i} | = n(1) } s_{ \boldsymbol{i} } 
$.
By the bounds on $s_i$ we have 
\begin{align*} 
E_V  |X_1| &\leq  
\sum_{n\geq 1} P\big(n(1) = n\big)\,  \max \Big\{
  \big| \log \big(N_{\sup}^n s_{\sup}^n \big)\big|,\   \big| \log \big(N_{\inf}^n s_{\inf}^n \big)\big|
                                           \Big\}           \\
&= \max \Big\{
  \big| \log \big(N_{\sup} s_{\sup} \big)\big|,\,   \big| \log \big(N_{\inf} s_{\inf} \big)\big|
                                           \Big\}  E_V   n(1) <\infty .
\end{align*} 

Hence, using \eqref{eq:spdecomp}, the $P_V$ almost sure convergence follows from the strong law of 
large numbers for the sequence $\{X_k\}$.
\end{proof}

\subsection{Hausdorff and Box Dimensions}\label{hbd}   Assume that the family of IFSs ${\boldsymbol{F}}  $ satisfies the open set condition as  in Section~\ref{afn}.  We do not here require the affine  nested condition.  Recall the notation from Section~\ref{afn} and from Notation~\ref{notn1}.

Splitting up and treating the necks in the manner here was done first by Scealy in his PhD thesis~\cite{Sce}. 

\begin{thm} \label{bdhd}Suppose $ K $ is the random $ V $-variable fractal generated from  $ \boldsymbol{F} $.
Then the  Hausdorff and box dimension of   $K$ is $P_V $ a.s.\ given by the unique $\alpha$ such that
\begin{equation} \label{} 
E_V  \log \sum_{|\bfi|={n(1)}}  {\ell_{\boldsymbol{i}}}^  \alpha  = 0.
\end{equation} 
\end{thm}

\begin{proof}  See the Main Theorem in Section 4.4 of  \cite{BHS2}.  The expression there for the pressure function 
is equal to the simpler expression here.  This in turn leads to a simpler proof of that theorem, still along the lines 
of  Lemma~5.7 in \cite{BHS2} but working with a single neck as in the (somewhat more complicated) proofs   of Theorems~\ref{rdzp} and~\ref{thm:Nspecdim}.
\end{proof} 

We give a slight refinement of this result. 

\begin{thm}\label{thm:lil}
There exists a constant $C$ such that
\begin{equation}
\limsup_{k\to\infty} \frac{1}{ \sqrt{ k\log{ \log{k} } } } 
\log \sum_{| \boldsymbol{i} | = n(k)} { \ell_{\boldsymbol{i}}}^ { \alpha}  = C, \text{ $P_V$  a.s.}
\label{eq:dimlil}
\end{equation}
\end{thm}

\begin{proof}
We can apply Lemma~\ref{lem:sumprod} with $s_i = l_i$


Since $  E_V \log \sum_{|\bfi|={n(1)}}{ \ell_{\boldsymbol{i}}}^ {  \alpha}  = 0$, 
$ 
\lim_{k\to \infty}  \frac{1}{k} \log \sum_{| \boldsymbol{i} | = n(k)}   { \ell_{\boldsymbol{i}}}^ {  \alpha} =  0 \text{ a.s.} 
 $
Using the   bounds \eqref{Nbd} on   $ N^F$ and $ \ell^F $ it is easy to check that 
$ E_V\left( \log \sum_{|\bfi|={n(1)}}{\ell_{\boldsymbol{i}}}^ {  \alpha}\right)^2  < \infty$. 
The law of the iterated logarithm for the sequence of random variables 
$X_k=\log \sum_{| \boldsymbol{i} | = n(k)-n(k-1)} \ell_{\bfi}^{\alpha}$ now implies the result.
\end{proof} 

\section{Analysis on $V$-Variable Fractals} \label{secanv} 

\subsection{Overview} Our $V$-variable affine nested gaskets are connected and nested by Lemma~\ref{lem:cnprop} but 
they need not have spatial symmetry, in contrast to the scale irregular nested gaskets considered 
in \cite{barham}. 

In order to study analysis on these $ V$-variable  affine nested  fractals we define in Section~\ref{secdrf} their Dirichlet
forms and show that these are resistance forms. We also show that the resistance metric between points is comparable to an   appropriate product  of resistance factors.  In Section~\ref{secwm} we introduce general families of weights and measures and prove a few basic properties.
We introduce in Section~\ref{sectanc} the notion of the cut set $ \Lambda_k $, where each 
cut is at a neck level and the crossing time for the corresponding neck cell is of order $ e^{-k} $.  Asymptotic 
properties of various quantities associated with these neck cells are established.  In Section~\ref{secrd} we show the Hausdorff dimension in the resistance metric is given by the zero of an appropriate pressure function.

\subsection{Dirichlet and Resistance Form} \label{secdrf}

The construction of the Dirichlet form follows \cite{Kig}. 

Assume as given a harmonic structure $(D,\rho^F)$ for each IFS $F$ in
the family $\bff$. Since all our affine nested fractals are based on 
the same triangle or $d$-dimensional tetrahedron with vertices $ V_0 $, the matrix $D$ 
will be independent of $F$ and is
given by 
\begin{equation} \label{} 
D(x,y)=1,\ \forall x,y \in V_0  \text{ with } x\neq y, \quad   D(x,x)=-d\ \forall x \in V_0.
\end{equation} 
Vectors
$\rho^F=(\rho^F_1,\dots,\rho^F_{N^F})$, 
specifying the \emph{conductance scaling factors} to be applied to each cell, will be chosen to
respect the symmetries of the limiting fractal. 

Assume
\begin{equation} \label{rhoest} 
\begin{aligned}
 1<  \rho_{\inf}  &:=  \inf\{\rho^F_i  :     1\leq  i \leq N^F,\,  F\in \bff\}, \\
 \rho_{\sup}  &:=  \sup\{\rho^F_i  :   1\leq  i \leq N^F,\,  F\in \boldsymbol{F} \} < \infty.
\end{aligned} 
\end{equation} 
   
\bigskip 
The associated renormalization map for each $F \in \boldsymbol{F} $ is assumed to have the usual
fixed point property. We now state this more formally.

Let 
\begin{equation} \label{dfE0} 
\mathcal{E}_0 (f,g) = \frac12 \sum_{x,y \in V_0} \big(f(x)-f(y)\big)\big(g(x)-g(y)\big)
\end{equation} 
 be the
Dirichlet form on the graph $G_0=(V_0,E_0)$ with conductances
determined by the matrix $D$.  Each edge is summed over twice, and hence the factor $ 1/2 $.

The  choice of $\rho^F$ is such that
\begin{equation} \label{}
 \ce_0(f,f) = \inf\bigg\{\sum_{i=1}^{N^F}
\rho^F_i\ce_0\big(h\circ\psi_i^F,h\circ\psi_i^F\big)\,  \bigg| \,   h:V_1 \to \mathbb{R},\, h|_{V_0}=f\bigg\}. 
\end{equation} 
One can also regard this as placing conductors $\rho^F_i$ on each edge of
the 1-cell with address $i$, which ensures that the effective resistances between
  vertices from $G_0$ in the graph $G_1$
is the same as the effective resistance in $ G_0 $ itself --- see   Notation~\ref{notcc} and \eqref{dfrm}. 
 
Define $\cf_{n}:=\{f \mid f:V_n \to \br\}$. Use recursion on $ n \geq 1$ to define
\begin{equation} \label{dfen} 
\begin{aligned}
\ce_{n}^{\omega}(f,g) &=
\sum_{i=1}^{N^{\omega, \emptyset}}\rho_i^{\omega,\emptyset}  
\ce_{n-1}^{\sigma^i\omega} \big(f\circ\psi^{\omega,\emptyset}_i,
g\circ\psi^{\omega,\emptyset}_i\big) \quad \forall f,g \in \mathcal{F} _n^ \omega , \\
 \text{\ i.e.} \quad 
\ce_{n} (f,g)  &= \sum_{i=1}^{N^{ \emptyset}}\rho_i^{ \emptyset}  
\ce_{n-1}^{\sigma^i } \big(f\circ\psi^{ \emptyset}_i,
g\circ\psi^{ \emptyset}_i\big)\quad \forall f,g \in \mathcal{F} _n . 
\end{aligned} 
\end{equation} 
It follows that
\begin{equation} \label{} 
\ce_{n}(f,g) = \sum_{ | \boldsymbol{i} | = n}
    \rho_{ \boldsymbol{i} }\, \mathcal{E} _{0} \big(f \circ \psi_{ \boldsymbol{i} }, g \circ \psi_{ \boldsymbol{i} }\big),
\end{equation} 
where $ \rho_{ \boldsymbol{i} } $ and $   \psi_{ \boldsymbol{i} } $ are as in Notation~\ref{notn1}.

The sequence of forms $(\ce_{n},\cf_{n})$ can be thought of as
corresponding to conductances $ \rho_{ \boldsymbol{i} } $  on the edges of the cell 
$ \Delta_{ \boldsymbol{i} } $ in the graph $G_n$, where $ | \boldsymbol{i}  | = n $.

\medskip

One next defines a resistance form first on $ V_* :=   \bigcup_{n\geq 0} V_n  $ and then on its closure $ K $ in the standard manner as follows.
By the definition of the conductance scale factors $\rho_i^F$, one has
monotonicity of the sequence of quadratic forms $\ce_{n}^{\omega}(f,f)$. 
Define  
\begin{equation*} 
\ce^{\omega}(f,f) = \lim_{n\to\infty}
\ce^{\omega}_{n}(f,f) \quad
\end{equation*} 
for $ f : V_* \to \mathbb{R} $,  restricting to those $ f $ such that the limit is finite.
Using the definition of $ R $ in \eqref{dfrm} with $ K $ replace by $ V_* $, one shows that $ R $ is a metric on $ V_* $ as in Theorem 2.1.14, page 48 of \cite{Kig}.  
Noting definition \eqref{dfrsf}, one  next proves  the natural analogue of Lemma \ref{lem:resb} and Corollary~\ref{cor:diambd} for $ V_* $, without utilising Theorem~\ref{thmdrf}.  It follows that the metric $ R $ induces the Euclidean topology on $ V_* $  and the completion of this metric induces the Euclidean topology on $ K $.

\medskip

One can now define a limit form on $ K $ by
\begin{equation} \label{lmfm}
\ce^{\omega}(f,f) = \lim_{n\to\infty}
\ce^{\omega}_{n}(f,f) \quad \forall f\in \cf^{\omega} := \left\{ f: \sup_n
\ce^{\omega}_{n}(f,f)<\infty \right\}, 
\end{equation} 
where $ f : K \to \mathbb{R} $. Note that, from \eqref{dfrm}, if $ \mathcal{E} (f,f) < \infty $ then $ f $ is continuous and so is canonically determined on $ K $ by its values on the dense subset $ \bigcup_{n\geq 0} V_n \subset K $.

 \bigskip
 
It follows from the definitions that there is a decomposition of the limit form for any cut $ \Lambda $ of the tree $ T^ \omega $, see Notation~\ref{notn1}.
Namely,
\begin{equation}  \label{eq:formdecomp} 
\ce (f,g) = \sum_{\bfi\in \Lambda } \rho_{\bfi}\,
\ce^{\sigma^{\bfi} }(f\circ\psi_{\bfi}, g\circ\psi_{\bfi}) \quad \forall f,g\in \cf  . 
\end{equation}
 Note the   case  $ \Lambda = \{ \boldsymbol{i} : | \boldsymbol{i} | = k \}$ for some $ k $  and the case $ \Lambda = \Lambda_k $ as in \eqref{lkcut}. 
The result \eqref{eq:formdecomp} in the first of these cases with $ k=1 $  is essentially just a consequence of the scaling property \eqref{dfen} and letting $ n \to \infty $.  The general result follows from iterating this down the various levels corresponding to the partition $ \Lambda $.   

\bigskip

The \emph{effective resistance metric} between any pair of points $x,y \in K$
is defined by
\begin{equation} \label{dfrm}
\begin{aligned}
 R(x,y) ^{-1} &=  \inf\Big\{\ce(f,f): f(x)=0, f(y)=1\Big\} \\
&=  \inf \left\{ \frac{\ce(f,f)}{|f(x) - f(y)|^2} : f(x)\neq f(y) \right\} .
\end{aligned}
\end{equation} 
The proof this is a metric is essentially  as in Theorem 2.1.14, page 48 of \cite{Kig}.

\bigskip

Recall that $ (\mathcal{E} , \mathcal{F} ) $ is a \emph{local regular Dirichlet form} on $ L^2(K,\mu) $ if it 
has the following properties:
\begin{enumerate}
\item \emph{closed}: $ \mathcal{F}  $ is a Hilbert space under the inner product 
$ (f,g) \mapsto \mathcal{E} (f,g) + \int fg\, d \mu $;
\item \emph{Markov or Dirichlet}: $ \mathcal{E} (\overline{f},  \overline{f}) \leq \mathcal{E} (f,f) $ if $ \overline{f}$ 
is obtained by truncating $ f $ above by $ 1 $ and below by $ 0 $;
\item \emph{core or regular}: if $ C(K) $ is the space of continuous functions on $ K $ then $ C(K)\cap \mathcal{F} $ is 
dense in $ \mathcal{F} $ in the Hilbert space sense and dense in $ C(K) $ in the sup norm;
\item \emph{local}: $ \mathcal{E} (f,g) = 0 $ if $ f $ and $ g  $ have disjoint supports.
\end{enumerate} 
For $(\mathcal{E} , \mathcal{F} ) $ to be a \emph{resistance form} it is sufficient that in addition $ R $ defines a 
metric, and in particular that $ R(x,y) $ is finite and non zero if $ x \neq y $.

%
%

\begin{thm} \label{thmdrf} For each $\omega\in\Omega$ and each finite Borel regular measure $ \mu^\omega  $ on 
$ K^\omega $ with full topological support,  $(\ce^{\omega},\cf^{\omega})$   defines a local regular Dirichlet form on $L^2(K^{\omega},
\mu^\omega) $. 
The Dirichlet form is a resistance form with resistance metric $ R $.  
\end{thm}

\begin{proof} \mbox{}
The existence of the Dirichlet form $(\ce,\cf)$ as the limit of an increasing sequence of Dirichlet forms is essentially 
as summarised in the first paragraph of Section 3.4 of \cite{Kig}.  See \cite{Kig} Appendix B3 for a discussion of 
Dirichlet forms. The proof that the Dirichlet form is a resistance form is essentially as in Section 2.3 of~\cite{Kig}.
\end{proof} 

It will be convenient here and subsequently to work with \emph{resistance scaling factors} which are just the inverse of 
the conductance scaling factors introduced in Section~\ref{secdrf}. Thus we define
\begin{equation} \label{dfrsf}
r_i^F = (\rho_i^F)^{-1}, \  r_{ \boldsymbol{i} } = {\rho_{ \boldsymbol{i} }}^{-1}.
\end{equation}
We also note that for the resistance scale factors we have
\begin{equation} \label{rest} 
\begin{aligned}
0< r_{\inf}  &:=  \inf\{r^F_i  :      i  \in 1,\dots,N^F,\,  F\in \bff\} = {\rho_{\sup}}^{-1}, \\
r_{\sup}  &:=  \sup\{r^F_i  :  i\in 1,\dots,N^F,\,  F\in \boldsymbol{F} \} = {\rho_{\inf}}^{-1}  < 1.
\end{aligned} 
\end{equation} 

Next we see that the resistance metric distance between two vertices in a cell $ \Delta_{\bfi}$ (see \eqref{celldf}) is comparable to the resistance scaling factor $ r_{\bfi}$   for that cell.
\begin{lem}\label{lem:resb}
There is a constant nonrandom $ c_1 >0$ such that
if $x, y \in \Delta_{\bfi}$ and $ x \neq y $   then
\begin{equation} \label{estresb}  
c_1    r_{ \boldsymbol{i} }  \leq R(x,y) \leq  r_{ \boldsymbol{i} }. 
 \end{equation} 
\end{lem}

\begin{proof}  Fix $ x $, $ y $ and $ \bfi $ as in the statement of the lemma.

\medskip

If $ f(x) = 0 $ and $ f(y) = 1 $, then using \eqref{eq:formdecomp}, \eqref{dfE0}, monotonicity of the limit  in \eqref{lmfm}, and \eqref{V0ass},
\[
\ce(f,f)  \geq  
\rho_{\bfi} \ce_{0}(f\circ\psi_{\bfi},f\circ\psi_{\bfi})  
=   d\rho_{\bfi},
\]
where   $ d $ comes from the fact there are $ d $ edges in $ V_0 $ containing $ y $.

This gives  the upper bound for $ R $ in~\eqref{estresb}.

\medskip
For the lower bound, following Notation~\ref{notn1}, consider a cut $ \Lambda $ of the underlying tree such that  $  \boldsymbol{j} =j_1\dots j_n  \in \Lambda $ if $ r_{ \boldsymbol{j} } $ is comparable to $ r_{ \boldsymbol{i} } $.  More precisely,  $  \boldsymbol{j}    \in \Lambda $  if
\begin{equation} \label{rRint} 
 r_{ \boldsymbol{j} }  \leq r_{ \boldsymbol{i} }  \leq r_{j_1\dots j_{n-1}}.
\end{equation} 

Let $ \widetilde{V} = \bigcup_{ \boldsymbol{j} \in \Lambda } \psi_{\boldsymbol{j}} (V_0)    $ be the set of   vertices corresponding to cells $ \Delta_{ \boldsymbol{j} }$ for  $  \boldsymbol{j}  \in \Lambda $ (analogous to \eqref{GVE}).
Note that $ \boldsymbol{i} \in \Lambda $ and so $ x,y \in \widetilde{V} $.
Consider the function $ f $   such that $ f(y) = 1 $ and $ f(z) = 0 $ for all other $ z \in \widetilde{V} $, and harmonically interpolate.  
Then
\begin{equation} \label{Mprop} 
\ce(f ,f )  
 =   \sum_{ j \in \Lambda,\,  y\in\Delta_{\bfj}  }   \rho_{\bfj} \ce_{0}(f \circ\psi_{\bfj},f \circ\psi_{\bfj})   
=d  \sum_{j \in \Lambda,\,  y\in\Delta_{\bfj} }   \rho_{\bfj}   
   \leq \frac{dM}{r_{\inf}}   \rho_{\bfi}, 
\end{equation} 
using ~\eqref{rRint},  taking $ d $ as in  \eqref{V0ass},
and  $ M $ the maximum number of regular tetrahedra in $ \mathbb{R}^d $ with disjoint interiors that can have a common vertex.

This gives  the lower bound in~\eqref{estresb}.
\end{proof} 

\begin{cor}\label{cor:diambd}
There is an upper bound on the diameter  of the set $ K $ in the resistance
metric, in that there exists a nonrandom constant $C$ such that 
\begin{equation} \label{}
\diam_R K :=  \sup_{x,y\in K} R(x,y) \leq C.
 \end{equation} 
More generally, for all $\bfi\in T$, 
\begin{equation} \label{} 
\diam_R K_{\bfi} :=  \sup_{x,y\in K_{\bfi}} R(x,y) \leq Cr_{\bfi}.
 \end{equation} 
\end{cor}

\begin{proof} 
First consider points $x,y \in V_n$ (see \eqref{GVE}) and suppose $ x \in \Delta_{ \boldsymbol{i} } \subset V_n $, $ y \in 
\Delta_{ \boldsymbol{j} } \subset V_n $, with $ | \boldsymbol{i} | = | \boldsymbol{j} | = n $. 

Let $ x_0 = y_0 $, $ x_ k \in \Delta_{ \boldsymbol{i} | k } $, $ y_k \in \Delta_{ \boldsymbol{j} | k } $ for $ k = 1,\dots, n $, with $ x_n=x $ and $ y_n =y $. By the triangle inequality for the metric $R$,
\[
R(x,y) \leq \sum_{k=1}^{n } R(x_{k-1},x_{k}) +  \sum_{k=1}^{n } R(y_{k-1},y_{k}).
\]

Since $ x_{k-1}, x_k \in V_k $  and all cells are triangles or   tetrahedra, if a path from $ x_{k-1}$ to 
$  x_k  $ consisting of edges from $ V_{k} $ contains two edges from the same $k$-cell then it can be replaced by 
a shorter path from $ x_{k-1}$ to $  x_k  $ also consisting of edges from $ V_{k} $. 
It follows there is a path from $ x_{k-1}$ to $  x_k  $ consisting of at most $ N_{\sup} $ edges from $ V_k $.  Hence 
 \[
 R(x_{k-1}, x_k) \leq N_{\sup} r^k_{\sup},
 \]
 from \eqref{estresb}.  Hence
 \[
 R(x,y) \leq \frac{2N_{\sup}} {(1-r_{\sup})}.
 \]
 
Using the density of the vertices $\bigcup_n V_n$ in $K$ we have the result.

The second statement follows in the same way.
\end{proof}

Note that the result holds for all $\omega\in\Omega$.

\subsection{Weights and Measures} \label{secwm}

We next introduce a general family of measures on $ \partial T $ (see Notation~\ref{notn1}) and on the corresponding  fractal set  $ K $, by using a set of weights $(w_1^F,\dots,w_{N^F}^F)$ defined for each $F\in\bff$ with
$w_i^F>0$.  We do not require $ \sum_i w_i^F = 1 $.

\bigskip
Assume
\begin{equation} \label{west}
\begin{aligned} 
0< w_{\inf} &:=  \inf\{w^F_i :  1\leq i \leq N^F,\, F\in \bff\}, \\
 w_{\sup} &:=  \sup\{w^F_i : 1\leq i \leq N^F,\, F\in  \boldsymbol{F}  \} < \infty.
\end{aligned} 
\end{equation}

Following Notation~\ref{notn1} let the weight $ w_{\bfi}$ of the cell $\Delta_{\bfi } $ 
(corresponding complex, or corresponding cylinder) be the natural product of weights along the branch given by 
the node $\bfi$.  That is, if $ | \bfi | = n $, then
\begin{equation} \label{wprod}   
w _{ \bfi }  :=w_{i_1}^{F^\emptyset}   \cdot w_{i_2}^{F^{ i_1}} \cdot \ldots \cdot w_{i_n}^{F^{i_1\dots i_{n-1}}} .
\end{equation}

Of particular interest are weights of the form  $ w_i^F =(r_i^F)^{  \alpha } $  for all $ F \in \boldsymbol{F} $ 
and some fixed $\alpha>0$, in which case $w_{\bfi} = {r_{\bfi}}^{\alpha}$. This example is the reason we do not 
require $\sum_i w_i^F = 1 $, since it would not be possible to achieve the normalisation simultaneously for 
all $F \in \boldsymbol{F}$. 


\bigskip
The following   construction  is basic, and is special to the case of $ V $-variable fractals.
 
\begin{defn} \label{dfwm} Let $(w_1^F,\dots,w_{N^F}^F)$ for  $F\in\bff$ be a set of weights as before. 
For $ | \boldsymbol{i} |   $  a neck let 
\begin{equation} \label{dfmui} 
  \mu _{\bfi } := \mu ([ \bfi ] ) 
:= \frac{  w _{\bfi} }   { \sum_{|\boldsymbol{j}    | 
                     = | \boldsymbol{i} |   }   w _{\boldsymbol{j} } }.
                     \end{equation} 
The corresponding unit mass measure $ \mu $ on $ \partial T $ is called the \emph{unit mass measure with 
weights~$ w ^F_i $}.

The pushforward measure  on $ K $ under the address map $ \pi : \partial T \to K $ given by  
$ \bigcap_{n=1}^{\infty} K_{ \boldsymbol{i} | n } = \{  \pi (\boldsymbol{i} )\} $  is also denoted by $ \mu $.
\end{defn} 

Note that from  the definition of a neck, \eqref{dfmui}  is consistent via finite additivity from one level of neck 
to the next, it extends by addition to any complex or cylinder, and so by standard consistency conditions it extends 
to a unit mass   (probability) measure $ \mu $ on $\partial  T $.

We note the following simple estimates for use in the rest of this subsection and in Lemmas~\ref{lem:evalests},  
\ref{critest} and ~\ref{lem:muk}.

\begin{lem} \label{hineq} Suppose $\boldsymbol{i}$ and $\boldsymbol{j}$ are two nodes of the same type with
$| \boldsymbol{i} | = | \boldsymbol{j} | =  n $. Then
\begin{equation} \label{mibds}
\left(\frac{w_{\inf}}{w_{\sup}}\right)^n \mu_{ \boldsymbol{j} } \leq  \mu_{ \boldsymbol{i} } \leq 
\left(\frac{w_{\sup}}{w_{\inf}}\right)^n \mu_{ \boldsymbol{j}}.
\end{equation} 

If  $ \boldsymbol{i} $ is a neck node then
\begin{equation} \label{mibd} 
 \left(\frac{ w_{\inf}} {N_{\sup} w_{\sup}}\right)^n  \leq  \mu_{ \boldsymbol{i} }  < \left(1+ \left( \frac{w_{\inf}}{w_{\sup}}\right)^n \right)^{-1}<1.
 \end{equation}  
\end{lem}

\begin{proof} Suppose $ N $ is the first neck $ \geq n $. Then
\[
\mu_{ \boldsymbol{i} }  = \frac{
 \sum_{ | \boldsymbol{i} \boldsymbol{k} | = N ,\,   \boldsymbol{k} \in T^{ \sigma^{ \boldsymbol{i} }  } }\,  w_{ \boldsymbol{i} \boldsymbol{k}    }
  }
 { 
 \sum_{ | \boldsymbol{p} | = N,\, \boldsymbol{p} \in T}\, w_{  \boldsymbol{p}  }  
 }   
  = w_{ \boldsymbol{i} }\,  \frac{ 
  \sum_{  | \boldsymbol{k} | = N - n,\,  \boldsymbol{k} \in T^{ \sigma^{ \boldsymbol{i} } }    } \,    w^{\sigma^{ \boldsymbol{i}} } _ {\boldsymbol{k}} 
  } 
 { \sum_{| \boldsymbol{p} | = N,\, \boldsymbol{p} \in T }\, w_{  \boldsymbol{p}  }  
 } , 
 \]
where $ w^{\sigma^{ \boldsymbol{i}} }_{ \boldsymbol{k} } $ is the product of weights along any branch of 
$T^{\sigma^{ \boldsymbol{i}} } $ of length $ N - n $ beginning at $ \emptyset $, or equivalently any branch 
of $ T $ of length $ N - n $ beginning at $ \boldsymbol{i} $.
A similar expression is obtained for $ \mu_{\boldsymbol{j}}$.  Since $  \boldsymbol{i}   $ and $  \boldsymbol{j}$ 
are of the same type and level, the trees $T ^{\sigma^{ \boldsymbol{i}} } $ and  $ T ^{\sigma^{ \boldsymbol{j}} } $ 
are identical, and so 
 $\mu_{ \boldsymbol{i} } / w_{ \boldsymbol{i} } = \mu_{ \boldsymbol{j} } / w_{ \boldsymbol{j} } $.
 Then \eqref{mibds} follows from $ w_{\inf}^n \leq w_{ \boldsymbol{i} } \leq w_{\sup}^n $.

If $ n $ is a neck then 
\[
\mu_{ \boldsymbol{i} } = \frac{ w _ { \boldsymbol{i} } }  { \sum_{ | \boldsymbol{p} | = n}     w _ { \boldsymbol{p} }     }
  \geq \frac{ w_{\inf}^n }  { N_{ \sup } ^n w_{\sup}^n }. 
  \] 
Also we have that 
\begin{eqnarray*}
\mu_{ \boldsymbol{i} } = \frac{ w _ { \boldsymbol{i} } }  { \sum_{ | \boldsymbol{p} | = n}     w _ { \boldsymbol{p} }     }
&\leq & \left(1+ \sum_{|\boldsymbol{p}|=n, \boldsymbol{p}\neq \bfi} \frac{w_{\boldsymbol{p}}}{w_{\bfi}}\right)^{-1}  \\
 &\leq & \left(1+ (N_{\inf}^n-1) \left(\frac{w_{\inf}}{w_{\sup}}\right)^n\right)^{-1}  \\
 &<& \left(1+ \left(\frac{w_{\inf}}{w_{\sup}}\right)^n\right)^{-1} <1,
 \end{eqnarray*}
completing the proof of \eqref{mibd}.
  \end{proof}

We show in Lemma~\ref{simK} that the pushforward measure on $ K $ is given by a similar expression to that for $\mu$ on $\partial T$. 
For this we first show that the measure $ \mu $ on $ K $ is nonatomic.

\begin{lem}\label{noatoms}
$P_V$ a.s.\ we have for $\bfi\in\partial T$
\[ \mu (\boldsymbol{i} )  = 0. \]
\end{lem}

\begin{proof}
Since  $[\bfi|n]$ is a decreasing sequence of sets, from~\eqref{dfmui}
\[ \mu  ( \boldsymbol{i} ) = \lim_{k\to\infty} \mu_{\bfi|n(k)}. \]
By \eqref{eq:spdecomp} and \eqref{mibd}, writing $\zeta:=w_{\inf}/w_{\sup}$, the sequence of random variables 
\[ \mu^{(j)} = \frac{\mu_{\bfi|n(j)}}{\mu_{\bfi|n(j-1)}} = \frac{w_{\bfi|n(j)-n(j-1)}}
{\sum_{|\bfj|=n(j)-n(j-1)} w_{\bfj}} \leq \left(1+\zeta^{n(j)-n(j-1)}\right)^{-1}. \]
Taking logs and applying the law of large numbers we see,  $P_V$ a.s., for all $\bfi$,
\[ \lim_{k\to\infty} \frac{1}{k} \log  \mu_{\bfi|n(k)} \leq    -\lim_{k\to\infty} \frac{1}{k} \sum_{j=1}^k \log \left(1+\zeta^{n(j)-n(j-1)}
\right)  =- E_V \log \left(1+\zeta^{n(1)}\right). \]
Now, using the fact that $n(1)$ is a geometric random variable, $\zeta\leq 1$ 
and $\log(1+x) \geq x/2$ for $x\leq 1$, we conclude
\[ - E_v\log\left(1+\zeta^{n(1)}\right) \leq -\frac12 E_V \zeta^{n(1)} = -\frac12 \frac{\zeta}{E_V n(1)(1-\zeta)+\zeta}<0. \]
In particular, $P_V$ almost surely, for all $\bfi\in\partial T$, we have $ \lim_{k\to\infty} \mu_{\bfi|n(k)}= 0 $  (in fact exponentially fast) 
as required.
\end{proof}
  
 \begin{lem}\label{simK}  
 The address map $ \pi : \partial T \to K $ is one-one except on a countable set. The pushforward measure $ \mu $ on $ K $ is nonatomic.  Moreover, for $ | \boldsymbol{i} |   $  a neck,
 \begin{equation} \label{propmui} 
 \mu _{\bfi }  = \mu (K_{\bfi })   
= \frac{  w _{\bfi} }   { \sum_{|\boldsymbol{j}    | 
                     = | \boldsymbol{i} |   }   w _{\boldsymbol{j} } }.
                     \end{equation} 
 \end{lem} 
  
\begin{proof}
Suppose $ a = \pi ( \boldsymbol{i} ) = \pi ( \boldsymbol{j} ) $.  Then for some $ n $ we have $ i_1\dots i_n = j_1 \dots j_n $ and $ i_{n+1} \neq j_{n+1} $.
It follows that $ a  \in K_{i_1\dots i_n i_{n+1}} \cap K_{i_1\dots i_n j_{n+1}}  $.  From Lemma~\ref{lem:cnprop},
$ a  \in \psi_{i_1\dots i_n i_{n+1}}(V_0) $. This establishes countability of the set of points in $ K $ with more than one address.  From Lemma~\ref{noatoms} it follows this set has $ \mu $-measure zero.  The result \eqref{propmui} now follows from~\eqref{dfmui} and the definition of the pushforward measure.
\end{proof}

It follows from \eqref{propmui} that  
\begin{equation} \label{} 
\int_{ K_ { \bfi }   }f\, d\mu  =   \mu _{ \bfi }
     \int_{K^{\sigma^ { \bfi }  } }
 f \circ   \psi _{ \bfi } \  d\mu^{\sigma^ { \bfi }  } ,
\end{equation} 
where as usual  $ \mu = \mu^\omega $ is the measure on $ K = K^ \omega $ but here restricted to 
$ K_ { \bfi } = K^\omega _ { \bfi } $, and $ \mu^{\sigma^ { \bfi }  } =
\mu^{\omega , \sigma^ { \bfi }  } $ is the measure on $   K^{ \sigma^ { \bfi }  } =  K^{\omega, \sigma^ { \bfi }  }$ which 
is essentially just a scaled copy of the subfractal $ K_ { \bfi } $. 
By construction,  the left integral is a multiple of the right integral, with constant independent of~$ f $.  
Setting $ f = 1 $ gives the constant.  Note that $ | \boldsymbol{i} | = n $ need not be a neck.

The inner product (or any integral) can   be decomposed as follows:
\begin{equation} \label{inndec} 
(f,g)_{\mu  } = \sum_{ \bfi \in \Lambda  }    \mu  _ {\bfi }\,
            ( f\circ   \psi _{ \bfi }, g \circ   \psi _{ \bfi } )_{\mu^{\sigma^{ \bfi}   } }            
\end{equation}
for any cut $ \Lambda $, see Notation~\ref{notn1}.

Note that \eqref{inndec} is analogous to the decomposition \eqref{eq:formdecomp} for the Dirichlet form. 
The difference is that the scaling factors $\rho_{\boldsymbol{i}}$ in \eqref{eq:formdecomp} are simply computed 
from the prescribed quantities $ \rho_i^F $, unlike the scaling factors $ \mu_{ \boldsymbol{i} } $ in \eqref{inndec} 
which are  related to the prescribed quantities $ w_i^F $ in a simple manner only in the case where the $ \boldsymbol{i} $ 
are all neck nodes.    

We write 
\begin{equation} \label{}
\|f\|_2 =(f,f)_{\mu }^{1/2}
\end{equation}
for the natural norm on $L^2(K ,\mu )$.

\subsection{Time and Neck Cuts}\label{sectanc}
 

 We now introduce the  special cut sets which will be essential for our analysis. The idea is to cut at neck nodes in such 
 a manner that crossing times are comparable. 

\bigskip
Define 
\begin{equation} \label{dfti} 
 t_{ \boldsymbol{i} } = \mu_{ \boldsymbol{i} } r_{\boldsymbol{i} } .
 \end{equation} 
From the Einstein relation 
$t_{\bfi} $ can be thought of as a \emph{crossing time} for the continuous time random walk on the cell 
$\Delta_{\boldsymbol{i}}$, with resistance given by $r_{\boldsymbol{i}}$ and expected jump time given by~$\mu_{\boldsymbol{i}}$.  
	
	
Note that whereas $ w _{ \boldsymbol{i}} $ defined in \eqref{wprod} is a simple product of factors, as 
are $ \ell _ { \boldsymbol{i} } $, $ \rho_ { \boldsymbol{i} } $ and $r_ { \boldsymbol{i} } $ following 
the notation of \eqref{lprod}, this is not the case for $   \mu_{ \boldsymbol{i} } $ and hence not 
for $ t _ { \boldsymbol{i} } $.

%
%
%

\bigskip
Define
\begin{equation} \label{dfeta} 
\eta = \frac{ r_{\inf}\, w_{\inf}} {N_{\sup} w_{\sup}}
\end{equation}
and note that $0<\eta< 1$. Then from \eqref{dfti} and \eqref{mibd},
\begin{equation} \label{tbd} 
\eta^n\leq t_{ \boldsymbol{i} } \leq r_{\sup}^n \quad  \text{if } | \boldsymbol{i}|= n    \text{ is a neck}  .
\end{equation} 
The second inequality is clearly true for any $ \boldsymbol{i} $, not necessarily at a neck.  

\bigskip
 Recalling from \eqref{seqneck} the notation $ n (\ell) $ for the $ \ell $th neck, define the \emph{cut sets} of $T$
\begin{equation} \label{lkcut}
\Lambda_0 = \{ \emptyset \}, \quad \Lambda_k =\Big\{\bfi \in T \ : \ \exists \ell \Big(| \boldsymbol{i} | = n(\ell)  ,\ t_{\bfi} \leq e^{-k} < 
t_{\bfi|n(\ell-1)}\Big) \Big\} \text{ if } k \geq 1,
\end{equation}
where $ \emptyset $ is the root node.  Thus $ \Lambda_k $ is the set of \emph{neck}  nodes  for which the crossing times of the corresponding cells are comparable to~$e^{-k}$.

For any $ \boldsymbol{i} $ such that $ | \boldsymbol{i} | $ is a neck, and in particular if $ \boldsymbol{i} \in \Lambda_k $, then we define
\begin{equation} \label{dfl}
\ell( \boldsymbol{i} ) := \ell \quad \text{if} \quad |  \boldsymbol{i} | = n (\ell) .
\end{equation}
That is, $ \ell( \boldsymbol{i} ) $ is the number of the neck corresponding to $ \boldsymbol{i} $.


\bigskip
 We introduce further notation to capture the scale factors.
\begin{gather} \label{}
M_k = |\Lambda_k|   ,\quad 
\overline{t}_k  = M_k^{-1}\sum_{ \boldsymbol{i} \in \Lambda_k} t _ { \boldsymbol{i} } , \quad 
T_k  =  {\overline{t}_k}^{-1}; \label{dfmk} \\
y_k( \boldsymbol{i} )  = n(\ell) - n( \ell-1)\ \text{ if } \boldsymbol{i} \in \Lambda_k \text{ and } | \boldsymbol{i} | = n ( \ell) ,\quad  
y_k  = \max_{ \boldsymbol{i} \in \Lambda_k} y_k( \boldsymbol{i} ); \label{dfyk} \\
z_k =  \max\{ | \boldsymbol{i} | : \boldsymbol{i} \in \Lambda_k \} \label{dfzk} .
\end{gather} 

Thus $ M_k $ is the cardinality of the cut set $ \Lambda _ k $,  $ \overline{t}_k $ is the average crossing time for cells 
$K_ { \boldsymbol{i} } $ with  $  \boldsymbol{i} \in  \Lambda _k $ or equivalently the average time scaling when passing 
from $ K $ to $ K_ { \boldsymbol{i} } $, conversely $ T_k $ is the average time scaling when passing from  $ K_ { \boldsymbol{i} } $ 
to $ K $ for $  \boldsymbol{i} \in  \Lambda _k $;  $ y_k( \boldsymbol{i} ) $ is the number of generations between 
$ \Delta_{ \boldsymbol{i} } $ and its most recent ancestor also at a neck level, and $y_k $ is the maximum such number of ancestral generations over $ \boldsymbol{i} \in \Lambda_k $;
$ z_k $ is the maximum branch length of nodes in $ \Lambda_k $. 

Trivially,
\begin{equation} \label{} 
\min_{ \boldsymbol{i}  \in \Lambda_k} t_{ \boldsymbol{i} } \leq \overline{t}_k \leq \max_{ \boldsymbol{i} \in \Lambda_k } t_{ \boldsymbol{i} } .
\end{equation} 
 
\bigskip
For functions $ f  (k) $ and $ g  (k) $ we will use the notation
\begin{equation} \label{} 
f (k) \preccurlyeq g (k)\quad  \text{iff} \quad \limsup_{k\to \infty } \frac{f  (k) } { g (k) } \leq  1 .
\end{equation} 
That is, $f(k) \preccurlyeq g(k)$ means $ f $ is asymptotically dominated by $g $.

 In the next lemma we use Lemma~\ref{lem:geomrvs} to estimate the asymptotic behaviour of $ y_k $ and $ z_k$, 
and of the fluctuations of $ \ell( \boldsymbol{i}) $   and $t _{ \boldsymbol{i} }$ for $ \boldsymbol{i} \in \Lambda_k $. 
Note that sharper estimates for the simple case $V=1$ are given in Lemma~\ref{lem:spatialv=1}.
%



\begin{lem}\label{lem:spatial} Suppose $ \eta $ is as in \eqref{dfeta}.
\begin{itemize} 
\item[(a)]  There exist $ c_1,c_2>0 $ such that $ P_V $ a.s.,  if  $ \boldsymbol{i} \in \Lambda_k$  then
\[
 c_1 k (\log k)^{-1} \preccurlyeq \ell( \boldsymbol{i} )  \leq c_2 k .
\]
\item[(b)] There exist $ c_3, c_4>0 $ such that $ P_V $ a.s.
\[
1 \leq y_k \preccurlyeq c_3\log k , \quad  z_k\preccurlyeq  c_4 k .
\]
\item[(c)] There exists $ \beta' >0 $    such that $ P_V $ a.s.,  if $ \boldsymbol{i} \in \Lambda_k $ then
\[
k^{- \beta' } e^{-k} \preccurlyeq \eta ^{y_k}\, e^{-k} \leq t_{ \boldsymbol{i} }  \leq e^{-k}.
\]
\end{itemize} 
\end{lem}

\begin{proof} (a) \;\;  
Suppose $ \boldsymbol{i} \in \Lambda_k  $ and let $  \ell = \ell( \boldsymbol{i}  ) $.  From \eqref{tbd} and 
the definition of $ \Lambda _k $, 
\begin{equation} \label{bige} 
\eta ^{n(\ell)} \leq  t_{ \boldsymbol{i} } \leq e^{-k} < t_{ \boldsymbol{i} | n(\ell-1)} \leq r_{\sup}^{n(\ell-1)}.
\end{equation} 
In particular, $  n(\ell-1) \leq k /\log(1/r_{\sup}) $.  

 It follows that
\[
\ell = 1 + (\ell-1) \leq 1 + n(\ell-1) \leq 1 + \frac{k}{\log 1/r_{\sup} } \leq c_2 k.
\]

On the other hand from \eqref{bige},  $ n(\ell) \geq k/\log( 1/ \eta  ) $.  Using also $ \log k \geq \log \ell + \log  1/c_2  $, it follows 
from Lemma \ref{lem:geomrvs} \eqref{geomrvs2}, since $ n(\ell) = \sum_{i=1}^\ell \big(n(i) - n(i-1)\big) $    is a sum of geometric random variables, that  
a.s.\ (where $ \boldsymbol{i} \in \Lambda_k $) 
\[
\limsup_{k\to \infty} \frac{k}{\ell( \boldsymbol{i} )  \log k } \leq 
\limsup_{\ell \to \infty } \frac{ n(\ell) \log   1/\eta} { \ell( \log \ell + \log  1/c_2 ) }    \leq   \frac{2 \log 1/\eta }{\log 1/p} =: 1/c_1.
\]
Here $ p $ is the constant probability of not obtaining a neck at any particular level~$ \geq 1 $.

\bigskip
 \noindent (b)  \;\; Trivially, $ y_k \geq 1 $. By definition
\[
y_k  = \max_{ \boldsymbol{i} \in \Lambda_k} y_k( \boldsymbol{i} ) 
                  = \max_{ \boldsymbol{i} \in \Lambda _k } \Big(n \big( \ell ( \boldsymbol{i}  )\big) - n\big( \ell ( \boldsymbol{i} ) -1\big)\Big) 
       \leq  \max_{1\leq j \leq c_2 k} \Big(n(j) - n(j-1)\Big) ,
\]
   where the inequality comes from (a). 
   
   By Lemma \ref{lem:geomrvs} \eqref{geomrvs1} with $ Y_j = n(j) - n(j-1) $, $ P_V$  a.s.
   \[
   \limsup_{k \to \infty} \frac{y_k}{\log k} \leq \limsup_{k \to \infty} \frac{ \max_{1 \leq j \leq c_2 k} Y_j }{ \log c_2 k - \log c_2 }
   \leq \frac{1}{\log 1/p} =: c_3,
   \]
where $ p $ is as in (a).

It follows that with $ \boldsymbol{i} \in \Lambda_k $ and $ \ell = \ell( \boldsymbol{i}   ) $,
$ P_V $ a.s.
\begin{align*} 
| \boldsymbol{i} | &= n( \ell) = n( \ell - 1) +  n (\ell) - n( \ell -1) \leq n( \ell -1) + y_k \\
                & \preccurlyeq k / \log (1/ r_{\sup}) + c_3 \log k 
                 \preccurlyeq k / \log (1/ r_{\sup}) .
\end{align*} 
This gives the last inequality in (b).

\bigskip
 \noindent (c) \;\; The third inequality in (c) is immediate from the definition of $ \Lambda_k $.

\smallskip For the second inequality suppose $ \boldsymbol{i} \in \Lambda_k $ with $ | \boldsymbol{i} | = n(\ell) $.  Then 
\[
t_{ \boldsymbol{i} } = r _ { \boldsymbol{i} } \mu_{ \boldsymbol{i} } \geq 
               r _ { \boldsymbol{i} | n(\ell - 1) } \mu_{ \boldsymbol{i}| n(\ell - 1) } \eta^{ n(\ell) - n(\ell-1)}
 \]
by a similar argument to that for the first inequality in ~\eqref{tbd}.  More precisely, note that   by   definition   $ \mu_{ \boldsymbol{i} } $   is a product of    $ \mu_{ \boldsymbol{i}|n(\ell-1) } $     with factors that depend only on weights $ w $ defined along edges in the subtree rooted at $ \boldsymbol{i} | n(\ell-1) $, followed by a normalisation that depends only on the same weights since $ | \boldsymbol{i} | = n(\ell) $ is  a neck. 

Hence
\[
t_{ \boldsymbol{i} }  \geq     t_  { \boldsymbol{i} | n(\ell - 1) }     \eta^{ n(\ell) - n(\ell-1)}   \\
                 \geq  e^{-k}     \eta^{ y_k  }  
\]
by the definition of $ y_k $ and $ \Lambda_k $.  This  gives the second inequality in (c).

\smallskip For the first inequality take any $ \epsilon > 0 $, in which case by (b), a.s.\ there exists 
$ k_0 = k_0( \omega ) $ such that 
$ k \geq k_0 $ implies $ y_k \leq (c_3 + \epsilon ) \log k $, and so $ k \geq k_0 $ implies
\[
\eta^{y_k} \geq \eta^{(c_3 + \epsilon ) \log k} = k^{- ( c_3 + \epsilon ) \log 1/\eta} = k^{- \beta' }, 
\]
where $ \beta' = (c_3 + \epsilon ) \log 1/\eta $. Since $ \epsilon > 0 $ is arbitrary,
this completes the proof.
\end{proof}   
 
If $V=1$ the above can be sharpened to the following.

\begin{lem}\label{lem:spatialv=1}
In the case $V=1$ we have the following.
\begin{itemize} 
\item[(a)]  There exist $ c_1,c_2>0 $ such that  if  $ \boldsymbol{i} \in \Lambda_k$  then
\[ c_1 k  \leq \ell( \boldsymbol{i} )  \leq c_2 k .
\]
\item[(b)] There exists $ c_3>0 $ such that
\[ y_k =1, \quad  z_k \leq  c_3 k . \]
\item[(c)] There exists $ c_4>0 $  such that if $ \boldsymbol{i} \in \Lambda_k $ then
\[
c_4  e^{-k} \leq t_{ \boldsymbol{i} }  \leq e^{-k}.
\]
\end{itemize} 
\end{lem} 
 
\begin{proof}
The first claim follows from \eqref{tbd} and the fact that for $V=1$ every level is a neck. The second and third follow similarly.
\end{proof}

\subsection{The Haudorff Dimension in the Resistance Metric} \label{secrd}

\begin{defn} The $ \alpha $-dimensional Hausdorff measure of $ K $ using the resistance metric $ R $  is denoted by  $ \mathcal{H}_R^ \alpha (K)  $.  The Hausdorff dimension of $ K$ in the resistance metric  is denoted by $ d_f^r = d_f^r(K)$.  
\end{defn}

The following theorem is the analogue of Theorem~\ref{bdhd}. However, the resistance metric $ R $ does not scale in the same way as the standard metric in $ \mathbb{R}^d $ and so the proof needs to be modified.   The proof combines ideas from Section~2 of~\cite{Kig-1},  Section~2 of~\cite{Kig-2} and Section~4 of~\cite{BHS2}.  In the case of  \cite{BHS2} the corresponding argument   is simplified here because of the use of necks. Note that  we do not expect  the appropriate Hausdorff measure function to be a power function, unlike in~\cite{Kig-1} and ~\cite{Kig-2}.

\begin{thm} \label{rdzp}
 The Hausdorff dimension in the resistance metric $ d^r_f$ of  $ K $ is the unique power $ \alpha_0 $ such that 
\begin{equation} \label{dfxa1}
E_V \log\sum_{|\bfi|=n(1)} r_{\bfi}^{ \alpha _0}=0,
\end{equation} 
\end{thm} 

\begin{proof}
This will follow from Lemmas~\ref{lemubd} and~\ref{lemlbd}.
\end{proof} 

\begin{lem}\label{dfa0}
The function 
\begin{equation} \label{}
\gamma( \alpha ) := E_V \log\sum_{|\bfi|=n(1)} r_{\bfi}^{ \alpha },
\end{equation} 
 is finite, strictly  decreasing and Lipschitz, with derivative in  the interval
\[
 [ (\log r_{\inf}) E_V n(1), (\log r_{\sup}) E_V n(1)] .
 \]  
Since $ \gamma (0) > 0 $ there is a unique $ \alpha_0 $ such that $ \gamma ( \alpha_0  )= 0 $ and moreover 
$ \alpha_0 > 0 $.
\end{lem} 

\begin{proof}
If $ \alpha < \beta $, then from \eqref{rest},
\[
 \gamma ( \alpha ) + (\beta - \alpha )(\log r_{ \inf })E_V n(1) 
\leq \gamma ( \beta ) 
\leq \gamma ( \alpha ) + (\beta - \alpha ) (\log r_{ \sup } )E_V n(1) .
\]
This gives the Lipschitz estimate.

Since  $ \gamma (0) = E_V \bigl( \log \# \{ \boldsymbol{i} \in T \mid |  \boldsymbol{i}  | = n(1) \} \bigr) $, it follows that $ 0
 < \gamma (0) < \infty $.

The rest of the lemma now follows.
\end{proof}

\begin{lem}\label{lemubd}
Suppose  $ \alpha _0 $ is as in Lemma~\ref{dfa0}.  Then $ d^r_f(K) \leq \alpha _0 $, $ P_V $~a.s.
\end{lem} 

\begin{proof} Suppose $ \alpha > \alpha _0 $.
Using Corollary~\ref{cor:diambd}, 
\[
K = \bigcup_{| \bfi | = n(k) }K_{ \bfi }, \qquad
 \sum_{| \bfi | = n(k) } \diam_R^{ \alpha } K_{ \bfi } \leq C^ \alpha \sum_{| \bfi | = n(k) } r^{ \alpha }_{ \bfi } .
 \]
  From \eqref{eq:slln} and Lemma~\ref{dfa0},  
  \[
 \lim_{k\to \infty} \frac{1}{k} \log \sum_{ | \boldsymbol{i} | = n(k) } r^ \alpha _{ \boldsymbol{i} } = E_V  \log \sum_{ | \boldsymbol{i} | = n(1) }  r^ \alpha_{ \boldsymbol{i} } < 0, \quad P_V \text{ a.s.}  
\] 
Hence $ P_V $ a.s.,
\[   
\lim_{k\to \infty}  \log \sum_{ | \boldsymbol{i} | = n(k) } r^ \alpha _{ \boldsymbol{i} } = -\infty,
\qquad 
 \lim_{k\to \infty}   \sum_{ | \boldsymbol{i} | = n(k) } r^ \alpha _{ \boldsymbol{i} } = 0.
 \]
 Hence $  \mathcal{H}_R^ \alpha (K)  = 0 $, and so $ d_f^r(K) \leq \alpha _0 $.
  \end{proof} 

\begin{defn} Suppose $ \epsilon > 0 $.  Then $ \Lambda_ \epsilon  $ is the cut set of $T$ consisting of  those nodes  $  \boldsymbol{j} =j_1\dots j_n  $ such that 
\begin{equation} \label{rRinte} 
 r_{ \boldsymbol{j} }  \leq  \epsilon \leq r_{j_1\dots j_{n-1}}.
\end{equation}   
\end{defn} 

\begin{lem}\label{dislem}
 There exist non random constants $ c $ and $ M_1 $, such that for any  $ \epsilon >0$ and  $ x \in K $,  
\begin{equation} \label{bkest}
\# \left\{ \boldsymbol{j}  \in \Lambda _ \epsilon :  \overline{B}_{c \epsilon } (x) \cap K_{ \boldsymbol{j} }   \neq \emptyset \right\} \leq M_1,
\end{equation} 
where 
\[   
\overline{B}_{c \epsilon } (x)  = \{ y \in K : R(x,y) \leq c \epsilon \}.
\]
\end{lem}

\begin{proof}
Suppose $ x \in K_{\boldsymbol{i}}$ where $   \boldsymbol{i} \in \Lambda_ \epsilon $.

\medskip
First note
\begin{equation} \label{cibd}
\# \left\{ \boldsymbol{j}  \in \Lambda _ \epsilon :  K_{\boldsymbol{i}} \cap  K_{\boldsymbol{j}} \neq \emptyset \right\} \leq M(d+1),
\end{equation} 
where   $ d+1 $ is the number of vertices of a regular tetrahedron in $ \mathbb{R}^d $ (recall \eqref{V0ass}) and $ M $ is as in~\eqref{Mprop}.
This follows immediately from  Lemma~\ref{lem:cnprop}.
  
\medskip
   Let $ V_ \epsilon = \bigcup_{ \boldsymbol{j} \in \Lambda_{\epsilon} } \psi_{\boldsymbol{j}} (V_0) =: \bigcup_{ \boldsymbol{j} \in 
   \Lambda_{\epsilon} } V_{ \boldsymbol{j}}   $ denote the set of  vertices corresponding to the partition $ \Lambda_ \epsilon $.  
   
   Define $ u: V_ \epsilon \to \mathbb{R} $ by $ u(y) = 1 $ if $ y \in V_{\boldsymbol{i}} $ and $ u(y) = 0 $ otherwise.  Extend $ u $ to $ u:K \to \mathbb{R} $ by harmonic extension on each $ K_{\boldsymbol{j}}$ for $\boldsymbol{j} \in \Lambda_ \epsilon $.  Then $ u $ is constant on 
 $ K_{\boldsymbol{j}} $  if    $ K_{\boldsymbol{i}} \cap  K_{\boldsymbol{j}} = \emptyset $, and so
 \[
 \mathcal{E}(u) = \sum_{ \{ \boldsymbol{j} :  K_{\boldsymbol{i}} \cap  K_{\boldsymbol{j}} \neq \emptyset  \}   }   
 \rho_{\boldsymbol{j}}  \mathcal{E}_0(u\circ \psi_{\boldsymbol{j}} ) \leq M(d+1)d  \rho_{\boldsymbol{j}}
 \leq \frac{Md(d+1)}{r_{\inf } \epsilon }   ,
\]
where $ M(d+1) $ is from \eqref{cibd} and   $ d $ is the number of edges in $   \psi_{\boldsymbol{j}} (\Delta_0) $ with one vertex in $ V_{ \boldsymbol{i}} $.  

Setting $ c = r_{\inf}/ 2Md(d+1) $, it follows   $ R(x,y) > c \epsilon $ if $ y \in K_{\boldsymbol{j}}$ where $\boldsymbol{j} \in \Lambda_ \epsilon $  and $ K_{\boldsymbol{i}} \cap  K_{\boldsymbol{j}} = \emptyset $.  That is,
\begin{equation} \label{BKc}
    \overline{B}_{c \epsilon } (x) \cap K_{ \boldsymbol{j} } \neq \emptyset   \Longrightarrow  K_{\boldsymbol{i}} \cap  K_{\boldsymbol{j}} \neq \emptyset .
\end{equation}

Combining \eqref{BKc} and \eqref{cibd} gives \eqref{bkest}.
\end{proof}

\begin{lem}\label{lemlbd}
Suppose $ \alpha < \alpha _0 $. Let $ \mu $ be the unit  mass measure on $ K $ constructed as in Definition~\ref{dfwm} and Lemma~\ref{simK}, with weights $ w_i^F = (r_i^F)^ \alpha $ for  $F\in\bff$. Then $ P_V $ a.s., for any $ x \in K $ and   $ \delta > 0 $, $ \mu\big(B_ \delta (x) \big) < c_1 \delta ^ \alpha $, where the random constant $ c_1 $ depends on $ \omega $ but not on $ x $ or $ \delta $.

In particular, by the mass distribution principle, $ d_f^r(K) \geq \alpha   $ $ P_V $~a.s., and so $ d_f^r(K) \geq \alpha_0   $ $ P_V $~a.s.
\end{lem}

\begin{proof}
Fix $ x \in K $ and $ \delta > 0 $.  
If $ k$ is a level in the construction of $ T $, let $ s(k) $ denote the first neck level $ \geq k $. 
All balls are with respect to the resistance metric.

\medskip
From Lemma~\ref{dislem} applied to the cut $ \Lambda_{ \delta / c } $, and with $ c $ and $ M_1 $ as in that lemma, there are at most $ M_1 $ sets $ K_{\boldsymbol{j}} $ which meet $ \overline{B}_{\delta} (x) $ and satisfy $ \boldsymbol{j} \in \Lambda_{ \delta / c } $.  That is, satisfy, on setting $ \boldsymbol{j} = j_1\dots j_k $, 
\begin{equation} \label{partdc} 
r _{ \boldsymbol{j} } \leq \delta  / c < r_{j_1\dots j_{k-1}} .
\end{equation} 
It follows that
\begin{equation} \label{compdr}
\mu\big( B_{ \delta }(x) \big) \leq \sum_ { \boldsymbol{j} \in \Lambda_{ \delta /c } , B_{ \delta }(x) \cap K_{ \boldsymbol{j} } \neq \emptyset} 
\mu(K_{ \boldsymbol{j} } ) ,
\end{equation} 
and there are at most $ M_1 $ terms in the sum.
For each such $ K_{ \boldsymbol{j} } $, using Lemma~\ref{simK},
\begin{equation} \label{Kjest}
\begin{aligned}  
\mu (K_{ \boldsymbol{j} } ) &\leq \sum_{ \boldsymbol{j} \prec \boldsymbol{i} , | \boldsymbol{i} | = s(k) } \mu (K_{ \boldsymbol{i} }) =
\frac{  \sum_{ \boldsymbol{j} \prec \boldsymbol{i} , | \boldsymbol{i} | = s(k) } r_{ \boldsymbol{i} } ^\alpha } 
            { \sum_{  | \boldsymbol{i} | = s(k) } r_{ \boldsymbol{i} } ^\alpha}\\
&\leq \frac{ N_{\sup}^{ s(k) - k}  }   { \sum_{ | \boldsymbol{i} | = s(k) }  r_{ \boldsymbol{i} }^\alpha }\, r_{ \boldsymbol{j} }^\alpha 
\leq \frac{  N_{\sup}^{ s(k) - k}  }   { c^\alpha\sum_{ | \boldsymbol{i} | = s(k) }  r_{ \boldsymbol{i} }^\alpha }\,\delta ^\alpha
=: \theta(k) \, \delta ^ \alpha .
\end{aligned} 
\end{equation} 
Here $ N_{\sup} $ is  an upper bound for the branching number, see~\eqref{Nbd}.

\medskip
We need to estimate the  numerator and denominator of $\theta(k)$  in~\eqref{Kjest}.
 For this we use  estimates~\eqref{skest} and ~\eqref{sumraest}.  
 
 \emph{Until we establish ~\eqref{sumraest} we allow $ k $ to be an arbitrary positive integer,} not necessarily satisfying \eqref{partdc}.

Since $ s(k) - k $ is a geometric random variable, by the same argument as in Lemma~\ref{lem:spatial}(b), there is a constant $ c_1 $ such that 
$ s(k) - k  \preccurlyeq c_1 \log k  $ 
$ P_V $ a.s., and so   there is a constant $ c_2(\omega) $ such that 
\[
s(k) - k \leq c_2 \log k \quad P_V \text{ a.s.}
\]
for all $ k >1 $.
Hence $ P_V $ a.s., for $ k>1 $, 
\begin{equation} \label{skest}
 N_{\sup}^{ s(k) - k} \leq  N_{\sup}^{ c_2 \log k }
\end{equation} 

Next let 
\[
\beta = E_V \log \sum_{ | \boldsymbol{i} | = n(1) } r_{ \boldsymbol{i} } ^ \alpha .
\] 
Then $ \beta > 0 $ since $ \alpha < \alpha _0 $, see Lemma~\ref{dfa0}. It follows by \eqref{eq:slln} that as $k\to\infty$  
\[
\frac{1}{k} \log \sum_{ | \boldsymbol{i} | = n(k) } r_{ \boldsymbol{i} }^ \alpha \to \beta  \quad P_V \text{ a.s.}
\]
Hence for some $ \epsilon _0 = \epsilon _0 ( \omega ) > 0 $,
\begin{equation} \label{sumraest1}
\sum_{ | \boldsymbol{i} | = n(k) } r_{ \boldsymbol{i} }^ \alpha \geq \epsilon _0\, e^{k \beta / 2 } \quad \text{for } k >1 ,\quad   P_V \text{ a.s.}
\end{equation} 

However, we need an estimate similar to \eqref{sumraest1} involving $ s(k) $  rather than   $ n(k) $. 
First note, by setting $ Y_i = n(i) - n(i-1) $ and $ n=1 $ in   \eqref{geomrvs2},    that for some $ c_3 = c_3( \omega ) $ we have 
$ n(k) \leq c_3 k \log k $ $ P_V $ a.s.\ if $ k>1 $. Hence
\[
\sum_{ | \boldsymbol{i} | = n(k) } r_{ \boldsymbol{i} }^ \alpha \geq \epsilon _0 \exp\left(\frac{n(k) \beta}{2 c_3 \log k}\right) 
\quad \text{for } k >1 ,\quad   P_V \text{ a.s.}
\]
Since $ n(k) $ is an arbitrary neck,
\[
\sum_{ | \boldsymbol{i} | = s(k) } r_{ \boldsymbol{i} }^ \alpha \geq \epsilon _0 \exp\left(\frac{s(k) \beta}{2 c_3 \log \widetilde{k} }\right)
 \quad \text{for } k >1 ,\quad   P_V \text{ a.s.}
\]
where $ \widetilde{k} $ is the number of the neck $ s(k) $.  
Note $ s(k) \geq k $. Also note that $ \widetilde{k} \leq k $. (Otherwise there are at least $ k+ 1 $ necks between levels 1 and $ s(k) $ inclusive, and so in particular $ s(k) > k $.  But then there are at least $ k  $ necks between levels 1 and $ k $ inclusive, and so $ k $ is a neck.  However that gives $ s(k) = k $, a contradiction).  Hence
\begin{equation} \label{sumraest}
\sum_{ | \boldsymbol{i} | = s(k) } r_{ \boldsymbol{i} }^ \alpha \geq \epsilon _0\, \exp\left(\frac{k\beta}{ 2c_3 \log k}\right)
 \quad \text{for } k >1 ,\quad   P_V \text{ a.s.}
\end{equation}

It follows from ~\eqref{skest},  \eqref{sumraest} and the definition of $ \theta (k) $ in \eqref{Kjest},   that $ \theta (k) \to  0 $ as 
$ k \to \infty $.
On the other hand, with $ k := | \boldsymbol{j} | $ we have from \eqref{partdc} that 
\[  k := | \boldsymbol{j} | \geq \log(c/ \delta ) / \log (1/r_{\min}) \to \infty \]
uniformly for $ \boldsymbol{j} \in \Lambda_{ \delta / c } $ as $ \delta \to 0 $.  From  \eqref{Kjest}, \eqref{compdr}  and the uniform bound $ M_1$ on the number of terms, there exists $ \delta _0=  \delta _0(\omega) > 0 $ such that 
\begin{equation} \label{} 
 \mu\big( B_{ \delta }(x) \big) \leq  \delta ^ \alpha \text{ for } \delta \leq \delta _0  \quad P_V \text{ a.s.}.
 \end{equation}
 
It now follows by the mass distribution principle that  $ d_f^r(K) \geq \alpha   $ $ P_V $~a.s., and so $ d_f^r(K) \geq \alpha_0   $ $ P_V $~a.s.
\end{proof}

\section{Eigenvalue Counting Function}

\subsection{Overview}
In this section we consider random $V$-variable fractals constructed from essentially arbitrary resistances 
$ r_i^F $, from weights $ w_i^F $ which determine a 
measure $ \mu $, and from a probability measure $P$ on $\boldsymbol{F}$. See Sections~\ref{rvtf}, \ref{secdrf} 
and~\ref{secwm}. With every realisation of such a random fractal there is an associated Dirichlet form 
and a Laplacian. The growth rate of the corresponding eigenvalue counting function is defined to be $ d_s/2 $, 
where $ d_s $ is called the \emph{spectral exponent}. We see in Theorem~\ref{thm:Nspecdim} that $P_V$-a.s.\ $ d_s $ 
exists, is constant and is the zero of a pressure function constructed from the crossing times 
$ t_{\boldsymbol{i} } $. The proof relies on estimates concerning the occurrences of necks and on a 
Dirichlet-Neumann bracketing argument, see Lemmas~\ref{lem:spatial} and~\ref{seceest}. Lemma~\ref{seceest} gives
a result which holds for all realizations. (In the case 
$ w_i^F = (r_i^F)^{-1} $ some of the estimates can be sharpened, see Remark~\ref{casepl}.)

The natural metric on fractal sets constructed with resistances as here is the resistance metric. 
We saw in Theorem~\ref{rdzp} that the   Hausdorff dimension $d_f^r $ in this metric is given by the zero of a certain pressure function.
 A natural set of weights is  $ w_i^F = (r_i^F)^{d_f}$. 
The   measure $ \nu $  constructed from this set of weights   is called the \emph{flat measure}  with respect to the resistance metric. 

We see in Theorem~\ref{spfm} that $  d_s(\nu)/2 =  d_f^r/(d_f^r+1) $.  This establishes the analogue of Conjecture~4.6 in \cite{Kig-3}
for $ V $-variable fractals.  In Theorem~\ref{spmax} we show that for a fixed set of resistances $ r_i^F $, and for arbitrary weights $ w_i^F $ and corresponding
measure $ \mu $, the spectral exponent $ d_s(\mu) $  has a unique maximum when $ \mu $ is the flat measure  $ \nu$.  The spectral exponent in this case is called the \emph{spectral dimension} associated with the 
given resistances.

Finally, in the case of the flat measure $ \nu $, we give in Theorem~\ref{imprest} an improved $P_V$ almost sure estimate 
for the counting function itself rather than its log asymptotics. 

\subsection{Preliminaries}

Following the notation of the previous section, we consider a fractal $ K =K^\omega $ and write $\partial K = V_0$ 
for the boundary of  $K$. We fix a measure $\mu =  \mu^\omega  $  on $K$ and, together with the Dirichlet form 
$ \mathcal{E} = \mathcal{E}^ \omega $, this allows one to define a Laplace operator 
$ \triangle_\mu=\triangle_\mu^ \omega $.
%
%
We will be interested in the spectrum of $-\triangle_\mu$ as this consists of positive eigenvalues. However, 
instead of working directly with $-\triangle_\mu$, we use a formulation of the Dirichlet and Neumann eigenvalue problems
in terms of the Dirichlet form, see \cite{Kig}. 

\bigskip
Recall the definition of $ \mathcal{F} $ from~\eqref{lmfm}. Let 
\begin{equation} \label{DDform}
\cf^{\omega}_D =
\{ f\in \cf^{\omega} : f(x)=0,\ x\in \partial K\}, \quad \ce^{\omega}_D(f,f) =
\ce^{\omega}(f,f)  \text{ for } f\in \cf^{\omega}_D,
\end{equation} 
and let $(\cdot,\cdot )_{\mu^{\omega}}$ be the inner product on
$L^2(K^{\omega},\mu^{\omega})$. It follows as in Theorem~\ref{thmdrf} that $(\ce_D,\cf_D)$ is a local regular Dirichlet 
form on $L^2(K \setminus \partial K,\mu)$. Now  \emph{$\lambda$ is a  Dirichlet eigenvalue with
eigenfunction $u \in \mathcal{F} _D ^\omega$, $u\neq 0$}, if 
\begin{equation} \label{dfde}
\ce^{\omega}_D(u,v) = \lambda(u,v)_{\mu^{\omega}}  \;\;\forall v\in \cf^{\omega}_D.
\end{equation} 

Similarly, \emph{$\lambda$ is a Neumann eigenvalue with eigenfunction 
$u\in \mathcal{F} ^\omega$, $u\neq 0$}, if
\begin{equation} \label{dfne}
\ce^{\omega}(u,v) = \lambda(u,v)_{\mu^{\omega}}  \;\;\forall v\in \cf^{\omega}.
\end{equation} 

As usual, we will in future normally omit the dependence on $\omega$.

\bigskip

By standard results \cite{Kig} the Dirichlet Laplacian has a
discrete spectrum  
\begin{equation} \label{} 
0< \lambda_1 < \lambda_2 \leq
\dots  \text{ where  } \lambda_n\to\infty \text{ as } n\to \infty,
\end{equation} 
and similarly for the Neumann Laplacian but with $ 0 = \lambda_1 $.

The Dirichlet and Neumann eigenvalue counting functions are defined by
\begin{equation} \label{}   \label{dfdec} 
\begin{aligned} 
\mathcal{N}  _D(s) &= \max\{i: \lambda_i  \leq s,\ \lambda_i \mbox{ is a Dirichlet eigenvalue}\},\\
\mathcal{N} _N(s) &= \max\{i: \lambda_i \leq  s,\  \lambda_i \mbox{ is a Neumann eigenvalue}\}.
\end{aligned} 
\end{equation} 
As usual, eigenvalues are counted according to their multiplicity. 

\bigskip

The following lemma implies the spectral exponent $ d_s(\mu)  $ in Definition~\ref{dfspm} is at most~$ 2 $ for 
any realization of our $V$-variable fractals. It is used to prove the second estimate in Lemma~\ref{firsteest}. 

\begin{lem} \label{lem:lingrowth}
With the same constant $ C $ as in Corollary~\ref{cor:diambd},
\[ 
\mathcal{N} _D(s) \leq Cs, \quad  \forall s>0.
\]
\end{lem}

\begin{proof} The effective resistance between $ x\in K $ and the boundary set $ \partial K = V_0 $ is defined by
\[
R(x,\partial K)^{-1} = \inf \big\{ \mathcal{E} (f,f) : f \in \mathcal{F} _D,\, f(x) = 1 \big\}.
\]
From Corollary~\ref{cor:diambd} with the same constant $ C $, and for any $ y \in \partial K $,
\[
R(x,\partial K) \leq R(x,y)  \leq C .
\]

The Green function for the Dirichlet problem in $ K $
%
is a symmetric function $ g(x,y) $ which has $g(x,y) \leq  g(x,x) = R(x, \partial K ) $.  
See, for example, Proposition 4.2 of \cite{Kig1}. In particular, 
\[
g(x,y) \leq C  
\]
independently of $ \omega $. Moreover, from Theorem 4.5 of \cite{Kig1},
\[
\big| g(x,y) - g(x,z) \big| \leq R(y,z).
\]
Hence $g$ is continuous, and in particular uniformly Lipschitz continuous, in the resistance metric.  

It follows from Mercer's theorem (for a proof of the theorem see the argument in \cite{Lax} pages 344--345) that  
\[ 
g(x,x) = \sum_{i\geq 1} \left(\lambda _i^D\right)^{-1} \phi_i(x)^2 
\]
and the series converges uniformly, where $ \phi_i $ are the orthonormal eigenfunctions corresponding to 
the Dirichlet eigenvalues $ \lambda _i^D $. Integrating with respect to $ x $,  
\[
C  \geq \sum_{i\geq 1} \left(\lambda _i^D\right)^{-1}  
      \geq \frac{1}{s} \mathcal{N} ^D(s),
\]
for any $ s>0 $.     
\end{proof} 

\subsection{Dirichlet-Neumann Bracketing}
In this and the following sections, fix a set of weights $ w_i^F$  as in Section~\ref{secwm} and let $ \mu $ be the 
corresponding measure.

In order to deduce properties of the counting function for $V$-variable fractals we use the method of 
Dirichlet-Neumann bracketing. 

\bigskip
Let  $\Lambda_k$  be the sequence of cutsets  \eqref{lkcut}. Using the notation of  \eqref{V0ass} and analogously to 
\eqref{GVE}, define
\begin{equation} \label{}   \label{} 
\begin{aligned} 
\widetilde{V}_k &= \bigcup \{ \psi_{\boldsymbol{i}}(V_0)   : \boldsymbol{i} \in \Lambda_k \}, \\
\widetilde{E}_k &= \bigcup \{ \psi_{\boldsymbol{i}}(E_0 )  : \boldsymbol{i} \in \Lambda_k \},\\
\widetilde{G}_k & = (\widetilde{V}_k, \widetilde{E}_k).
\end{aligned} 
\end{equation} 
Thus $ \widetilde{G}_k = (\widetilde{V}_k,\widetilde{E}_k)$  is the graph associated 
with the vertices $\widetilde{V}_k$ of the cells determined by $\Lambda_k$. 
 
Define $(\mathcal{E}^k,\mathcal{F}^k)$ by
\begin{equation} \label{Dkform}
\begin{aligned} 
\mathcal{F}^k &= \big\{f: K\backslash \widetilde{V}_k\to \br \ \big| \  
\forall \boldsymbol{i} \in \Lambda_k \  \exists f_{\bfi} \in \cf^{\sigma^{\bfi}} :
f\circ \psi_{\bfi} = f_{\bfi} \mbox{ on } K^{\sigma^{\bfi}} \backslash \partial K^{\sigma^{\bfi}}  
\big\}, \\
 \mathcal{E}^k(f,g) &= \sum_{\bfi\in \Lambda_k} \rho_{\bfi}
\ce^{\sigma^{\bfi}} (f\circ\psi_{\bfi},g\circ\psi_{\bfi})
\quad \text{for } f,g \in\mathcal{F}^k.
\end{aligned}
\end{equation}
The functions in $ \mathcal{F}^k $ should be regarded as continuous functions on the disjoint union 
$\mbox{\LARGE $ \sqcup$}_{ \boldsymbol{i} \in \Lambda_k}K_{\boldsymbol{i}} $ together with its 
natural direct sum topology.


Define $(\mathcal{E}^k_D,\mathcal{F}^k_D)$ by
\begin{equation} \label{DDkform}
\begin{aligned}
 \mathcal{F}^k_D &= \big\{ f\in\cf^{k}  \   \big| \  
\forall \boldsymbol{i} \in \Lambda_k \  
 f_{\bfi} |_{\widetilde{V}_0} = 0, \text{ where $ f_{\bfi}$ is as in } \eqref{Dkform} \big\}, \\
%
%
%
 \mathcal{E}^k_D(f,g) &=   \mathcal{E}^k (f,g)
\quad \text{for } f,g \in\mathcal{F}^k_D. 
%
\end{aligned} 
\end{equation} 
Thus $ \mathcal{F} ^k_D $ is the restriction of $\mathcal{F}^k$ (and of $\mathcal{F}$) to those functions 
which are zero on $ \widetilde{V}_k $, and $ \mathcal{E} ^k_D $ is the restricted energy functional.

It is straightforward to see that 
\begin{equation} \label{}
 \mathcal{F}^k_D \subset \cf _D \subset \cf \subset \mathcal{F}^k, \quad  
  \mathcal{E}^k_D \subset \ce _D \subset \ce \subset \mathcal{E}^k. 
\end{equation} 
That is, $ \mathcal{E} $ is just the restriction to $ \mathcal{F} $ of the functional $ \mathcal{E} ^ k $ and 
similarly for the other cases.
 
Note that $(\mathcal{E}^k,\mathcal{F}^k)$ and
$(\mathcal{E}^k_D,\mathcal{F}^k_D)$ are local regular
Dirichlet forms on the spaces 
$L^2(\mbox{\LARGE $ \sqcup$}_{ \boldsymbol{i} \in \Lambda_k}K_{\boldsymbol{i}},\mu)$ 
and  $L^2(K\setminus \widetilde{V}_k,\mu)$ respectively, with discrete spectra and bounded reproducing Dirichlet kernels, see~\cite{Kig1}.

%
%

\bigskip
Analogously to \eqref{dfne} and \eqref{dfde} we define the notion that \emph{$ \lambda $ is an 
$(\mathcal{E}^k,\mathcal{F}^k)$, respectively
$(\mathcal{E}^k_D,\mathcal{F}^k_D)$, eigenvalue with eigenfunction $ u $}. The corresponding counting functions are 
\begin{equation} \label{}
\begin{aligned}
\mathcal{N}^k_N (s) &= \max\{i: \lambda_i  \leq s,\ \lambda_i \mbox{ is an $ (\mathcal{E}^k,\mathcal{F}^k) $  eigenvalue}\},\\
\mathcal{N}_D^k(s) &= \max\{i: \lambda_i  \leq s,\ \lambda_i \mbox{ is an $ (\mathcal{E}^k_D,\mathcal{F}^k_D)$  eigenvalue}\}.
\end{aligned} 
\end{equation} 

In order to compare the various counting functions, first note that if $ \Lambda = \Lambda_k $ then the 
decomposition  (\ref{eq:formdecomp}) with $ \mathcal{E} $ replaced by $ \mathcal{E} ^k $, and the decomposition 
(\ref{inndec}), both generalise  to   functions $ f,g \in  \mathcal{F}^k $. The key observation now is that 
if $\lambda$ is a (Neumann) $ (\mathcal{E}^k,\mathcal{F}^k) $ eigenvalue with
eigenfunction $u$, then  we have for all $v\in\mathcal{F}^k $ that
\begin{equation} \label{}
\sum_{\bfi\in \Lambda_k} \rho_\bfi \ce^{\sigma^{\bfi}}(u\circ
\psi_{\bfi},v\circ\psi_{\bfi}) = \mathcal{E}^k(u,v) =
\lambda(u,v)_{\mu} = \lambda \sum_{\bfi\in \Lambda_k }
\mu_{\bfi} (u\circ\psi_{\bfi},v\circ\psi_{\bfi})_{\mu^{\sigma^{\bfi}}} .
\end{equation} 

If we take $v$ to be a function supported on a complex with address
$\bfi\in \Lambda_k $, we see that
\begin{equation} \label{} 
 \ce^{\sigma^{\bfi}}(u\circ\psi_{\bfi},v\circ\psi_{\bfi}) =
t_{\bfi} \lambda  (u\circ\psi_{\bfi},v\circ\psi_{\bfi})_{\mu^{\sigma^{\bfi}}}, 
\end{equation} 
since $ t_{\bfi} = \rho_{\bfi}^{-1}\mu_{\bfi}$. Thus $t_{\bfi} \lambda$ is an eigenvalue of
$(\ce^{\sigma^{\bfi}},\cf^{\sigma^{\bfi}})$ with eigenfunction
$u_{ \boldsymbol{i} }=u\circ\psi_{\bfi}$. Conversely, from $u_{\boldsymbol{i} }$ we can construct (Neumann) eigenfunctions and
eigenvalues for $(\mathcal{E}^k,\mathcal{F}^k)$, since
\begin{equation} \label{buef} 
\widetilde{u}_{ \boldsymbol{i} } (x) := \left\{ \begin {array}{ll} u_{ \boldsymbol{i} }(x)  & x\in K\cap
  \mbox{int}\, \psi_{\bfi}(K)  \\ 0  & x\notin K\cap
  \mbox{int}\, \psi_{\bfi}(K)  \end{array}
\right. \text{ is an eigenfunction with eigenvalue $\lambda$.}
\end{equation} 
 Hence  \begin{equation} \label{}
  \mathcal{N}^k_N (s) = \sum_{\bfi\in \Lambda_k} \mathcal{N}_N ^{\sigma^{\bfi} }(t_{\bfi} s ),  
  \quad  \mathcal{N} ^k_D(s) = \sum_{\bfi\in \Lambda_k} \mathcal{N} ^{\sigma^{\bfi} }_D(t_{\bfi} s ),
\end{equation}
 with the argument in the Dirichlet case being similar to that for the Neumann case. 
  
\begin{lem}\label{lem:scale}
The following relationships hold for all $s>0$
\begin{equation}  \label{eq:dnb} 
\begin{gathered}
\sum_{\bfi\in \Lambda_k} \mathcal{N} ^{\sigma^{\bfi}}_D(t_{\bfi}s)  \leq \mathcal{N} _D(s) \leq \mathcal{N}_N (s) 
  \leq   \sum_{\bfi\in \Lambda_k}  \mathcal{N}_N ^{\sigma^{\bfi}}(t_{\bfi}s)   ,\\
\mathcal{N} _D (s) \leq \mathcal{N}_N  (s) \leq \mathcal{N}_D  (s)+ d+1. 
\end{gathered} 
\end{equation} 
\end{lem}

\begin{proof} The proofs are a consequence of Dirichlet-Neumann
bracketing and are straightforward extensions of those found 
in \cite{Kig}~Section~4.1 for the p.c.f.\ fractal case. 
The upper bound on the difference
in the Neumann and Dirichlet counting functions is given by the number
of vertices of $V_0$, which is $d+1$ in our setting.
\end{proof} 

\subsection{Eigenvalue Estimates}  
As in the previous section, fix    weights $ w_i^F$    and  the corresponding measure  $ \mu $.
Let the random variable $\lambda_1^D$ denote the first   Dirichlet eigenvalue.

%
%
%
%
%
%
%
 
\begin{lem}\label{lem:evalests} If $ C $ is the upper bound on the diameter of $ K $ in the resistance metric given in Corollary~\ref{cor:diambd} then 
for $ n \geq 2 $, 
\begin{equation} \label{fevb} 
\begin{aligned}   
C^{-1} \leq \lambda _1^D &\leq \frac{d(d+1)\rho^n_{\sup}}{\mu(K\setminus K_{b,n})}, \\
\quad  V=1\, \Longrightarrow\,   \lambda _1^D &\leq   \frac{ (d+1)^2\rho^2_{\sup} w_{\sup}^2} { w_{\inf}^2},
\end{aligned} 
\end{equation} 
where $ K_{b,n} $ is the union of the $d+1$ boundary $ n $-complexes attached to the $d+1$ boundary vertices in $ V_0 $. 
\end{lem}

\begin{proof} 
Since the Dirichlet form is a resistance form we have   for $f \in \cf_D$ that
\[ 
|f(x)-f(y)|^2 \leq R (x,y) \ce(f,f).
\]
Since $\mu$ is a probability measure and $f \in \cf_D$,   using Corollary~\ref{cor:diambd}  and the definition of $ R(x,y) $ in \eqref{dfrm}, it follows that 
\[
\|f\|_2^2  \leq  \sup_{x\in K^{\omega}} |f(x)|^2  
  \leq  \sup_{x,y\in K } |f(x)-f(y)|^2  \leq   \sup_{x,y\in K }  R (x,y)\,  \ce (f,f) \leq C\,  \ce (f,f).
\]
Hence by Rayleigh-Ritz, 
\[  
 \lambda_1^D = \inf_{f \in \cf_D} \frac{\ce(f,f)}{ \|f\|_2^2} \geq C^{-1}.
\]

\medskip Next let $ f(x) = 0 $ for $ x \in V_0 $, $ f(x) = 1 $ for $ x \in V_n \setminus V_0 $, 
and harmonically interpolate. Then
\begin{align*} 
\mathcal{E} (f,f) &= \mathcal{E} _n(f,f)  \leq d(d+1) \rho^n_{\sup}, \\
 \int_K f(x)^2 \mu(dx)   & \geq \mu(K\setminus K_{b,n}).
\end{align*} 
Again by Rayleigh-Ritz, this gives the upper bound.

\medskip If $ V=1 $ note that with $ n=2 $ there are at least $ d(d+1) $ interior cells as well as $ d+1 $ boundary cells.  Since all cells have the same type,  from \eqref{mibds} in Lemma~\ref{hineq}  with $ \zeta = w_{\inf}/w_{\sup} $, 
\begin{align*} 
\mu(K \setminus K_{b,2}) &\geq d \zeta^2 \mu(K_{b,2}) = d \zeta^2 \big( 1- \mu(K \setminus K_{b,2}) \big)\\
\therefore\  \mu(K \setminus K_{b,2}) &\geq \frac{d\zeta^2}{1 + d \zeta^2} \geq \frac{d \zeta^2}{d+1}.
\end{align*} 
This now gives the result for $ V=1 $.
\end{proof}

In order to obtain $P_V$-almost sure results we need to estimate the tail of the bottom eigenvalue random variable.
Note that this result is only relevant in the case where $V>1$ as if $ V=1 $ then $ \lambda _1^D $ is bounded 
above by~\eqref{fevb}. 
\begin{lem}\label{critest}
There exist constants $A>0$, $ \beta > 0 $  and $  \gamma >0$, such that
\begin{equation} \label{}
 P_V(\lambda_1^D > x ) \leq   A\exp(- \beta x^{\gamma}).
\end{equation} 
\end{lem}

\begin{proof}  Let $ n = n^\omega $ be the first level such that the following is true: if $ v $ is the type of a boundary $ n $-complex of maximum mass, then at least one interior (i.e.\ non-boundary) $ n $-complex is also of type $ v $. 
 
For any such $ n $ it follows from \eqref{mibds} that
\[
\text{if }  \zeta:= \frac{w_{\inf}}{w_{\sup}}\quad \text{then }  \mu(K \setminus K_{b,n} ) > \frac{\zeta^n}{d+1} \mu(K_{b,n} ) = \frac{\zeta^n }{d+1}  \big( 1 - \mu (K \setminus K_{b,n}) \big).
\]
Hence 
\[
\mu(K \setminus K_{b,n} ) > \frac{1}{1+(d+1)\zeta^{-n}} > \frac{\zeta^n}{d+2}.
\]
From Lemma~\ref{lem:evalests} it follows that 
\[
\lambda _1^D < d(d+1)(d+2) \rho_{\sup}^n   \zeta^{-n} =  d(d+1)(d+2) \xi^{ n} , \quad \text{where } 
\xi := \frac{ \rho_{\sup}w_{ \sup} }{ w_{\inf} }.
\]

Hence 
\[
P_V\big( \lambda _1 ^D > x \big) \leq P_V\big(d(d+1)(d+2) \xi^n > x \big) = P_V\left( n > \frac{ \log \frac{x}{d(d+1)(d+2)} }{\log \xi }\right).
\]
Since there are $ d+1 $ boundary $ n $-complexes and at least $ (d+1)^n -(d+1) $ interior $ n $-complexes, and since the type
of each $ n $-complex is selected independently of each other $ n $-complex,  
\[
 P_V(n>y) \leq p ^{(d+1)^{y-1}-(d+1)} = A p^{(d+1)^{y-1} }  = A \exp\Big( \frac{\log p}{d+1}  \exp(y\log (d+1)) \Big)
 \]
where $ p $ is the probability that a particular complex is not of type $ v  $ and $A=p^{-d-1}$.
Hence setting $ y = \log(x/(d(d+1)(d+2)))/ \log \xi $, 
 \[
P_V\big( \lambda _1 ^D > x \big) \leq A \exp \left(\frac{\log p}{d+1}  \left(\frac{x}{d(d+1)(d+2)}\right)^{\log (d+1)/\log \xi} \right) 
=  A \exp ( - \beta x^ \gamma )  ,
\]
where $ \gamma = \log (d+1) / \log \xi $ and $ \beta = \frac{1}{d+1} \log (1/p) / (d(d+1)(d+2))^ \gamma  $.
\end{proof} 

Define
\begin{equation} \label{maxfe}
 \widehat{\lambda_1^k}  = \max \big\{ \lambda_1^{\sigma^{ \boldsymbol{i} } ,D} :   \bfi\in\Lambda_k \big\},
\end{equation} 
where $ \lambda_1^{\sigma^{ \boldsymbol{i} } ,D} $ is the first Dirichlet eigenvalue of the Dirichlet form
$ ( \mathcal{E} ^{ \sigma ^{\boldsymbol{i} }}, \mathcal{F} ) = ( \mathcal{E} ^{ \sigma ^{ \boldsymbol{i}} \omega  }, 
\mathcal{F}^ \omega ) $ with respect to the measure $ \mu ^{ \sigma ^{ \boldsymbol{i} }}  
= \mu ^ { \sigma ^{ \boldsymbol{i} }\omega  } $. 
Note that $  \widehat{\lambda_1^0} =  \lambda_1^D $.

If $V=1$ by~\eqref{fevb} we have $\widehat{\lambda_1^k} \leq (d+1)^2\rho^2_{\sup} w_{\sup}^2 / w_{\inf}^2$ for all $k$.

\begin{lem}\label{cortail} If $V>1$, then with $ \beta $ and $ \gamma $ as in Lemma~\ref{critest}, we have $ P_V $ a.s.\ that 
\begin{equation} \label{uefe} 
\widehat{\lambda_1^k}   \preccurlyeq  \beta^{-1/ \gamma }  (\log k)^{1/ \gamma } .
\end{equation} 
\end{lem}

\begin{proof} In order to apply the growth estimate in the previous lemma and use Lemma~\ref{lem:geomrvs} we 
use two additional properties: 
\begin{enumerate}
\item The number of distinct subtrees, and hence eigenvalues, corresponding to each level of $ T $ is uniformly bounded (by $ V $);
\item The maximum level  corresponding to nodes in $ \Lambda_k $ is asymptotically bounded by a multiple of $ k $, see Lemma~\ref{lem:spatial}. 
\end{enumerate} 

\smallskip  First consider any sequence of random $ V $-variable IFS trees $ (T_j)_{j\geq 1 } $, not necessarily independent but all with the same distribution $ P=P_V $, see Definition~\ref{dfpv2}.  Let the corresponding    random first eigenvalues be $Y_j  $. 

Then for all $ x \geq 0 $, 
\begin{align}
P(Y_j>x) &\leq A \exp(- \beta x^ \gamma ) = A(e^{ - \beta} )^{ x^ \gamma } \quad \text{by Lemma \ref{critest},} \notag \\
\therefore\ P(Y_j ^ \gamma > x )  &\leq A(e^{- \beta })^x , \notag \\
\therefore\ \max_{1\leq j \leq k }Y_j^\gamma   &\preccurlyeq  \beta^{-1} \log k \quad  P_V \text{ a.s.}  \quad \text{by Lemma~\ref{lem:geomrvs}},\notag \\
\therefore\ \max_{1\leq j \leq k }Y_j  & \preccurlyeq  \beta^{-1/ \gamma }  (\log k)^{1/ \gamma } \label{c44bd} \quad  P_V \text{ a.s.} 
\end{align} 

For any tree $ T = T^ \omega $ there are at most $ V $ non-isomorphic  subtrees rooted at each level.  Let $( Y_j )_{j \geq 0} $ be the sequence of random variables given by the first eigenvalue of $ T $, followed by the first eigenvalues of non-isomorphic IFS subtrees of $ T $   at level one (there are at most $ V $), followed  by the first eigenvalues of non-isomorphic IFS subtrees of $ T $   at level two (again there are at most $ V $), etc.  
If $ Y_j $ corresponds to a subtree rooted at level $ p $ then by construction $ j \leq Vp $. With $ z_k $ as in \eqref{dfzk} it follows that
$ 
  \widehat{\lambda_1^k}  \leq \max_{1\leq j \leq Vz_k} Y_j 
$.
 
Hence $ P_V $ a.s.,
\begin{align*} 
\limsup_{k\to \infty} \frac{  \widehat{\lambda_1^k} }{(\log k )^{1/ \gamma } }
   &\leq \limsup_{k\to \infty} \frac{ \max_{1\leq j \leq Vz_k} Y_j}{\big(\log V z_k\big)^{1/ \gamma }}
      \left(\frac{ \log V z_k } {\log k} \right)^{1/ \gamma } \\
      &\leq \beta ^{-1/ \gamma }   \left( \limsup_{k\to \infty}   \frac{ \log V +\log z_k } {\log k} \right)^{1/ \gamma } 
      \quad \text{from \eqref{c44bd}} \\
      &\leq \beta ^{-1/ \gamma } ,
\end{align*} 
since $ z_k\preccurlyeq c_4 k $ from Lemma~\ref{lem:spatial}(b)  which implies $ \limsup_{k\to \infty} (\log z_k/\log k  ) 
\leq 1$.
\end{proof} 

We now wish to determine the limiting behaviour of the counting function. We first give the following result 
that is true for all $\omega\in\Omega_V'$, which we recall from~\eqref{dfwvd} assumes that $ \omega $ has 
an infinite number of necks.  
 
Recall  $ \eta\ (<1)$ defined in \eqref{dfeta}, and the quantities defined in \eqref{maxfe} and \eqref{dfmk}--\eqref{dfzk}.
  
\begin{lem}\label{firsteest}
There exists a constant $c_1$ such that if $ \omega \in \Omega_V' $ then
\begin{equation} \label{eqfe}   
  \mathcal{N} _D(T_k) \leq c_1 M_k, \quad 
  M_k \leq \mathcal{N} _D\big( \widehat{\lambda_1^k}  T_k  {\eta}^{-y_k} \big)    
\end{equation} 
for all $k\geq 0$. 
\end{lem}

\begin{proof} 
  For the first estimate we have from \eqref{eq:dnb},  \eqref{dfmk} and Lemma~\ref{lem:lingrowth},
\begin{align*} 
\mathcal{N} _D(T_k)   &\leq \sum_{ \boldsymbol{i} \in \Lambda_k} \mathcal{N} _N^{\sigma^{\boldsymbol{i}}} 
( t_{ \boldsymbol{i} } T_ k )  \leq (d+1)M_k + \sum_{ \boldsymbol{i} \in \Lambda_k} 
\mathcal{N} _D^{ \sigma ^ { \boldsymbol{i} }} ( t_{ \boldsymbol{i} } T_ k ) \\
&   \leq (d+1)M_k + c\, T_ k \sum_{ \boldsymbol{i} \in \Lambda_k} t_{ \boldsymbol{i} }     
 \leq c_1  M_k  .   
\end{align*} 

Next note   from   definitions \eqref{dfmk}, \eqref{dfdec}  and \eqref{maxfe} of $ M_k $, 
$ \mathcal{N}_D$  and $ \widehat{\lambda_1^k}$ respectively, from the fact $   \lambda _1^{ \sigma ^{ \boldsymbol{i} },D}  <   \lambda _2^{ \sigma ^{ \boldsymbol{i} },D}  $ for the  equality below, and 
from   \eqref{eq:dnb} for the last inequality 
provided  $ t_{ \boldsymbol{i} }^{-1} \leq c(k) $ for all $ \boldsymbol{i} \in \Lambda_k $,  that
\begin{equation} \label{mknest} 
M_k  =  \sum_{ \boldsymbol{i} \in \Lambda_k} \mathcal{N} _D^{ \sigma ^{ \boldsymbol{i} }} 
                                 \big(   \lambda _1^{ \sigma ^{ \boldsymbol{i} },D}    \big)
 \leq  \sum_{ \boldsymbol{i} \in \Lambda_k} \mathcal{N} _D^{ \sigma ^{ \boldsymbol{i} }}
                               \big( \widehat{ \lambda _1^k}  \big) 
 \leq \mathcal{N} _D\big(\widehat{ \lambda _1^k}\, c(k)  \big). 
\end{equation} 
But
$ 
 t_{ \boldsymbol{i} }^{-1} \leq \eta^{-y_k} e^k \leq  \eta^{-y_k} T_k 
$ 
 from Lemma~\ref{lem:spatial}(c) and the definition~\eqref{dfmk} of $ T_k $. This gives the second estimate.
%
\end{proof} 

%
%

For $V=1$ we can improve this to the same estimate as that obtained in \cite{barham}~Section~7.
\begin{cor}\label{cor:v=1}
For $V=1$ there exist constants $c_1$ and $c_2$ such that for all $k\geq 0$,
\[  M_k \leq \mathcal{N}_D(c_1 T_k) \;\;\text{and} \;\;   \mathcal{N}_D(T_k) \leq c_2 M_k. \]
\end{cor}

\begin{proof}
As $V=1$ we know that $\widehat{\lambda_1^k}$ is bounded above and $y_k=1$ and hence the inequality on the right in 
\eqref{eqfe} reduces to $M_k \leq \mathcal{N}_D(c_1 T_k)$ as required. 
\end{proof}

We next use asymptotic information about the frequency of necks to obtain the following.

\begin{lem} \label{seceest}
For $V>1$ there exist constants $ c_1$ and $   \alpha  $  such that $ P_V $ a.s.\ there is a  $ k_0(\omega) $  for which 
\[
\mathcal{N} _D(T_k) \leq c_1 M_k, \quad 
M_k \leq \mathcal{N} _D(  k^{\alpha}T_k) \text{\quad  if } k>k_0(\omega).
\]
\end{lem}

\begin{proof} This follows from Lemma~\ref{firsteest}, since $  \widehat{\lambda_1^k}   
\preccurlyeq  \beta^{-1/ \gamma }  (\log k)^{1/ \gamma } $  by  \eqref{uefe} and $\eta^{-y_k} 
\preccurlyeq  k^{\beta '}$ by Lemma~\ref{lem:spatial}(c). 
\end{proof} 


\begin{rem}\label{casepl}
Suppose  $w_i^F  = \rho_i^F$ for all $ F $ and $ i $. Then $\mu_{\bfi}=\rho_{\bfi}/\sum_{\bfj\in \Lambda_k} \rho_{\bfj}$ 
for $\bfi\in\Lambda_k$. This measure corresponds to the random walk with constant expected waiting time at each node, see the sentence following~\eqref{dfti}. 
Note that  $t_{\bfi} = \mu_{ \boldsymbol{i} } \rho_{\boldsymbol{i} }^{-1}  = 1/\sum_{\bfj\in \Lambda_k} \rho_{\bfj}$  
is independent of $ \boldsymbol{i} \in \Lambda_k $, and that
$ T_k  =\sum_{\bfj\in \Lambda_k} \rho_{\bfj}$. 

For fixed $ k $ there is no spatial variability of the $ t_ {\boldsymbol{i}} $
with $  \bfi\in\Lambda_k$,  and so we can take $ c(k) = T_k $ in \eqref{mknest}, which improves~\eqref{eqfe} and gives
\begin{equation} \label{}
\mathcal{N} _D(T_k) \leq c M_k    \mbox{ and }   M_k \leq \mathcal{N} _D(\widehat{\lambda_1^k}\, T_k) .
\end{equation}
Using \eqref{uefe} we can improve the second estimate to have that $P_V$ almost surely  there is a 
$k_0(\omega)$ such that
\[  M_k \leq \mathcal{N} _D\big(2\beta^{-1/\gamma } (\log{k})^{1/\gamma} T_k\big), \;\; k>k_0(\omega). \]
\end{rem} 
 
\subsection{Spectral Exponent} \label{secse}  

We again fix weights $ w_i^F$ and let $ \mu $ be the corresponding measure as in Definition~\ref{dfwm}.  


\begin{defn} 
The \emph{pressure function} $\gamma = \gamma ( \beta )  $ where $ \beta \in \mathbb{R} $, and  the constant 
$ \beta_0 $, are defined  by
\begin{equation} \label{gprf}
\gamma ( \beta ) = E_V   \log \sum_{| \boldsymbol{i}  | = n(1) } t_{ \boldsymbol{i} } ^{ \beta  /2 }  , 
\quad \gamma (\beta_0 ) = 0 .
\end{equation} 
(It follows from Lemma~\ref{se1} that $ \beta_0 $ is  unique.)  
\end{defn} 

The pressure function and its zero can be found computationally.  See \cite{BHS2} for similar computations for the 
fractal dimension.

\begin{defn}\label{dfspm} 
The \emph{spectral exponent  $ d_s(\mu) $ for $ \mu $} is  defined by
\begin{equation} \label{} 
\frac{d_s(\mu)}{2} =  \lim_{t\to\infty}  \frac{\log{\mathcal{N}_D (t)}}{\log{t} } .
\end{equation} 
\end{defn} 

We see in Theorem~\ref{thm:Nspecdim} that a.s.\ the  spectral exponent  exists and equals the constant $\beta _0$.  
By Lemma~\ref{lem:scale}  we could replace $ \mathcal{N} _D $ by $ \mathcal{N}_N $. 

\bigskip
Recall the definition of $ \eta $ in \eqref{dfeta} and the estimate for   $ t_{\bfi} $ from ~\eqref{tbd}.


\begin{lem}\label{se1}
The function $ \gamma ( \beta ) $ is finite, strictly  decreasing and Lipschitz, with derivative in the interval 
$ \left[ \log \left(\eta^{1/2} \right)E_V n(1), \log \left(r_{\sup}^{1/2}\right) E_V n(1)\right] $.  
Since $ \gamma (0) > 0 $ there is a unique $ \beta_0 $ such that $ \gamma ( \beta_0  )= 0 $ and moreover 
$ \beta_0 > 0 $.
\end{lem} 

\begin{proof}
If $ \alpha < \beta $ then from \eqref{dfeta} and \eqref{tbd},
\[
 \gamma ( \alpha ) + \frac{ \beta - \alpha } { 2} (\log \eta)E_V n(1) 
\leq \gamma ( \beta ) 
\leq \gamma ( \alpha ) + \frac{ \beta - \alpha } { 2} (\log r_{ \sup } )E_V n(1) .
\]
This gives the Lipschitz estimate.

Since  $ \gamma (0) = E_V \bigl( \log \# \{ \boldsymbol{i} \in T \mid |  \boldsymbol{i}  | = n(1) \} \bigr) $, it follows that $ 0
 < \gamma (0) < \infty $.

The rest of the lemma follows.
\end{proof}

\begin{propn}\label{se2}
 $ P_V $ a.s.\ we have
\begin{equation} \label{dfxg} 
 \lim_{k\to\infty} \frac{1}{k} \log \sum_{|\bfi|= n(k)}  t_{\bfi}^{\beta/2} 
= \gamma(\beta).
\end{equation} 
\end{propn}

%

\begin{proof} 
The idea is that from the definition of a neck, $\log  \sum_{|\bfi|= n(k)}  t_{\bfi}^{\beta/2} $ is the difference of two 
random variables, each of which is the sum of $ k $ iid random variables having  the same distribution as $ \log 
\sum_{|\bfi|= n(1)}  (r_{\bfi} w_{\bfi} )^{\beta/2} $ and  $ (\beta / 2 )   \log \sum_{|\bfi |= n(1)}  w_{\bfi}$ 
respectively.

\bigskip More precisely, suppose 
 $| \boldsymbol{i} | = n(k) $ and in particular is a neck. Then
\[
t_{ \boldsymbol{i} } = r_{ \boldsymbol{i} } \mu _ { \boldsymbol{i} } 
     = \frac{  r_{ \boldsymbol{i} }w_ { \boldsymbol{i} } }{ \sum_{ |\boldsymbol{j} |= n(k) } w_ { \boldsymbol{j} } },
\] 
and so
\begin{equation} \label{spl} 
 \log \sum_{|\bfi|= n(k)}  t_{\bfi}^{\beta/2} =  \log \sum_{|\bfi|= n(k)}  (r_{\bfi}w_{\bfi}) ^{\beta/2} - \frac{ \beta } { 2}  
 \log \sum_{|\bfi|= n(k)}  w_{\bfi} .
\end{equation} 

If we let $s_i^F = (r_i^F w_i^F)^{ \beta / 2 }  $ or  $ s_i^F = w_i^F $, it follows from \eqref{rhoest} and 
\eqref{west} that 
\begin{equation} 
\begin{aligned}
0< s_{\inf}  &:=  \inf\{s^F_i  :      i  \in 1,\dots,N^F,\,  F\in \bff\} ,  \\
s_{\sup}  &:=  \sup\{s^F_i  :  i\in 1,\dots,N^F,\,  F\in \boldsymbol{F} \}   < \infty.
\end{aligned} 
\end{equation} 
and we can apply Lemma~\ref{lem:sumprod}. Thus \eqref{eq:slln} applied to each term on the right hand side 
of \eqref{spl} gives the result.
\end{proof} 

\emph{Subsequently we write $ \mathcal{N} $ for $ \mathcal{N} _D $.}  But note that from the second line in  Lemma~\ref{lem:scale} the main estimates in the rest of the paper also apply immediately to $ \mathcal{N} _N $.

\bigskip
The proof of the following theorem relies on the Dirichlet-Neumann bracketing result in Lemma~\ref{seceest} and the   estimates in Lemma~\ref{lem:spatial}(c).

%
%

\begin{thm}\label{thm:Nspecdim}
The spectral exponent is given by $\beta _0$ in that
\begin{equation} \label{dfspece}
\frac{d_s(\mu)}{2} :=  \lim_{t\to\infty} \frac{\log{\mathcal{N} (t)}}{\log{t}} = \frac{\beta_0}{2}, \quad P_V \text{ a.s.}
\end{equation} 
\end{thm}


\begin{proof}[Proof of Theorem] Define the unit mass measure $\nu_{\beta}$ on $\partial T$ by setting, for any 
$\beta$ and for $| \boldsymbol{i} | = n(k)$,
\[ 
\nu_{\beta}[\bfi] = \frac{t_{\bfi}^{\beta/2}}{\sum_{|\bfi |=n(k) }t_{\bfi}^{\beta/2}}.
 \]
 It is straightforward to check  that $ \nu_ \beta  $ is just the unit mass measure with weights $ (r_iw_i)^{ \beta /2 } $ 
 as in Definition~\ref{dfwm}.

If $\gamma(\beta)<0$ or equivalently $ \beta > \beta_0 $,  then from \eqref{dfxg} for $\epsilon>0$ small enough 
we have $  P_V$   a.s.\    that there is a constant $k_0$ such that
\[  \nu_{\beta}[\bfi] \geq t_{\bfi}^{\beta/2} e^{-k(\gamma(\beta)+\epsilon)} \geq c t_{\bfi}^{\beta/2} \text{\quad if } 
k \geq k_0 .  \]
As $\Lambda_k$ is a cut set, by using the lower estimate above  we have from Lemma~\ref{lem:spatial}(c) that $ P_V $ a.s.\ if $ k \geq k_0$ then
\[
1  =  \sum_{\bfi\in \Lambda_k} \nu_{\beta}[\bfi] 
 \geq   \sum_{\bfi\in\Lambda_k} c t_{\bfi}^{\beta/2} 
 \succcurlyeq  c M_k k^{-\beta\beta'/2} e ^{-k\beta /2}, \quad  \text{for some } \beta' > 0 .
\]
Thus 
\begin{equation} \label{eq:mkupper}
M_k \preccurlyeq c k^{ \beta  \beta'/2} e^{k\beta/2}, \quad P_V \text{ a.s.}   
\end{equation}

Suppose $ t > 1 $ and let $k$ be such that $e^{k-1} <t \leq e^{k}$. Then $ t \leq T_k $ by Lemma~\ref{lem:spatial}(c)  and so
\[ 
\frac{\log{\mathcal{N} (t)}}{\log{t}} \leq \frac{\log{\mathcal{N} (T_k)}}{\log{t}} \leq \frac{\log(c M_k)}{k-1}
\preccurlyeq \frac{ \beta }{ 2 }, \quad P_V \text{ a.s.}   , 
 \]
where  the second inequality is from the first estimate in Lemma~\ref{seceest} and the third inequality is from \eqref{eq:mkupper}.

As this holds for all $\beta> \beta_0$ we have
\begin{equation} \label{lim1} 
 \frac{\log{\mathcal{N} (s)}}{\log{s}} \preccurlyeq \frac{\beta_0}{2}, \quad P_V \text{ a.s.}  
\end{equation}

\medskip Similarly we have an asymptotic  lower bound. 
For this choose  $ \beta < \beta_0 $,  or equivalently such that $\gamma(\beta)>0$.  Then   for
small enough $\epsilon>0$ we have $ P_V $ a.s.\ that for some $ k_0 = k_0(\omega) $  
\[ 
\nu_{\beta}(\bfi) \leq c  t_{\bfi}^{\beta/2} \text{\quad if } k \geq k_0,
\]
and hence from Lemma~\ref{lem:spatial}(c), that $ P_V $ a.s.\  then 
\[
1= \sum_{\bfi\in \Lambda_k} \nu_{\beta}(\bfi) 
\leq  \sum_{\bfi\in\Lambda_k} c t_{\bfi}^{\beta/2} 
\leq  c M_k e^{-k\beta/2}    \text{\quad if } k \geq k_0.
\]
Thus $ P_V $ a.s.\
\begin{equation}  \label{eq:mklower}
M_k \geq c e^{k\beta/2}  \text{\quad if } k \geq k_0.
\end{equation}

From the second estimate in   Lemma~\ref{seceest} and using \eqref{eq:mklower},
\begin{equation} \label{parest}
\frac{ \log \mathcal{N} (k^{\alpha } T_k) } { k } \geq \frac{ \log M_k } { k } \succcurlyeq \frac{ \beta }{2} \quad 
P_V \text{ a.s.} 
\end{equation}

Again choosing $k$ such that $e^{k-1} \leq t < e^k$, we have from Lemma~\ref{lem:spatial}(c) that
for some $ \alpha '$, 
\[ k^{\alpha }T_k \preccurlyeq k^{\alpha ' }  e^k \leq  e(1+ \log t)^{\alpha ' }t, \quad P_V \text{ a.s.}  \]
Hence
\[
\liminf_{k\to \infty} \frac{ \log \mathcal{N} \big( k^{\alpha  } T_k \big)} { k }\leq \liminf_{t\to \infty}  \frac{ \log \mathcal{N} \big( e ( 1+\log t)^{ \alpha '} t \big) } { \log t }, \quad P_V \text{ a.s.} 
\]
Setting $ y = y(t) =   e ( 1+\log t)^{ \alpha '} t $, 
since $ \lim_{t\to \infty }  \log y(t)/\log t =1$  and   $ y(t) \to \infty $ as $ t \to \infty $, it follows  
\[
\liminf_{k\to \infty} \frac{\log  \mathcal{N} \big( k^{ \alpha  } T_k \big)} { k } \leq \liminf_{t\to \infty}  \frac{ \log \mathcal{N}(t) } { \log t }, \quad P_V \text{ a.s.} 
\]

Combining this with \eqref{parest}, since $ \beta < \beta_0 $ is arbitrary, implies 
\begin{equation} \label{lim2} 
 \frac{\log{\mathcal{N} (s)}}{\log{s}} \succcurlyeq \frac{\beta_0}{2}, \quad P_V \text{ a.s.}  
\end{equation} 

\medskip The required result follows from \eqref{lim1} and \eqref{lim2}.
\end{proof} 
\subsection{Spectral Dimension}   

\begin{defn} \label{dffm}The \emph{flat measure}   with respect to the resistance metric   is the unit mass measure $ \nu $ with weights   $w_i^F=(r_i^F)^{ d_f^r}$,  where $ d_f^r $ is the Hausdorff dimension in the resistance metric (see Definition~\ref{dfwm}).

The \emph{spectral dimension} $d_s$ is the spectral exponent for the flat measure.  
\end{defn} 
Further justification for the definition of $ d_s $ is given in Theorem~\ref{spmax}.

\medskip
Recall from Theorem~\ref{rdzp} that  $ d_f^r $ is uniquely characterised by 
\begin{equation} \label{dfxa}
E_V \log\sum_{|\bfi|=n(1)} r_{\bfi}^{ d_f^r}=0,
\end{equation} 
As a consequence, the following theorem establishes the analogue of Conjecture~4.6 in \cite{Kig-3}
for $ V $-variable fractals.
%

\begin{thm}\label{spfm}
The spectral exponent for the flat measure $ \nu $  is given $P_V$ a.s.\  by
\begin{equation} \label{per}
\frac{ d_s(\nu)}{2} = \frac{d_f^r}{d_f^r+1}.
\end{equation} 
\end{thm}

\begin{proof}  From Definition~\ref{dffm},
 \eqref{gprf}, \eqref{dfti} and  \eqref{dfmui}, if $ | \boldsymbol{i} | = n(\ell) $ is a neck then
\begin{equation} \label{mttc}
t_{\bfi} := r_{ \boldsymbol{i} } \nu_ { \boldsymbol{i} } =     \frac{r_{ \boldsymbol{i} } w_{\bfi} }{\sum_{|\bfj|= n(\ell)} w_{\bfj} } 
=        \frac{r_{\bfi}^{1+d_f^r}}{\sum_{|\bfj|= n(\ell)} r_{\bfj}^{ d_f^r}} .
\end{equation} 
Hence the spectral exponent $  d_s(\nu) $ is uniquely characterised  by
 \begin{equation} \label{dfsd} 
0 = \gamma(d_s(\nu) )  :=  E_V \log \sum_{|\bfi|=n(1) } t_{\bfi}^{d_s(\nu) /2} = 
 E_V  \log \sum_{|\bfi|=n(1) } \left(  \frac{r_{\bfi}^{1+d_f^r}}{\sum_{|\bfj|= n(1)} r_{\bfj}^{ d_f^r}}  \right) ^{d_s(\nu) /2}.
\end{equation} 
 
Using   \eqref{dfxa}, 
\[
0  = E_V \log \sum_{ | \boldsymbol{i} | = n(1) } r_{ \boldsymbol{i} } ^{ (1+ d_f^r) d_s(\nu) /2 }- 
\frac{ d_s(\nu) } { 2 }\, E_V \log  \sum_{ | \boldsymbol{j} | = 1 } r_{ \boldsymbol{j} }  ^{ d_f^r}  
 =  E_V \log \sum_{ | \boldsymbol{i} | = n(1) } r_{ \boldsymbol{i} } ^{ (1+ d_f^r) d_s(\nu) /2}.
\]
Using \eqref{dfxa} again and the  uniqueness of $ d_f^r$, it follows that  $d_f^r = (1+d_f^r)d_s(\nu) /2$, which gives~\eqref{per}.
\end{proof}

\bigskip
We next show that the spectral dimension maximises the spectral exponent $ d_s (\mu) $ over \emph{all} 
measures $ \mu $ defined from a set  of weights $ w_i^F $ as in  Section~\ref{secwm}.  A related  result for   
deterministic fractals is established in Theorem A2 of \cite{kiglap} using Lagrange multipliers.  Here we need  
a different argument, but this also  establishes uniqueness of the $ w_i^F $ and hence of $ \mu$.

The proof is partly motivated by \cite{HR}, in particular Section 4 and the discussion following Corollary 2.7.
We first need the following general inequality.

\begin{propn} \label{prpiiq}Suppose $ \{p_1, \dots, p_N\} $ and $ \{q_1, \dots, q_N\} $ are sets of positive real 
valued random variables, each with the same random cardinality $ N $, on a probability space $(\Omega, \bp)$. 
Suppose $ \mathbb{E} \log \sum_{k=1}^N p_k = 0 $ and that the constant $ \gamma $ satisfies $ 0< \gamma < 1 $.
Then 
\begin{equation} \label{iiq}
\mathbb{E} \log \sum_{k=1}^N p_k q_k ^ \gamma \leq \mathbb{E} \log \left( \sum_{k=1}^N p_k q_k \right)^ \gamma ,
\end{equation}
with equality iff $ q_1 = \dots = q_N $ a.s. 
\end{propn} 

\begin{proof} 
For any $N$, a suitable version of H\"older's inequality for sequences yields
\begin{equation} \label{iiq2} 
 \sum_{k=1}^N p_k q_k ^ \gamma \leq \left(  \sum_{k=1}^N p_k \right)^{ 1 - \gamma } \left(  \sum_{k=1}^N p_k q_k \right)^ \gamma.
\end{equation} 
 
Taking logs and expectations, and using the assumption on the random sets $ \{ p_1, \dots, p_N \} $ gives 
\begin{equation} \label{iiq3} 
 E_V \log  \sum_{k=1}^N p_k q_k ^ \gamma   \leq ( 1 - \gamma ) \mathbb{E} \log \sum_{k=1}^N p_k 
                                                                + \mathbb{E} \log \left(  \sum_{k=1}^N p_k q_k \right)^ \gamma   
                                                                 =  \mathbb{E} \log \left(  \sum_{k=1}^N p_k q_k \right)^ \gamma.
\end{equation} 
This gives \eqref{iiq}.

If $   q_1 = \dots =  q_N = c $ a.s.\ where $ c $ is a random variable, then equality holds in \eqref{iiq} since both sides equal $ \mathbb{E} \log c^{\gamma} $. 

If it is not the case that $   q_1 = \dots =  q_N  $ a.s.\ then strict inequality holds in \eqref{iiq2} with positive probability and hence strict inequality holds in~\eqref{iiq3}. 
 \end{proof}

\begin{thm} \label{spmax}
 The spectral dimension  $ d_ s $ is the maximum spectral exponent $ d_ s (\mu) $ over all    measures $ \mu$ defined from weights $ w_i^F $.  Equality holds if and only if for some constant $ c $,  $ w_i^F = c\, (r_i^F)^{d_f^r} $ a.s., in which case the corresponding measure $ \mu $ is the flat measure with respect to the resistance metric.
\end{thm}  

\begin{proof}
For $  | \boldsymbol{i} | = n(1) $ let $ p_{ \boldsymbol{i} } = r_{ \boldsymbol{i} } ^ { d_f^r } $, so that 
$ E_V \log \sum_{ | \boldsymbol{i} | = n(1) }p_{ \boldsymbol{i} } =0$. 

Suppose $ w = \{ w_j^F \mid f \in \boldsymbol{F}, \ 1\leq j \leq N^F  \} $ is a set of weights and   consider the corresponding 
$ w_{ \boldsymbol{i} } $.  Let $ q_{ \boldsymbol{i} } = w _ { \boldsymbol{i} } / r_{ \boldsymbol{i} } ^{ d_ f^r} $.

Then from \eqref{iiq},
\[
 E_V\log \sum_{ |\boldsymbol{i} | = n(1) } \Big(  r_{ \boldsymbol{i} } ^{ d_ f^r}  \Big) ^ { 1 - \gamma } w _ { \boldsymbol{i} } ^ \gamma
        \  \leq\  E_V \log \bigg( \sum_{ | \boldsymbol{i} | = n(1) } w_{ \boldsymbol{i} } \bigg)^ \gamma .
  \]
  
Choosing $ \gamma $ so that the powers of $ r_{ \boldsymbol{i} } $ and $ w_{ \boldsymbol{i} } $ are equal, gives 
$ \gamma = d_f^r/(d_f^r + 1) $, i.e.\ $ \gamma = d_s/2 $.  Hence 
\[  
 E_V\log \sum_{ | \boldsymbol{i} | = n(1) } t_{ \boldsymbol{i} }   ^{ d_s/2}   = 
 E_V \log \frac{ \sum_{ | \boldsymbol{i} | = n(1) } ( r_{ \boldsymbol{i} } w_{ \boldsymbol{i} } ) ^{ d_s/2} }
             { \left( \sum_{ | \boldsymbol{i} | = n(1) } w_{ \boldsymbol{i} } \right) ^{ d_s/2} } \leq 0 .
\]
Moreover,  by Proposition~\ref{prpiiq} equality holds  if and only if a.s.\ it is the case that $ w_{ \boldsymbol{i} } / 
r_{ \boldsymbol{i} } ^{ d_f^r} $ is independent of $ \boldsymbol{i} $ for $  | \boldsymbol{i} | = n( 1 ) $.    
Clearly, this is true iff $ w_j^F = c ( r_j^F)^{ d_f^r } $ a.s.\ for some constant $ c $. 

From the definition \eqref{gprf} of $ d_s(\mu) $,  we have $ E_V \log   \sum_{ | \boldsymbol{i} | = n(1) } t_{ \boldsymbol{i} }^{ d_s(\mu)/2}
  =  0 $.
From Lemma~\ref{se1}  and the previous inequality, it follows that $ d_s(\mu) \leq d_s $, and equality holds iff 
$ w_j^F =   c ( r_j^F)^{ d_f } $ a.s.\ for some constant $ c $ 
 \end{proof} 

We next give a sharpening of Theorem~\ref{thm:Nspecdim} in the case of the flat measure with respect to the resistance 
metric.  This shows that for this measure, for all $V>1$, we have the same fluctuations as observed in the version of the 
$V=1$ case treated in \cite{barham}. For this, let
\begin{eqnarray}
\Phi(s) &=&  \sqrt{s\log{\log{s}}}, \nonumber \\
\phi(t) &=& \exp\big(\Phi(\log{t})\big) = \exp\big(  \sqrt{\log{t}\log{\log{\log{t}}}}\big). \label{phidef} 
\end{eqnarray}  


\begin{thm}\label{imprest}
Suppose $\mu$ is the flat measure in the resistance metric.
Then there exist positive (non-random) constants $c_1,c_2,c_3,c_4$, and  there exists a positive finite random variable $ c_0 = c_0 (\omega) $, such that  if  $ t \geq c_0 $ then
\[ c_1 \phi(t)^{-c_2} \leq \frac{\mathcal{N} (t)}{t^{d_s/2}} \leq c_3 \phi(t)^{c_4} \quad P_V \text{ a.s.} \]
\end{thm}

\begin{proof} 

Consider the unit mass measure $\nu_ \beta $ constructed in  the proof of Theorem~\ref{thm:Nspecdim}, where now 
$\beta = d_s$ is the spectral dimension as in~\eqref{dfsd}. 

In the following the constant $ c $ may change from line to line, and even from one inequality to the next.

If $ |\boldsymbol{i} |$ is a neck and $ | \boldsymbol{i}  | = n(\ell) $ then from \eqref{per} and \eqref{mttc}, 
\begin{equation} \label{4121} 
\nu_{d_s}[\bfi]  =  \frac{t_{\bfi}^{d_s/2}}{\sum_{|\bfj | = {n(\ell)}} t_{\bfj}^{d_s/2}}  
 =  \frac{t_{\bfi}^{d_s/2}}{  \left(  \sum_{|\bfj | = {n(\ell)}} r _{\bfj}^{d_f^r} \right)^{1/(1+d_f^r)} }.
\end{equation} 
 
Using the law of the iterated logarithm,
as in Theorem~\ref{thm:lil} and from the decomposition~\eqref{et}, $ P_V $ a.s.\ there exists a   constant $ c $ such that,  
for   $ \ell $ sufficiently large,
\begin{equation} \label{4122} 
-c \leq \frac{ \log \sum_{| \boldsymbol{i} | = n (\ell) } r_{ \boldsymbol{i} }^{d_f^r}}{ \Phi(\ell) } \leq c, \qquad 
  \text{i.e.}  \quad e^{-c\Phi(\ell)} \leq \sum_{| \boldsymbol{i} |=  n (\ell) } r_{ \boldsymbol{i} }^{d_f^r} \leq e^{ c\Phi(\ell)}.
\end{equation} 

 Since $ \nu_{d_s} $ is a unit measure and $ \Lambda_k $ is a cut set, it follows from \eqref{4121} and \eqref{4122} 
 by summing over $ \boldsymbol{i} \in \Lambda_k $ that, for $ k $ sufficiently large,
\begin{equation} \label{88in} 
\sum_{ \boldsymbol{i} \in \Lambda_k} t_{ \boldsymbol{i} } ^{ d_s/2} e^{ -c\Phi(\ell( \boldsymbol{i} )) }\leq 1 
\leq  \sum_{ \boldsymbol{i} \in \Lambda_k} t_{ \boldsymbol{i} } ^{ d_s/2} e^{  c\Phi(\ell( \boldsymbol{i} ))},
\end{equation} 
where $ \ell( \boldsymbol{i} ) $ is defined in \eqref{dfl}.    But from Lemma~\ref{lem:spatial}(c)   and Lemma~\ref{lem:spatial}(a) respectively, the following hold $ P_V $ a.s.\ for  $ \boldsymbol{i} \in \Lambda_k $ and $  k $ sufficiently large:
 \[
 c^{-1}k^{- \beta '}e^{-k} \leq t_{ \boldsymbol{i} } \leq e^{ -k},
 \quad \ell( \boldsymbol{i} ) \leq c_2 k.
 \]
Moreover, $ \Phi(ck) \leq c^*\Phi(k) $ for some $ c^*=c^*(c) $ and all $ k \geq 3 $.  It follows from \eqref{88in} that, for $ k $ sufficiently large, 
\[
c^{-1} M_k  e^{ -k d_s/2} e^{-c\Phi(k)}\leq 1 \leq  c M_k e^{- k d_s/2}e^{c\Phi(k)},
\] 
since   $ k^{-  \beta ' d_s/2} $ can be absorbed into $ e^{-c\Phi(k)} $, with a new $ c $.  That is
\begin{equation} \label{ulmest} 
c^{-1} e^{  k d_s/2}e^{-c\Phi(k)} \leq M_k \leq c e^{  k d_s/2}e^{c\Phi(k)}.
\end{equation} 
 
 Given $ t>0 $ choose $ k $ so $ e^{k-1} < t \leq e^k $. Note also from Lemma~\ref{lem:spatial}(c) that $ e^k \leq T_k \leq c k^{ \beta '} e^k $, for $  k $ sufficiently large.  Then
 from Lemma~\ref{seceest} and \eqref{ulmest},
 \begin{equation} \label{nuest} 
 \mathcal{N} (t) \leq \mathcal{N} (T_k) \leq cM_k \leq ce^{ kd_s/2} e^{ c\Phi(k) } \leq c t ^{d_s/2} \phi(t)^c,
\end{equation} 
 where for the last inequality we note that $ \Phi(k) \leq \Phi(1+\log t) \leq c \Phi(\log t) $.
 
 Similarly, again from Lemma~\ref{seceest} and \eqref{ulmest},
 \[
 \mathcal{N} (k^{ \beta ''} T_k ) \geq M_k \geq c^{-1} e^{ k d_s/2} e^{- c\Phi(k) } \geq c^{-1} t^{ d_s/2} \phi(t)^{-c}.
 \]
 But 
$ 
 k^{ \beta ''} T_k \leq c (\log t)^{ \beta '' + \beta ' } t \leq c^* t 
$ 
 for $ t \geq 2 $ and $ c^* = c^*(c, \beta ', \beta '') $.
 It follows that  $ \mathcal{N} (c^* t ) \geq c^{-1} t^{ d_s/2} \phi(t)^{-c}$ and so
\begin{equation} \label{llmest} 
 \mathcal{N} ( t ) \geq c^{-1} t^{ d_s/2} \phi(t)^{-c}
 \end{equation} 
 if $ \log \log \log t > 0 $, hence if  $t \geq 16  $.
  
 The result follows from \eqref{nuest} and \eqref{llmest}. 
 \end{proof} 


\begin{rem}{\rm 
By using the law of the iterated logarithm in the above we can show that the Weyl limit does not exist in that there
is a constant $c$ such that
\[0< \limsup_{s\to\infty} \frac{\mathcal{N}(s)}{s^{d_s/2}\phi(s)^c}, \;\; P_V \text{ a.s}. \] 
}
\end{rem}


\section{On-Diagonal Heat Kernel Estimates}

\subsection{Overview} The on-diagonal heat kernel is determined for resistance forms by the
volume growth of balls. In \cite{Cro} it is shown how volume
estimates can be translated into heat kernel estimates in the case of non-uniform volume growth. 
We are in the same setting
but will express the bounds in a slightly different way. As we have
scale irregularity these will give rise to larger scale fluctuations than
the fluctuations arising from the spatial irregularity. Note that we will establish bounds for the Neumann heat kernel
and are in a setting where the measure is not volume doubling.

In previous work, in the $V=1$ setting of \cite{barham}, using our notation in (\ref{dfmk}) and (\ref{phidef}), the results 
obtained were that for all realizations there are non-random constants $c_1,c_2 $ such that
\[ c_1 M_k \leq p_{T_k^{-1}}(x,x) \leq c_2 M_k, \;\; \forall x\in K, \;\;k\geq 0, \]
while using a sequence chosen according to $P_1$, there are non-random constants $c_1,c_2,c_3,c_4$ and a
random variable $c_5 \in (0,\infty)$ under $P_1$, such that
\[ c_1 t^{-d_s/2} \phi(1/t)^{-c_3} \leq p_{t}(x,x) \leq c_2 t^{-d_s/2} \phi(1/t)^{c_4} , \;\;\forall x \in K, \;\; 
0<t< c_5, \;\; P_1\;a.s. \]
In the random recursive case ($V=\infty$) with its natural flat measure, as considered in \cite{HamKum}, the fluctuations 
were shown to be smaller in that there are fixed constants $c_1,c_2,a>0$ and a random variable $c_3 \in (0,\infty)$
under $P_{\infty}$ such that  
\[ c_1 t^{-d_s/2}|\log{t}|^{-a} \leq p_t(x,x) \leq c_2 t^{-d_s/2}|\log{t}|^a,\;\;\forall 0<t<c_3, \;\; \mu-a.e. \; x \in K, \;\;
P_{\infty}\;a.s. \]

We will show here that the on-diagonal heat kernel estimates for $V$ variable fractals are determined by
the local environment, see Theorems~\ref{thm:hkub} and ~\ref{thm:hklb}. In the case
of the flat measure in the resistance metric, see Definition~\ref{dffm}, we show in Theorem~\ref{thm:flatfluc} that
the global fluctuations are
of the same order as the $V=1$ case for nested Sierpinski gaskets with uniform measure as described in \cite{barham}. 
In the case of a general class of measures we will see in Theorem~\ref{thm:locspecd} that $\mu$-almost every 
$x\in K$ does not have the same spectral exponent
as the counting function (except when we choose the flat measure) and there will be a multifractal structure to the 
local heat kernel estimates in the same way as observed in \cite{BarKum}, \cite{HamKigKum}.

In order to transfer the fluctuations in the measure to the on-diagonal 
heat kernel we could apply a local Nash inequality, for example 
\cite{Kig2} or use \cite{Cro}. However we use more bare hands arguments adapted from those of 
\cite{barham}, \cite{BarKum} and \cite{HamKigKum} in order to keep the 
scale and spatial fluctuations separate.

Note that in \cite{Cro}, \cite{Kig4} it is shown that, in the case of resistance forms with 
non-uniform volume growth and under assumptions which hold in our setting, there 
exists a heat kernel which is jointly continuous in $(t,x,y) \in (0,\infty) \times K \times K$ for every $\omega\in\Omega$.

\subsection{Upper Bound}

We adapt the scaling argument given in \cite{HamKigKum} Appendix B to this setting. 

Firstly, recall from Theorem~\ref{thmdrf} and the definitions and discussions around \eqref{DDform}, 
\eqref{Dkform}, \eqref{DDkform}, that $(\ce,\cf), (\ce_D,\cf_D), (\ce^{k},\cf^{k})$ and
$(\ce^{k}_D,\cf^{k}_D)$ are local regular
Dirichlet forms on $L^2(K,\mu), L^2(K\backslash V_0,\mu), L^2(\mbox{\LARGE $ \sqcup$}_{ \boldsymbol{i} \in
\Lambda_k}K_{\boldsymbol{i}},\mu), L^2(K\setminus \widetilde{V}_k,\mu)$ respectively.  For $\lambda>0$ let
\[ \ce_{\lambda}(f,g) = \ce(f,g) + \lambda(f,g)_{\mu}, \]
with similar expressions for the other Dirichlet forms. The space $\cf$ equipped with norm 
$\ce_{\lambda}^{1/2}$ is again a reproducing kernel Hilbert space and we write $g_{\lambda}, g_{\lambda}^D, 
g^k_{\lambda}, g_{\lambda}^{k,D}$ for the corresponding reproducing kernels.

We state a scaling property of the Dirichlet form.

\begin{lem}
For all $f,g \in\cf$ we have
\[ \ce_{\lambda}(f,g) = \sum_{\bfi\in\Lambda_k} \rho_{\bfi} \ce_{\lambda t_{\bfi}}^{\sigma^{\bfi}}
(f\circ\psi_{\bfi},g\circ \psi_{\bfi}). \]
\end{lem}

\begin{proof}  This follows by the scaling in (\ref{eq:formdecomp}) and (\ref{inndec}) and the definiton of 
$t_{\bfi}$ in (\ref{dfti}).
\end{proof} 

Let $g_{\lambda}^{D,\sigma^\bfi}$ be the reproducing kernel associated with the Dirichlet form 
$\ce_{D,\lambda}^{\sigma^{\bfi}}$ on $K^{\sigma^\bfi}$ with Dirichlet boundary conditions and 
let $g_{\lambda}^{\sigma^\bfi}$ be the reproducing kernel for the Dirichlet form $\ce_{\lambda}^{\sigma^{\bfi}}$ 
on $K^{\sigma^\bfi}$ with Neumann boundary conditions.

\begin{lem}\label{lem:greenscale}
We have for all $\bfi \in \Lambda_k$ and $x\in K_{\bfi}$, that
\[ g_{\lambda}^{D,\sigma^{\bfi}} (\psi^{-1}_{\bfi}(x),\psi_{\bfi}^{-1}(x)) = \rho_{\bfi} 
g^{k,D}_{\lambda/t_{\bfi}}(x,x). \]
and
\[ g_{\lambda}^{\sigma^{\bfi}}(\psi^{-1}_{\bfi}(x),\psi^{-1}_{\bfi}(x)) = \rho_{\bfi} g^k_{\lambda/t_{\bfi}}(x,x). \]
\end{lem}

\begin{proof} 
We consider $g_{\lambda}^{D,\sigma^{\bfi}}(\psi_{\bfi}^{-1}(x),\psi_{\bfi}^{-1}(x))$, for $x\in K_{\bfi}$, which is
the reproducing kernel for $(\ce_{D,\lambda}^{\sigma^{\bfi}},\cf_D^{\sigma^{\bfi}})$ 
on $L^2(K^{\sigma^{\bfi}},  \mu^{\sigma^{\bfi}})$.
We note that $g_{\lambda}^{D,\sigma^{\bfi}}(\psi_{\bfi}^{-1}(y),\psi_{\bfi}^{-1}(x))=0$ for all 
$y\in K\backslash K_{\bfi}$.
Using this, the reproducing kernel property and the scaling, we have for $x\in K_{\bfi}$,
\begin{eqnarray*}
g_{\lambda t_{\bfi}}^{D,\sigma^{\bfi}}(\psi_{\bfi}^{-1}(x),\psi_{\bfi}^{-1}(x)) &=& \ce^k_{D,\lambda}
(g^{k,D}_{\lambda}(.,x),
g^{D,\sigma^{\bfi}}_{\lambda t_{\bfi}}(\psi_{\bfi}^{-1}(.),\psi_{\bfi}^{-1}(x))) \\
&=& \sum_{\bfj\in\Lambda_k} \rho_{\bfj} \ce_{D,\lambda t_{\bfj}}^{\sigma^{\bfj}}
(g^{k,D}_{\lambda}(\psi_{\bfj}(.),x),
g^{D,\sigma^{\bfi}}_{\lambda t_{\bfi}}(\psi_{\bfi}^{-1}(\psi_{\bfj}(.)),\psi_{\bfi}^{-1}(x))) \\
&=& \rho_{\bfi} \ce_{D,\lambda t_{\bfi}}^{\sigma^{\bfi}}(g^{k,D}_{\lambda}(\psi_{\bfi}(.),x),
g^{D,\sigma^{\bfi}}_{\lambda t_{\bfi}}(.,\psi_{\bfi}^{-1}(x))) \\
&=& \rho_{\bfi} g_{\lambda}^{k,D}(x,x)
\end{eqnarray*}
as required. 

The second equation follows by the same argument.
\end{proof} 

It is straightforward to see that, as
\[ \cf^{k}_D \subset \cf_D \subset \cf \subset \cf^k, \]
and $g_{\lambda}(x,x) = [\inf\{\ce_{\lambda}(f,f):f\in\cf, f(x)\geq 1\}]^{-1}$, 
(with similar expressions for $g^k_{\lambda}, g_{\lambda}^{k,D},g_{\lambda}^D$) we have 
\begin{equation}
g^{k,D}_{\lambda}(x,x) \leq g^{D}_{\lambda}(x,x) \leq g_{\lambda}(x,x) \leq g^k_{\lambda}(x,x), \;\;\forall x\in 
K\backslash \tilde{V}_k. \label{eq:green}
\end{equation}

\begin{lem}\label{lem:brk}
There exists a function $C(\lambda)$ such that for all $\lambda<\infty$
\[ \sup_{x\in K} g_{\lambda}(x,x) \leq C(\lambda) < \infty. \]
\end{lem}

\begin{proof}
We follow the proof of \cite{bar2}~Theorem~7.20. Note that for any fixed $x\in K$ we have 
$g_{\lambda}(x,.)\in \cf$ and hence using (\ref{dfrm}) 
\[ |g_{\lambda}(x,y)-g_{\lambda}(x,x)|^2 \leq R(x,y) \ce_{\lambda}(g_{\lambda}(x,.),g_{\lambda}(x,.)). \]
By the reproducing kernel property and the global bound on the resistance across $K$ from Corollary~\ref{cor:diambd} 
we have
\[ |g_{\lambda}(x,y)-g_{\lambda}(x,x)|^2 \leq C g_{\lambda}(x,x). \]
Rearranging
\[ g_{\lambda}(x,y) \geq g_{\lambda}(x,x) - (Cg_{\lambda}(x,x))^{1/2}, \]
and integrating over $y$ against $\mu$ we have
\[ g_{\lambda}(x,x) \leq \frac1{\lambda} + (Cg_{\lambda}(x,x))^{1/2}. \]
The result then follows easily.
\end{proof}

\begin{lem}
There exists a constant $C$ such that for all $\bfi\in\Lambda_k$ and $x\in int(K_{\bfi})$, 
\[ g_{t_{\bfi}^{-1}}(x,x) \leq C \rho_{\bfi}^{-1}. \] 
\end{lem}

\begin{proof} 
By Lemma~\ref{lem:greenscale} and (\ref{eq:green}) we have for $x\in K_{\bfi}$
\begin{eqnarray*}
g_{\lambda}^{D,\sigma^{\bfi}}(\psi^{-1}_{\bfi}(x),\psi^{-1}_{\bfi}(x)) &=& \rho_{\bfi} 
g^{k,D}_{\lambda/t_{\bfi}}(x,x) \leq \rho_{\bfi} g_{\lambda/t_{\bfi}}(x,x) \\
&\leq & \rho_{\bfi} g^k_{\lambda/t_{\bfi}}(x,x) = g^{\sigma^{\bfi}}_{\lambda}(\psi^{-1}_{\bfi}(x),\psi^{-1}_{\bfi}(x)).
\end{eqnarray*}

Now set $\lambda=1$ and note that by Lemma~\ref{lem:brk} $g_1$ is uniformly bounded. Thus
\[ g_1^{D,\sigma^{\bfi}}(\psi^{-1}_{\bfi}(x),\psi^{-1}_{\bfi}(x)) \leq \rho_{\bfi} g_{1/t_{\bfi}}(x,x) \leq
 g_1^{\sigma^{\bfi}}(\psi^{-1}_{\bfi}(x),\psi^{-1}_{\bfi}(x)) \leq C. \]
Rearranging we have
\[ g_{t_{\bfi}^{-1}}(x,x) \leq C\rho_{\bfi}^{-1}, \]
as required.
\end{proof} 

\begin{thm}\label{thm:hkub}
There exists a constant $c$ such that
\[ p_{t_{\bfi}}(x,x) \leq c \mu_{\bfi}^{-1}, \;\; \forall x \in K_{\bfi}, \forall \bfi\in\Lambda_k. \]
\end{thm}

\begin{proof} 
As 
\[ g_{\lambda}(x,x) = \int_0^{\infty} e^{-\lambda t} p_t(x,x) dt, \]
we have, by the monotonicity of $p_t(x,x)$ in $t$, that for all $u$
\[ g_{\lambda}(x,x) \geq p_u(x,x) \int_0^u e^{-\lambda t} dt = p_u(x,x) \frac{1-e^{-\lambda u}}{\lambda}. \]
Thus, setting $\lambda=t_{\bfi}^{-1}=1/u$, we have
\[ p_{t_{\bfi}}(x,x)(1-e^{-1})t_{\bfi} \leq g_{t_{\bfi}^{-1}}(x,x) \leq C \rho_{\bfi}^{-1}. \]
Rearranging and the definition of $t_{\bfi}$ then gives the result. 
\end{proof} 


\subsection{Lower Bound}

We follow a standard approach see for instance \cite{barham}, \cite{BarKum}. 
For this we require an estimate on the exit time distribution 
for balls. We start with some preliminary results. 

Let $\{X_t:t\geq 0\}$ be the diffusion with law $\bp$ associated with the Dirichlet
form $(\ce,\cf)$. We write $\bp^x$ for the law of the process with $X_0=x$ and $\be^x$ for the corresponding expectation.
We write $T_A=\inf\{t\geq 0:X_t\in A\}$ for the first hitting time of the set $A$.
For $\bfi\in \Lambda_k$ we write 
\[ D_{\bfi}=\bigcup_{\bfj\in\Lambda_k} \{K_{\bfj}:K_{\bfj}\cap K_{\bfi}\neq \emptyset\} \] 
for the union of the complex $K_{\bfi}$ and its neighbours. Let $\Lambda_k(z):=  \{\bfj \in \Lambda_k: z\in K_{\bfj}\}$.
For $z\in \tilde{V}_k$ we define 
\[ D_{\bfi}^z :=  \bigcup_{\bfj \in \Lambda_k(z)} K_{\bfj}, \;\; \partial D_k^z :=   \bigcup_{\bfj \in \Lambda_k(z)} \psi_j(V_0)\backslash \{z\},
\;\; \partial D_{\bfi}:= \bigcup_{\bfj\in\Lambda_k, K_{\bfj}\cap K_{\bfi}\neq \emptyset} \psi_{\bfj}(V_0) \backslash\psi_{\bfi}(V_0). \]
We will also use the notation $\partial K_{\bfi}:= \psi_{\bfi}(V_0)$. 


Recalling \eqref{Nbd}, \eqref{west} and  \eqref{dfyk} we let $y_k^{n_0}=\sum_{i=k}^{k+n_0} y_i$, $\hat{\eta}=N_{\inf}w_{\inf}/N_{\sup}w_{\sup}$ and write $\chi(k,n_0)= (\eta/\heta)^{y_k}\heta^{y_k^{n_0}}$.

\begin{lem}\label{lem:resbd}
There exist constants $c_i$ and $n_0$ such that
\[ c_1 \chi(k,n_0) e^{-k} \leq \be^xT_{\partial D_{\bfi}} \leq 
\sup_{z\in K_{\bfi}} \be^z T_{\partial D_{\bfi}} \leq c_2 e^{-k},\;\;\forall x\in K_{\bfi},\forall \bfi\in\Lambda_k. \]
\end{lem}

\begin{proof} 

We begin by observing that 
\begin{equation}
 \be^x T_{\partial D_{\bfi}} = \be^x T_{\partial K_{\bfi}} + \sum_{y\in \partial K_{\bfi}} \bp^x(X_{T_{\partial K_{\bfi}}}=y) \be^y T_{\partial D_{\bfi}}. \label{eq:htdecomp}
 \end{equation}

To treat the first term we note that
the Dirichlet form restricted to $K_{\bfi}$ with Dirichlet boundary conditions is a reproducing kernel Hilbert space with the 
associated Green function $g_{K_{\bfi}}(x,.)$ as the kernel. Let $f(y)=g_{K_{\bfi}}(x,y)/g_{K_{\bfi}}(x,x)$. 
By the definition of $f$ and the reproducing kernel property we have $\ce(f,f) = 1/g_{K_{\bfi}}(x,x)$. 
By the definition of the effective resistance we also have that $g_{K_{\bfi}}(x,x) = R(x,\partial K_{\bfi})$.  
As $g_{K_{\bfi}}$ is harmonic away from $x$ and is 0 on $\partial K_{\bfi}$ we have that 
$0\leq f(y) \leq 1$ for all $y$. Hence, putting these observations together and using Corollary~\ref{cor:diambd},
we have that, for any $y \in K_{\bfi}$, 
\[ \be^y T_{\partial K_{\bfi}} = \int_{K_{\bfi}} g_{K_{\bfi}}(y,z) \mu(dz) \leq R(y, \partial K_{\bfi}) \mu(K_{\bfi}) \leq  c r_{\bfi} \mu_{\bfi} \leq c e^{-k},
\]
as $\bfi\in\Lambda_k$.

We next consider the exit time from $D_{\bfi}$ started at a point $y\in \partial K_{\bfi}$.

Let $U_0=0$ and set $U_i=\inf\left\lbrace t>U_{i-1}:X_t \in \tilde{V}_k\backslash \{X_{U_{i-1}}\}\right\rbrace$. Then
$\hat{X}_i = X_{U_i}$ is a discrete time Markov chain on $\tilde{V}_k$. Let $S=\inf\{n: \hat{X}_n\in \partial 
D_{\bfi}\}$. By construction we see that $\{\hat{X}_n:n\leq S\}$ can be viewed as a $V_0+1$ state discrete
time Markov chain with $V_0$ states
as the vertices of $K_{\bfi}$ and an absorbing state given by amalgamating the vertices in $\partial D_{\bfi}$. 
By construction this Markov chain has transition probabilities given by the
conductances on $\tilde{G}_k$. As two of the vertices in $\partial K_{\bfi}$ must be internal to a triangle or 
$d$-dimensional tretrahedron in $K_{\bfi||\bfi|-1}$  the conductance between the edges across $\Delta_{\bfi}$ 
and at least one edge to $\partial D_{\bfi}$ are comparable or otherwise the conductances across $\Delta_{\bfi}$ 
are smaller and hence $E S<\infty$ independent of $k$. 

The time taken for the original process to exit is then $\be^y U_S$. 
We now compute the time for a step. 

The same argument as before for the first term in \eqref{eq:htdecomp}  but using $g_{D^y_{\bfi}}$ gives
\[ \be^y T_{\partial D_{\bfi}^y} = \int_{D_{\bfi}^y} g_{D_{\bfi}^y}(y,z) \mu(dz) \leq R(y, \partial D_{\bfi}^y) \mu(D_{\bfi}^y).
\]
Now observe that by the definition of resistance we have
\[ R(y,\partial D_{\bfi}^y) \leq R(y,z), \;\;\forall z\in \partial D_{\bfi}^y. \]
Thus we have $R(y,\partial D_{\bfi}^y) \leq \min_{z\in \partial D_{\bfi}^y} R(y,z)$. By our
estimate on the resistance in Lemma~\ref{lem:resbd} this gives $R(y,\partial D_{\bfi}^y) \leq  \min_{\bfj\in\Lambda_k(y)} r_{\bfj}$.
Hence, as the number of cells that meet at $y$ is bounded,
\[ \be^y T_{\partial D_{\bfi}^y} \leq   \min_{\bfj\in\Lambda_k(y)} r_{\bfj} \sum_{\bfj\in\Lambda_k(y)} \mu_{\bfj} 
 \leq  c \max_{\bfj \in\Lambda_k(y)}r_{\bfj} \mu_{\bfj} \leq  c e^{-k}.
\]

We are now ready to show $\be^y T_{\partial D_{\bfi}} \leq C e^{-k}$. To see this we use
\[ \be^y T_{\partial D_{\bfi}}  = \be^y U_S = \be^y \sum_{i=1}^S \left( U_i-U_{i-1}\right). \]
Note that $S$ is a stopping time with respect to $\{\cf_{U_i}\}_{i=0}^{\infty}$, where $\{\cf_t\}_{t\geq 0}$ is the filtration generated by $X$.
As $\be(U_i-U_{i-1}|\cf_{U_{i-1}})=\be^{X_{U_{i-1}}} 
T_{\tilde{V}_k \backslash \{X_{U_{i-1}}\}}$, a minor modification of Wald's identity shows that
\[ \be^y T_{\partial D_{\bfi}} \leq c e^{-k} \be^y S, \;\;\forall y \in \partial K_{\bfi}. \]
Putting this back into (\ref{eq:htdecomp}) gives the upper bound.

For the mean hitting time lower bound we return to (\ref{eq:htdecomp}) to see that
\[ \be^x T_{\partial D_{\bfi}} \geq  \min_{y\in \partial K_{\bfi}} \be^y T_{\partial D_{\bfi}} \geq  
\min_{y\in \partial K_{\bfi}} \be^y T_{\partial D^y_{\bfi}}. \]
Using the properties of $g_{D^y_{\bfi}}$, and setting $f(z)=g_{D_{\bfi}^y}(y,z)/g_{D_{\bfi}^y}(y,y)$, 
we see that
\[ |f(y)-f(z)|^2 \leq R(y,z) \ce(f,f) = \frac{R(y,z)}{g_{D^y_{\bfi}}(y,y)} = \frac{R(y,z)}{R(y,\partial D^y_{\bfi})}. \]
Let 
\[ A^c_y :=\{ z: R(y,z) \leq c R(y,\partial D^y_{\bfi})\}. \]
Let $\bfj^*\in \Lambda_k(y)$ denote the index at which $\min_{\bfj \in \Lambda_k(y)} r_{\bfj}$ is attained. 
Thus, by the boundedness of $| \Lambda_k(y)|$, we have
 \begin{equation} R(y,\partial D^y_{\bfi}) \geq c_1 r_{\bfj^*}. \label{eq:rydy}
\end{equation}
We now show that $A^c_y$ must have measure comparable with $\mu_{j^*}$.

By decomposing the cell $K_{\bfj^*}$ we have
 \[ K_{\bfj^*} = \bigcup_{\bfj:|\bfj|=n} \psi_{\bfj^*\bfj}(K(\sigma^{\bfj^*\bfj}T)), \]
and we write $\bfk$ with $|\bfk|=n$ such that $y\in \psi_{\bfj^*\bfk}(K(\sigma^{\bfj^*\bfk}T))=
K_{\bfj^*\bfk}$. Then, by Corollary~\ref{cor:diambd}, for any $z\in K_{\bfj^*\bfk}$ we have a constant $c$ such that
\[ R(y,z) \leq c r_{\bfj^*\bfk} \leq c r_{\bfj^*} r_{\sup}^n, \]
and hence by (\ref{eq:rydy})
 \[ R(y,z) \leq \frac{c r_{\sup}^n}{c_1} R(y,\partial D^y_{\bfi}).\]
Thus, if we take $n_0=\inf\{n: r_{\sup}^n <c_1/c\}$ and set $c_2 =  \frac{c r_{\sup}^{n_0}}{c_1}$,
we have $K_{\bfj^*\bfk} \subset A^{c_2}_y$ where $c_2<1$.

Hence for $z\in K_{\bfj^*\bfk}$ we have $|f(y)-f(z)|^2 \leq c_2$. As $f(y)=1$ we see that 
we must have $f(z)\geq c'=1-\sqrt{c_2}$. Thus for any $y\in \partial K_{\bfi}$ we have,
writing $k_{n_0}+|\bfj^*|$ for the first neck after $n_0+|\bfj^*|$,
\begin{eqnarray*}
\be^y T_{\partial D^y_{\bfi}} &=& \int_{D^y_{\bfi}} g_{D^y_{\bfi}}(y,z) \mu(dz) 
     \geq  c' g_{D^y_{\bfi}}(y,y)\mu(K_{\bfj^*\bfk}) \\
&= & c' R(y,\partial D^y_{\bfi})  \mu_{\bfj^*} \frac{\sum \{w_{\bfk\bfi}:|\bfk\bfi|=k_{n_0},\bfi\in 
T^{\sigma^{\bfj^*\bfk}}\}}{\sum \{w_{\bfi}:|\bfi|=k_{n_0},\bfi\in T^{\sigma^{\bfj^*}}\}}.
\end{eqnarray*}
Now apply (\ref{eq:rydy}), the upper and lower bounds on the weights and Lemma~\ref{lem:spatial}, 
\begin{eqnarray*}
\be^y T_{\partial D_{\bfi}} &\geq & c_3 r_{\bfj^*} \mu_{\bfj^*}N_{\inf}^{k_{n_0}-n_0}
\left( \frac{w_{\inf}}{N_{\sup}w_{\sup}}\right)^{k_{n_0}} \\
&\geq & c e^{-k} \eta^{y_k} \heta^{y_{k+1}+\dots+ y_{k+n_0}}
\end{eqnarray*}
as required.
\end{proof} 

\begin{lem}\label{lem:hittail}
There exist constants $c_3,c_4$ such that for $x\in K_{\bfi}, \bfi\in\Lambda_k$
\[ P^x (T_{\partial D_{\bfi}} \leq t) \leq 1-c_3 \chi(k,n_0), \;\;\mbox{for } t\leq c_4\frac12 
\chi(k,n_0)^2 e^{-k}. \]
\end{lem}

\begin{proof} 
We note that
\[ T_{\partial D_{\bfi}} \leq t + I_{\{T_{\partial D_{\bfi}}>t\}}(T_{\partial D_{\bfi}}-t). \]
Taking expectations
\begin{eqnarray*}
\be^x T_{\partial D_{\bfi}} &\leq&  t + \be^x \left(I_{\{T_{\partial D_{\bfi}}>t\}}\be^{X_t}
 T_{\partial D_{\bfi}}\right) \\
&\leq & t + \bp^x(T_{\partial D_{\bfi}}>t) \sup_{y\in D_{\bfi}}\be^y T_{\partial D_{\bfi}}.
\end{eqnarray*}
Rearranging and then applying our exit time estimates from Lemma~\ref{lem:resbd}
\begin{eqnarray*}
\bp^x(T_{\partial D_{\bfi}} \leq t) &\leq & \frac{t}{\sup_{y\in D_{\bfi}} \be^y T_{\partial D_{\bfi}}} + 
1- \frac{\be^x T_{\partial D_{\bfi}}}{\sup_{y\in D_{\bfi}} \be^y T_{\partial D_{\bfi}}} \\
&\leq & c_1 e^k t\chi(k,n_0)^{-1} + 1- c_2 \chi(k,n_0).
\end{eqnarray*}
Thus, if $t\leq \frac12 c_2 c_1^{-1} \chi(k,n_0)^2 e^{-k} $, we have
\[ \bp^x(T_{\partial D_{\bfi}} \leq t) \leq 1- \frac12 c_2 \chi(k,n_0), \]
as required. 
\end{proof} 

\begin{thm}\label{thm:hklb}
There are constants $c, \alpha'$ such that for $t\leq c_4 e^{-k} \chi(k,n_0)^2$
\[ p_{2t}(x,x) \geq c \chi(k,n_0)^2 \mu(D_{\bfi})^{-1}, \;\;\forall x\in K_{\bfi}, \bfi\in \Lambda_k. \] 
\end{thm}

\begin{proof} 
A standard argument gives the following. If $t\leq \frac12 c_4 \chi(k,n_0)^2 e^{-k}$, then by Lemma~\ref{lem:hittail}
\[ (c_2 \chi(k,n_0))^2 \leq P^x(X_t \in D_{\bfi})^2 = (\int_{D_{\bfi}} p_t(x,y) \mu(dy))^2  
\leq  \mu(D_{\bfi}) p_{2t}(x,x),
\]
as required. 
\end{proof} 

Finally we can use the estimates on $y_k$ to provide a $P_V$ a.s. estimate in terms of the scale factors.

\begin{thm}\label{thm:hklbk}
There are constants $c,\beta$ such that $P_V$ a.s.  for sufficiently large $k$, for $t\leq c e^{-k}k^{-2\beta}$
\[ p_t(x,x) \geq c \mu(D_{\bfi})^{-1} k^{-2\beta}, \forall x\in K_{\bfi}, \bfi\in\Lambda_k. \]
\end{thm}

\begin{rem}{\rm In a different setting \cite{BarKum} obtained a finer estimate on the exit time from a 
complex which enables the derivation of a finer form of this on-diagonal estimate. We do not
derive such a result here though we expect that the same techniques could be applied to do so. Our result
is enough to enable us to compute the $\mu$-almost everywhere local spectral exponent}
\end{rem}

\subsection{Local Spectral Exponent}

As in \cite{BarKum} we will see that the local spectral dimension obtained by considering the limit as $k\to\infty$
of  $p_{t_{\bfi}}(x,x)$ for $x\in K_{\bfi}, \;\; \bfi\in\Lambda_k$ will in general not coincide with the global spectral 
dimension.

We have the following preliminary result. Let $\bfi^x \in \partial T$ be such that $K_{\bfi^x|k}\to \{x\}$ as
$k\to\infty$. 
\begin{lem}\label{lem:muk}
There exists a constant $c$ such that $D_{\bfi^x|n(k+[c\log{k}])} \subset K_{\bfi^x|n(k)}$ for 
all sufficiently large $k$ for $\mu$-a.e. $x\in K$, $P_V$ a.s.
\end{lem}

\begin{proof}
Let $T_{n(m),b}$ denote the addresses of the three boundary cells at the $m$-th neck.
By Lemma~\ref{hineq}
we must have
\[ a:= E_V \max_{\bfj\in T_{n(1),b} } \mu_{\bfj} <1. \]
Now for $\bfi\in \{\bfj\in T:|\bfj|=n(k+m)\}$ we have $D_{\bfi} \subset K_{\bfi|n(k)}$ if $K_{\bfi}\cap\partial 
K_{\bfi|n(k)} = \emptyset$. Then, setting $A = \{K_{\bfi}: \bfi\in\{\bfj\in T:|\bfj|=n(k+m)\}, K_{\bfi}\cap \partial K_{\bfi|n(k)} \neq \emptyset\}$, we have
\begin{eqnarray*}
 E_V \mu(A)
 &=& E_V \sum_{\bfi\in\{\bfj\in T:|\bfj|=n(k+m)\}} \mu_{\bfi} I_{\{K_{\bfi}\cap \partial K_{\bfi|n(k)} \neq \emptyset \}} \\
 &=& E_V \sum_{\bfi\in\{\bfj\in T:|\bfj|=n(k)\}} \mu_{\bfi} \sum_{\bfj\in T_{n(k+m),b}} \frac{\mu_{\bfj}}{\mu_{\bfi}}.
\end{eqnarray*} 
By construction the terms $\mu^{(j)}_{\bfi}=\frac{\mu_{\bfi|n(j)}}{\mu_{\bfi|n(j-1)}}$ are independent and equal in 
distribution to $\mu_{\bfi|n(1)}$, allowing us to write
\begin{eqnarray*}
 E_V \mu(A)
 &=& E_V \sum_{\bfi\in\{\bfj\in T:|\bfj|=n(k)\}} \mu_{\bfi} E_V \sum_{\bfj\in T_{n(m),b}} \prod_{j=1}^m \mu^{(j)}_{\bfj}\\
&\leq & (d+1) a^m.
\end{eqnarray*}
Thus we have
\[ E_V \sum_{k=1}^{\infty} \mu(x\in K: D_{\bfi^x|n(k+[c\log{k}])} \not\subset K_{\bfi^x|n(k)})  
\leq c_1 \sum_{k=1}^{\infty} a^{c\log{k}} <\infty, \]
for large enough $c$. Hence $P_V$ a.s. we have
\[ \mu(x\in K: D_{\bfi^x|n(k+[c\log{k}])} \not\subset K_{\bfi^x|n(k)} \;\; i.o.) = 0, \]
as required.  
\end{proof}

For the rest of this section we write $T_{n(1)} = \{\bfj \in T: |\bfj|=n(1)\}$ for the tree up to the first neck.
Take another set of weights $\{\{\hat{w}_i^F\}_{i=1}^{|F|}\}_{F\in\mathbf{F}}$ satisfying the conditions of Section~3.3 and define the associated measure $\hmu$.

Observe that by (\ref{mibd}) and the definition of $\eta$ we have
\[ \mu_{\bfi} \geq (\frac{\eta}{r_{\inf}})^{n(1)}, \;\; t_{\bfi} \geq \eta^{n(1)}, \;\; \bfi \in T_{n(1)}. \]
Thus $\log \mu_{\bfi} \geq n(1)\log \frac{\eta}{r_{\inf}}$ and as 
\[ 0 > \sum_{\bfi\in T_{n(1)}} \hmu_{\bfi} \log \mu_{\bfi} \geq n(1) \log\frac{\eta}{r_{\inf}}, \]
we have
\[ E_V |\sum_{\bfi\in T_{n(1)}} \hmu_{\bfi} \log \mu_{\bfi}| \leq c E_V n(1) < \infty. \]
We can control $E_V |\sum_{\bfi\in T_{n(1)}} \hmu_{\bfi} \log t_{\bfi}| $ in the same way.

In the same way as \cite{BarKum} we can now determine the local spectral exponent for the heat kernel $p_t(x,x)$
defined with respect to the reference measure $\mu$ for $\hat{\mu}$ almost every  $x$.

\begin{thm}\label{thm:locspecd}
$P_V$ almost surely, for $\hmu$-almost every $x\in K$ we have
\[ \lim_{t\to 0} \frac{\log{p_t(x,x)}}{-\log{t}} = \frac{\widehat{d}_s(\hmu)}2 = \frac{E_V \sum_{\bfi\in T_{n(1)}} 
\hmu_{\bfi} \log \mu_{\bfi}}{E_V \sum_{\bfi\in T_{n(1)}} \hmu_{\bfi} \log t_{\bfi}}. \]
\end{thm}

\begin{proof} 
For $x\in K$ we have a sequence $\bfi|n(k)$ for which
$D_{\bfi|n(k)} \to \{x\}$ as $k\to\infty$. By monotonicity of the diagonal heat kernel in time
for $t\in (t_{\bfi|n(k)}, t_{\bfi|n(k-1)}]$ we have $p_t(x,x) \leq p_{t_{\bfi|n(k)}}(x,x)$ and thus
\[ \limsup_{t\to 0} \frac{\log{p_t(x,x)}}{-\log{t}} \leq \limsup_{k\to\infty} \frac{\log{p_{t_{\bfi|n(k)}}(x,x)}}
{-\log{t_{\bfi|n(k-1)}}}, \;\;P_V \; a.s. \]
Now using Theorem~\ref{thm:hkub} we have
\[ \log{p_{t_{\bfi|n(k)}}(x,x)} \leq C-\log \mu_{\bfi|n(k)} = C-\sum_{j=1}^k 
\log \frac{\mu_{\bfi|n(j)}}{\mu_{\bfi|n(j-1)}}. \]
We now consider the probability measure $d\hmu dP_V$on $\{1,\dots,N_{\sup}\}\times \Omega_V$ 
(with the product $\sigma$-algebra).
If the point $x$ is chosen according to $\hmu$, then the terms $\mu^{(j)}_{\bfi}=\frac{\mu_{\bfi|n(j)}}
{\mu_{\bfi|n(j-1)}}$ are independent and equal in 
distribution to $\mu_{\bfi|n(1)}$ under $d\hmu dP_V$. We can also express $-\log{t_{\bfi|n(k)}}$ in terms of independent 
random variables $t^{(j)}_{\bfi}$ defined in the same way. It is easy to see that $\log{t_{\bfi|n(k)}}/
\log{t_{\bfi|n(k-1)}} \to 1$ for any $x\in K$, $P_V$-almost surely and hence
\begin{equation}
\limsup_{t\to 0} \frac{\log{p_t(x,x)}}{-\log{t}} \leq \limsup_{k\to\infty} \frac{\log{p_{t_{\bfi|n(k)}}(x,x)}}
{-\log{t_{\bfi|n(k)}}} \leq \limsup_{k\to\infty} \frac{\frac{1}{k} \sum_{j=1}^k 
\log \mu^{(j)}_{\bfi}} {\frac{1}{k} \sum_{j=1}^k \log t^{(j)}_{\bfi}}. \label{eq:hkds}
\end{equation}
As the mean of $\log \mu^{(j)}$ is finite we can apply the strong law of large numbers under $d\hmu dP_V$ 
to see that
\[ \lim_{k\to\infty} \frac{1}{k} \sum_{j=1}^k 
\log \mu^{(j)}  = E_V \sum_{\bfi\in T_{n(1)}} \hmu_{\bfi} \log \mu_{\bfi}, \;\;\hmu \;a.e. \; x\in K, \;\; P_V \; a.s. \]
Similarly we can find the limit for the denominator in (\ref{eq:hkds}). Thus we have
\[ \limsup_{k\to\infty} \frac{\log{p_{t_{\bfi|n(k)}}(x,x)}}
{-\log{t_{\bfi|n(k)}}} \leq \frac{E_V \sum_{\bfi\in T_{n(1)}} \hmu_{\bfi} \log \mu_{\bfi}}
{E_V \sum_{\bfi\in T_{n(1)}} \hmu_{\bfi} \log t_{\bfi}}. \]

For the lower bound we define $\ell(\bfi,k) = \ell$ if $\bfi |n(\ell) \in \Lambda_k$. Thus
\[ -\log t_{\bfi |n(\ell(\bfi,k)-1)} < k \leq -\log t_{\bfi|n(\ell(\bfi,k))}. \]
Hence, it is clear that, by the independence
\begin{equation} 
\lim_{k\to\infty} \frac{\ell(\bfi,k)}{k} = \lim_{\ell\to\infty} \frac{-\ell}{\log t_{\bfi|n(\ell)}} 
= \frac{-1}{E_V \sum_{\bfi\in T_{n(1)}} \hmu_{\bfi} \log t_{\bfi}},\;\;\hmu  \; a.e. \; x\in K, \;\;P_V\; a.s. \label{eq:lkk}
\end{equation}
Now, from Theorem~\ref{thm:hklbk}, we have that $P_V$ a.s. for
$ce^{-(k+1)} (k+1)^{-\beta} <t \leq ce^{-k} k^{-\beta}$
we have for $x\in K_{\bfi}, \bfi\in\Lambda_k$,
\begin{eqnarray*}
\frac{\log p_{2t}(x,x)}{-\log t} &\geq & \frac{\log{(c \mu(D_{\bfi})^{-1} k^{-2\beta})}}
{\log{(ce^{-(k+1)} (k+1)^{-2\beta})}} \\
&=& \frac{\log{c}-2\beta\log{k} -\log\mu(D_{\bfi})}{\log{c}-k-1-2\beta\log{(k+1)}}.
\end{eqnarray*}
Thus
\begin{eqnarray*} 
\lim_{t\to 0} \frac{\log p_t(x,x)}{-\log t} &=& \lim_{t\to 0} \frac{\log p_{2t}(x,x)}{-\log t} \\
&\geq & \lim_{k\to \infty} \frac{-\log\mu(D_{\bfi})}{k}.
\end{eqnarray*}

We now observe that by Lemma~\ref{lem:muk} we have a constant $c'$ such that
\[  -\log\mu(D_{\bfi|n(\ell(\bfi,k))}) \geq  -\log\mu_{\bfi|n(\ell(\bfi,k)-[c'\log \ell(\bfi,k)])}. \]
Using this, (\ref{eq:lkk}) and writing $\tilde{\ell}(\bfi,k) = \ell(\bfi,k)-[c'\log \ell(\bfi,k)]$, we have
\begin{eqnarray*}
 \lim_{k\to \infty} \frac{\log\mu_{\bfi|n(\tilde{\ell}(\bfi,k))}}{k} &=& -\lim_{k\to \infty} \frac{\tilde{\ell}(\bfi,k)}{k} \\
& & \qquad\qquad 
\lim_{k\to\infty} \frac{1}{\tilde{\ell}(\bfi,k)} \sum_{j=1}^{\tilde{\ell}(\bfi,k)} \log \mu_{\bfi}^{(j)} \\ 
&=& \frac{E_V \sum_{\bfi\in T_{n(1)}} 
\hmu_{\bfi} \log \mu_{\bfi}}{E_V \sum_{\bfi\in T_{n(1)}} \hmu_{\bfi} \log t_{\bfi}}, 
\end{eqnarray*}
for $\hmu$ a.e. $x\in K$, $P_V$ almost surely, as required. 
\end{proof} 

In the case where the reference measure $\mu$ is the flat measure $\nu$ in the resistance metric, 
the weights are proportional to $r_{\bfi}^{d_f^r}$ and $E_V \log \sum_{\bfi\in T_{n(1)}} r_{\bfi}^{d_f^r} =0$,  a simple calculation
shows that 
\begin{eqnarray*}
\frac{\widehat{d}_s}2 &=& \frac{E_V \sum_{\bfi\in T_{n(1)}} 
\hmu_{\bfi} \log \mu_{\bfi}}{E_V \sum_{\bfi\in T_{n(1)}} \hmu_{\bfi} \log t_{\bfi}} \\
&=& \frac{E_V \sum_{\bfi\in T_{n(1)}} \hmu_{\bfi} 
\log \frac{r_{\bfi}^{d_f^r}}{\sum_{\bfj\in T_{n(1)}} r_{\bfj}^{d_f^r}}}
{E_V \sum_{\bfi\in T_{n(1)}} \hmu_{\bfi} 
\log \frac{r_{\bfi}^{1+d_f^r}}{\sum_{\bfj\in T_{n(1)}} r_{\bfj}^{d_f^r}}} \\
&=& \frac{d_fE_V \sum_{\bfi\in T_{n(1)}} \hmu_{\bfi} \log r_{\bfi}}{(1+d_f^r)E_V \sum_{\bfi\in T_{n(1)}} \hmu_{\bfi}
\log r_{\bfi}} \\
&=& \frac{d_f^r}{d_f^r+1} = \frac{d_s}2.
\end{eqnarray*}
Indeed in this case we can go further and give a bound on the size of 
the scale fluctuations.

\begin{thm}\label{thm:flatfluc}
If $\nu$ is the flat measure in the resistance metric we have constants $c_1,c_2,c_3,c_4 \in (0,\infty)$ 
and a random variable $0<c_5$ such that $P_V$ a.s. for any $x\in K$
\[ c_1\phi(1/t)^{-c_2} t^{-d_s/2} \leq p_t(x,x) \leq c_3 \phi(1/t)^{c_4} t^{-d_s/2}, \;\; 0<t<c_5. \]
\end{thm}

\begin{proof} 
We begin by observing that for $\bfi\in\Lambda_k$ we have $t_{\bfi} \leq e^{-k}$ and 
thus substituting in the upper bound estimate from Theorem~\ref{thm:hkub}, for $x\in K_{\bfi}$
\begin{equation} p_{e^{-k}}(x,x) \leq p_{t_{\bfi}}(x,x) \leq c \nu_{\bfi}^{-1}. \label{hkflat}
\end{equation}
By (\ref{4122}) we have that $P_V$ almost surely for sufficiently large $k$, 
$\nu_{\bfi}\geq r_{\bfi}^{d_f^r} \exp(-c\Phi(\ell(\bfi)))$ and hence
$ r_{\bfi}^{1+d_f^r} \exp(-c \Phi(\ell(\bfi))) \leq t_{\bfi} \leq e^{-k}$. Thus, using Lemma~\ref{lem:spatial},
\[ r_{\bfi} \leq e^{-k/(1+d_f^r)} \exp(c' \Phi(\ell(\bfi))), \]
and
\[ p_{e^{-k}}(x,x) \leq e^{k\frac{d_f^r}{1+d_f^r}} \exp(c'' \Phi(\ell(\bfi))). \]
Thus, for $e^{-k} \leq t < e^{-k+1}$, and as $l\leq ck \leq -c\log{t}$, we have for any $x\in K$,
\[ p_t(x,x) \leq C t^{-d_s/2} \exp(c'\Phi(\log{(1/t)})) =  C t^{-d_s/2} \phi(1/t)^{c'}, \;\; P_V \; a.s. \]

For the lower bound we observe from Theorem~\ref{thm:hklbk} that $P_V$ almost surely for sufficiently large $k$,
for $t\leq c e^{-k}k^{-2\beta}$
\[ p_t(x,x) \geq c \mu(D_{\bfi})^{-1} k^{-2\beta}, \forall x\in K_{\bfi}, \bfi\in\Lambda_k. \]
As $e^{-k} \geq t_{\bfj} =r_{\bfj}\nu_{\bfj} \geq r_{\bfj}^{1+d_f^r}  \exp(-c \Phi(\ell(\bfj)))$ $P_V$ almost surely, 
as before we have 
\[ r_{\bfj} \leq e^{-k/(1+d_f^r)}\exp(c'\Phi(\ell(\bfj))). \]
Then as the number of cells in $D_{\bfi}$ is bounded and $\ell(\bfj) \leq ck$ by Lemma~\ref{lem:spatial}(a), 
we have
\begin{eqnarray*}
 \nu(D_{\bfi}) &=& \sum_{\bfj\in\Lambda_k} \frac{r_{\bfj}^{d_f^r}}{\sum_{j':|\bfj'|=n(\ell(j))} r_{\bfj'}^{d_f^r}} 
I_{\{K_{\bfi}\cap K_{\bfj} \neq \emptyset\}} \\
&\leq &  \sum_{\bfj\in\Lambda_k} r_{\bfj}^{d_f^r} I_{\{K_{\bfi}\cap K_{\bfj} \neq \emptyset\}} \exp(c \Phi(\ell(\bfj))) \\
&\leq & c e^{-kd_f^r/(1+d_f^r)}\exp(c''\Phi(\ell(\bfj))).
\end{eqnarray*}
Thus, $P_V$ a.s. for sufficiently large $k$ for $ t\leq ce^{-k} k^{-2\beta}$,
\[ p_{t}(x,x) \geq c  k^{-2\beta} e^{kd_f^r/(1+d_f^r)}\exp(-c''\Phi(k)), \;\;\forall x\in K_{\bfi}. \] 
For $ ce^{-(k+1)} (k+1)^{-2\beta} < t\leq ce^{-k} k^{-2\beta}$ we have
$c_1ee^{-k} k^{-2\beta} <t$ so that $e^{k} k^{2\beta} > c_2 t^{-1}$ and
\[ p_{t}(x,x) \geq c  k^{-2(2d_f^r+1)\beta/(d_f^r+1)} t^{-d_f^r/(1+d_f^r)}\exp(-c''\Phi(k)), \;\;\forall x\in K_{\bfi}. \]
Now as $k\leq \log c + \log{(1/t)}$ we have for sufficiently small $0<t$, for any $x\in K$
\[ p_t(x,x) \geq b' |\log{t}|^{-\beta'} t^{-d_s/2} \exp(-c'' \Phi(|\log{t}|). \]
By adjusting $c''$ we can absorb the logarithm into the exponential term and we have the result.
\end{proof} 


\end{document}